\documentclass[a4paper, reqno]{amsart}
\usepackage{etex} 
\usepackage[margin=3cm]{geometry}

\usepackage{amsmath}
\usepackage{amssymb}
\usepackage{amsthm}
\usepackage{amsfonts}
\usepackage{booktabs}
\usepackage[usenames,dvipsnames,svgnames]{xcolor}
\usepackage{epsfig}
\usepackage{fancyvrb} 
\usepackage{hyperref}
\usepackage{mathrsfs}
\usepackage{mathtools}
\usepackage{pictexwd, dcpic}
\usepackage{pinlabel}
\usepackage{stmaryrd}
\usepackage{enumitem} 

\usepackage{tikz}
\usepackage{tikz-cd}
\usepackage{caption}
\usetikzlibrary{decorations.markings}

\newcommand{\ZZ}{\mathbb{Z}}

\newcommand{\vB}{\mathcal{B}}
\newcommand{\vC}{\mathcal{C}}

\newcommand{\vJ}{\mathcal{J}}

\newcommand{\vM}{\mathcal{M}}
\newcommand{\vN}{\mathcal{N}}

\newcommand{\vQ}{\mathcal{Q}}
\newcommand{\vR}{\mathcal{R}}
\newcommand{\vS}{\mathcal{S}}

\newcommand{\vV}{\mathcal{V}}

\newcommand{\vSP}{\mathcal{SP}}
\newcommand{\vDS}{\mathcal{DS}}

\newcommand{\ba}{{\bf a}}
\newcommand{\bb}{{\bf b}}
\newcommand{\bc}{{\bf c}}
\newcommand{\bd}{{\bf d}}
\newcommand{\bu}{{\bf u}}
\newcommand{\bv}{{\bf v}}
\newcommand{\bw}{{\bf w}}
\newcommand{\bx}{{\bf x}}

\newcommand{\bz}{{\bf z}}

\newcommand{\Aut}{\operatorname{Aut}}

\newcommand{\Comm}{\operatorname{Comm}}

\newcommand{\Mod}{\operatorname{Mod}}
\newcommand{\Mode}{\operatorname{Mod}^\pm}

\newcommand{\Out}{\operatorname{Out}}

\newcommand{\PMod}{\operatorname{PMod}}

\newcommand{\Lk}{\operatorname{Link}}
\newcommand{\Ex}{\operatorname{Ex}}

\newcommand{\sm}{\setminus}

\newcommand{\wt}[1]{\widetilde{#1}}
\newcommand{\wh}[1]{\widehat{#1}}

\newcommand{\eand}{\quad \text{ and } \quad}

\definecolor{lightgrey}{gray}{.85}

\theoremstyle{definition}

\theoremstyle{plain}
\newtheorem{thm}{Theorem}[section]
\newtheorem{main}{Theorem}
\newtheorem{lem}[thm]{Lemma}
\newtheorem{cor}[thm]{Corollary}
\newtheorem{prop}[thm]{Proposition}

\newtheorem{meta}[thm]{Metaconjecture}

\theoremstyle{definition}

\begin{document}
\title[Geometric normal subgroups in mapping class groups]{Geometric normal subgroups in mapping class groups of punctured surfaces}
\author{Alan McLeay}
\address{Mathematics Research Unit, University of Luxembourg, Esch-sur-Alzette, Luxembourg}
\email{mcleay.math@gmail.com}
\maketitle

\begin{abstract}
We prove that many normal subgroups of the extended mapping class group of a surface with punctures are geometric, that is, that their automorphism groups and abstract commensurator groups are isomorphic to the extended mapping class group.  In order to apply our theorem to a normal subgroup we require that the ``minimal supports'' of its elements satisfy a certain complexity condition that is easy to check in practice.  The key ingredient is proving that the automorphism groups of many simplicial complexes associated to punctured surfaces are isomorphic to the extended mapping class group.  This resolves many cases of a metaconjecture of N. V. Ivanov and extends work of Brendle-Margalit, who prove the result for surfaces without punctures.
\end{abstract}


\section{Introduction}\label{introduction}
The \emph{mapping class group} $\Mod(\Sigma)$ is the group of symmetries of an oriented surface $\Sigma$.  In more formal language it is the group of isotopy classes of orientation-preserving self-homeomorphisms of $\Sigma$, relative to boundary.  When denoting a specific surface we may use the notation $\Sigma_{g,n}^m$ for a surface homeomorphic to the complement of $n$ singular points and $m$ open discs in a closed surface of genus $g$.  We say that $\Sigma_{g,n}^m$ has $m$ \emph{boundary components} and $n$ \emph{punctures}.  If $\Sigma = \Sigma_{g,n}^m$ then we define $g(\Sigma) := g$ and $n(\Sigma) := n$.  When a surface has no boundary components we omit the superscript and when the surface has no punctures we usually omit the second subscript.

The \emph{extended mapping class group} $\Mode(\Sigma)$ of $\Sigma$ is the group of isotopy class of \emph{all} self-homeomorphisms of $\Sigma$, including the orientation-reversing ones.  We say that a normal subgroup $N$ of $\Mode(\Sigma)$ is \emph{geometric} if it has $\Mode(\Sigma)$ as its group of automorphisms.  In his seminal paper, Ivanov showed that if $\Sigma$ has genus at least three, or is a punctured surface of genus two, then $\Mod(\Sigma)$ is geometric \cite{IV}.  The equivalent result was given by Korkmaz for punctured tori and punctured spheres \cite{KOR}.  The proofs of these results use the action of $\Mode(\Sigma)$ on the \emph{curve complex}, a simplicial flag complex associated to $\Sigma$ which we define in Section \ref{introduction_complexes}.  Ivanov's result, and proof, acted as a springboard for a series of related results; see Bavard-Dowdall-Rafi \cite{BDR}, Brendle-Margalit \cite{Sep}, Bridson-Pettet-Souto \cite{BPS}, Irmak \cite{Nonsep}, and Kida \cite{Kida}, among many others.

\subsection{Main theorem on geometric normal subgroups}\label{introduction_normal}

In this paper we will show that many normal subgroups of $\Mode(\Sigma)$ are geometric.  The proof of this result extends work of Brendle-Margalit, who proved the theorem in the case of closed surfaces, that is, where $\Sigma = \Sigma_{g,0}$ \cite{BM17}.  In fact, these results also determine $\Comm N$, the \emph{group of abstract commensurators} of the normal subgroup $N$.  Recall that elements of $\Comm N$ are equivalence classes of isomorphisms between finite index subgroups of $N$.  Here, two isomorphisms are equivalent if they agree on some common finite index subgroup.  In this sense, the elements of $\Comm N$ are \emph{virtual} automorphisms.

Roughly, the theorem requires that some elements of the normal subgroup are supported in subsurfaces that are topologically ``small enough''.  To that end, for a mapping class $f \in \Mode(\Sigma)$ we write $R_f$ for a single-boundary subsurface such that $f$ is supported in $R_f$ and $f$ is not supported in any single-boundary proper subsurface of $R_f$.  It follows that $R_f \cong \Sigma_{k,l}^1$ some $k \le g$ and $l \le n$.  Note that there are some elements of $f$ for which $R_f$ is not defined, for example, if the support of $f$ is the entire surface $\Sigma$.

\subsection*{Elements of minimal support}
Fix a normal subgroup $N$ of $\Mode(\Sigma)$.  We say that $f \in N$ is of \emph{minimal support} if for all elements $h \in N$ such that $R_h \subset R_f$ we have that $R_h$ and $R_f$ are homeomorphic.

Consider a closed surface with positive genus, or a punctured sphere.  If $f,h \in N$ are two elements which both have minimal support then $R_f$ and $R_h$ must be homeomorphic.  For punctured surfaces with positive genus this is not true in general.

\subsection*{Elements of small support}

Let $\Sigma = \Sigma_{g,n}$ and let $N$ be a normal subgroup of $\Mode(\Sigma)$. We say that $f \in N$ is of \emph{small support} if there exist elements $h_1, h_2 \in N$ such that
\begin{align}
g &\ge g(R_f) + \max \{ g(R_{h_1}) + g(R_{h_2}) , 2\} + 1, \mbox{ and} \\ 
n &\ge n(R_f) + \max \{ n(R_{h_1}) + n(R_{h_2}) , 1 \} + 1.
\end{align}
If $g=0$ or $n=0$ we may ignore $(1)$ and $(2)$ respectively.

\begin{main}\label{AutComm}
Let $N$ be a normal subgroup of $\Mode(\Sigma)$.  If every element of minimal support in $N$ is of small support then the natural homomorphisms
\[
\Mode(\Sigma) \to \Aut N \to \Comm N
\]
are isomorphisms.
\end{main}

If $N$ is a normal subgroup of $\Mod(\Sigma)$ which is not normal in $\Mode(\Sigma)$ it can be shown using similar methods that $\Aut N \cong \Mod(\Sigma)$, see Brendle-Margalit \cite[Section 6]{BM17} and the author \cite[Section 5]{BraidMeta}.  We note that finding such a subgroup is itself an interesting problem.

Suppose $N$ contains an element of small support.  It follows that at least one of the elements of minimal support in $N$ will necessarily be of small support.  Furthermore, if $g=0$ or $n=0$ then \emph{all} elements of minimal support in $N$ are of small support.  This observation allows for the statement of the theorem to be consideribly simpler in these special cases.  In particular, if $n=0$ and $N$ contains an element $f$ of small support, that is, $g \ge 3g(R_f)+1$, then Theorem \ref{AutComm} applies, see \cite{BM17}.

We now discuss two applications of Theorem \ref{AutComm}.

\subsection*{The Johnson filtration is geometric}
We may apply Theorem \ref{AutComm} to a well known sequence of normal subgroups.  Write $\Gamma_0$ for the fundamental group of the surface $\Sigma$.  Consider now the lower central series of $\Gamma_0$, that is, $\Gamma_k := [\Gamma_0, \Gamma_{k-1}]$ for any $k > 0$.  There is a natural action of $\Mod(\Sigma)$ on the quotient group $\Gamma_0 / \Gamma_k$.  We may now define for each $k \ge 0$ the group
\[
\vJ_k(\Sigma) := \ker \big (\Mod(\Sigma) \to \Out(\Gamma_0 / \Gamma_k) \big ).
\]
It was shown that this sequence of groups is a filtration by Bass-Lubotzky \cite{BL94}.  Due to the work of Johnson, we name the sequence \emph{the Johnson filtration} \cite{DJ1} \cite{DJ2}.  The first term in the Johnson filtration is known as the \emph{Torelli group}.  This group has been studied by Brendle-Margalit-Putman \cite{BMP}, Kasahara \cite{YK01}, Mess \cite{GM92}, and Putman \cite{AP07} \cite{AP12}, to name only a few.  It was shown by Farb-Ivanov that the Torelli group is geometric \cite{FI05}.  Furthermore, the second term, \emph{the Johnson kernel}, is also geometric.  This is a result of Brendle-Margalit for closed surfaces \cite{Sep}, and Kida for punctured surfaces \cite{Kida}.  Farb then asked the question for what values of $k \ge 2$ is $\vJ_k(\Sigma)$ geometric \cite{FarbProblems}.  It was shown by Bridson-Pettet-Souto \cite{BPS} and Brendle-Margalit \cite{BM17} that if $g \ge 7$ then $\vJ_k(\Sigma_{g,0})$ is geometric for all $k \ge 0$.  We may apply Theorem \ref{AutComm} in order to answer this question for punctured surfaces.

\begin{cor}\label{JF}
Let $\Sigma = \Sigma_{g,n}$ such that $g,n \ge 5$.  Then the natural homomorphisms
\[
\Mode(\Sigma) \to \Aut \vJ_k(\Sigma) \to \Comm \vJ_k(\Sigma)
\]
are isomorphisms for any $k \ge 0$.
\end{cor}

See Section \ref{section_JF} for details on the bounds on $g$ and $n$ given in Corollary \ref{JF}.

\subsection*{Surface braid groups}
We may also apply Theorem \ref{AutComm} to the \emph{surface braid group} $\vB_{g,n}$, that is, the kernel of the homomorphism
\[
\Mod(\Sigma_{g,n}) \to \Mod(\Sigma_{g,0}),
\]
induced by the forgetful map $\Sigma_{g,n} \to \Sigma_{g,0}$.  Groups of this type have already been shown to be geometric by Irmak-Ivanov-McCarthy \cite{IIM} and An \cite{An}.  Now, if $f \in \vB_{g,n}$ is of minimal support then $R_f$ is homeomorphic to either $\Sigma_{1,1}^1$ or $\Sigma_{0,2}^1$.  Furthermore, if $n < 3$ then there are no elements of small support in $\vB_{g,n}$.  We therefore have the following corollary.

\begin{cor}
If $g \ge 4$ and $n \ge 5$ then the natural homomorphisms
\[
\Mode(\Sigma_{g,n}) \to \Aut \vB_{g,n} \to \Comm \vB_{g,n}
\]
are isomorphisms.
\end{cor}

\subsection*{Conjectured definition of small support}
We note that the bounds given for Theorem \ref{AutComm} are not strict.  Indeed, consider $\Sigma= \Sigma_{g,n}$ where $g >0$, and $N = \Mod(\Sigma)$ with an element $f$of minimal support such that $R_f \cong \Sigma_{1,0}^1$ (for example, a Dehn twist about a nonseparating curve). In order to apply Theorem \ref{AutComm} we require that
\[
g \ge g(R_f) + 2 + 1 = 4.
\]
It has been shown however by Ivanov \cite{IV} and Korkmaz \cite{KOR} that $N=\Mod(\Sigma)$ is geometric for surfaces of genus one, two, and three.

We conjecture that the definition of small support may improved as follows: an element $f \in N$ is of \emph{small support} if there exists some $h \in N$ such that $R_f$ and $R_h$ do not intersect and have non isotopic boundary components.  A similar conjecture was made by Brendle-Margalit for the closed case \cite[Conjecture 1.5]{BM17}.  This is supported by the recent work of Clay-Mangahas-Margalit \cite{CMM}.

\subsection{Complexes of regions}\label{introduction_complexes}
A \emph{region} is a compact, connected subsurface of a surface $\Sigma$ such that each boundary component is an essential simple closed curve.  We define $\vR(\Sigma)$ to be the set of $\Mode(\Sigma)$-orbits of regions in $\Sigma$.  For any subset of orbits $A \subset \vR(\Sigma)$ we say that a region $R$ is \emph{represented} in $A$ if the $\Mode(\Sigma)$-orbit of $R$ belongs to $A$.  We now define a \emph{complex of regions} $\vC_A(\Sigma)$ to be a simplicial flag complex whose vertices correspond to all homotopy classes of regions represented in $A$.  If a vertex $v$ corresponds to the homotopy class of a region $R$, we usually say that $v$ corresponds to $R$.  Two vertices of $\vC_A(\Sigma)$ span an edge when they correspond to disjoint regions.

\subsection*{The curve complex}
If $A \subset \vR(\Sigma)$ is the set of orbits of annular regions then the complex $\vC_A(\Sigma)$ is called the \emph{curve complex}.  In this case it makes sense to think of homotopy classes of annuli as isotopy classes of essential simple closed curves.  This complex has been of fundamental importance in the study of mapping class groups and Teichm\"{u}ller space, see Hamenst{\" a}dt \cite{Ham14}, Ivanov \cite{IV}, Masur-Minsky \cite{MMHyper}, and Rafi-Schleimer \cite{RS09}, to name only a few.  Due to its importance, we reserve the notation $\vC(\Sigma)$ for the curve complex.

As discussed before, Ivanov proved that $\Mod(\Sigma)$ is geometric by studying the action of the extended mapping class group on the curve complex.  In particular he showed that $\Aut \vC(\Sigma) \cong \Mode(\Sigma)$.  By using the fact that powers of Dehn twists about distinct curves commute if and only if the curves are disjoint one is able to construct an isomorphism
\[
\Comm \Mod(\Sigma) \to \Aut \vC(\Sigma).
\]
This is a key step in Ivanov's application of automorphisms of $\vC(\Sigma)$ to automorphisms (and commensurations) of $\Mod(\Sigma)$.

\subsection*{The complex of domains}
If $A = \vR(\Sigma)$ then $\vC_A(\Sigma)$ is called the \emph{complex of domains}.  In some sense, the complex of domains is the extreme generalisation of the curve complex.  Indeed, McCarthy-Papodopoulos proved that if $\Sigma$ is closed, or has a single puncture, then $\Aut \vC_A(\Sigma) \cong \Mode(\Sigma)$ \cite{MCP}.  If $\Sigma$ has more than one puncture there exist automorphisms of $\vC_A(\Sigma)$ that are not induced by mapping classes.

Suppose $v_1,v_2 \in \vC_A(\Sigma)$ are the vertices described in Figure \ref{corkhole}(i).
\begin{figure}[t]
\centering
\labellist \hair 1pt
	\pinlabel {(i)} at 300 -42
	\pinlabel {(ii)} at 1600 -42
    \endlabellist
\includegraphics[scale=0.15]{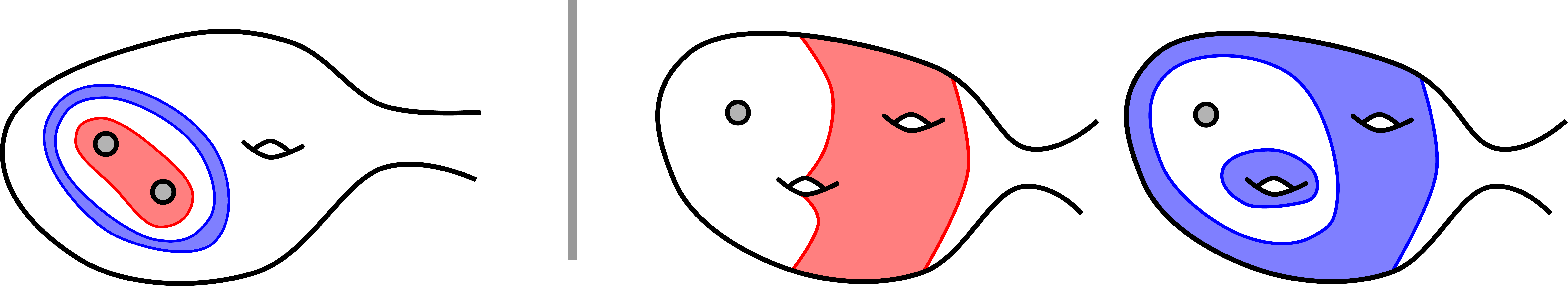}
\caption{(i) Any region that has essential intersection with the annulus must also intersect the disc with two punctures.  If $v_1$ and $v_2$ are vertices of a complex of regions that correspond to these two regions then $v_1$ and $v_2$ span an edge, and any other vertex spans an edge with $v_1$ if and only if it spans an edge with $v_2$. (ii) Suppose no subsurface of a once punctured nonseparating annulus is represented in $A$. If $v_1, v_2$ are vertices of $\vC_A(\Sigma)$ corresponding to the regions shown then a vertex spans an edge with $v_1$ if and only if spans an edge with $v_2$.}
\label{corkhole}
\end{figure}
We may define an order two automorphism $\phi \in \Aut \vC_A(\Sigma)$ such that $\phi(v_1) = v_2$, $\phi(v_2) = v_1$, and $\phi(v) = v$ for all other vertices $v \in \vC_A(\Sigma)$.  Any automorphism that swaps two distinct vertices and fixes all others in this way is called an \emph{exchange automorphism}.  In Section \ref{InjectEx} we discuss exchange automorphisms further.  In particular, we visit the fact that exchange automorphisms occur in complexes of regions if and only if there are vertices of the following type.

\subsection*{Corks and holes}
We say a vertex $v$ of $\vC_A(\Sigma)$ is a \emph{cork} if it corresponds to an annulus with complementary region $Q$ represented in $A$ with no proper, non-peripheral subsurface of $Q$ represented in $A$.  If $u$ corresponds to $Q$ we call the vertices $u$ and $v$ a \emph{cork pair}.  See Figure \ref{corkhole}(i) for an example of a cork pair.

A vertex $v$ of $\vC_A(\Sigma)$ is a \emph{hole} if it corresponds to a region that has a complementary region $Q$ such that no subsurface of $Q$ is represented in $A$.  See Figure \ref{corkhole}(ii) for an example vertices of a complex of regions that are holes.

\subsection*{The metaconjecture of Ivanov}
There are many other complexes of regions that have been studied; such as the complex of strongly separating curves by Bowditch \cite{Strongly}, the complex of separating curves by Brendle-Margalit \cite{Sep}, the complex of nonseparating curves by Irmak \cite{Nonsep}, the arc complex by Irmak-McCarthy \cite{Arc}, the arc and curve complex by Korkmaz-Papadopolous \cite{Arccurve}, and the truncated complex of domains by McCarthy-Papadopolous \cite{MCP}, among others.  Each of these complexes has been shown to have the extended mapping class group as its group of automorphisms for all but finitely many low complexity surfaces.  Furthermore, there are numerous other complexes associated to surfaces which are not complexes of regions. It was shown that the extended mapping class group is the group of automorphisms of; the Torelli complex by Farb-Ivanov \cite{FI05}, the flip graph by Korkmaz-Papadopolous \cite{Triangulation}, and the pants complex by Margalit \cite{Pants}.  In each case there are restrictions on the surfaces for which the result holds.  These results led Ivanov to make a metaconjecture \cite{FarbProblems}.
\begin{meta}[Ivanov]
Every object naturally associated to a surface $\Sigma$ and having sufficiently rich structure has $\Mode(\Sigma)$ as its group of automorphisms.  Moreover, this can be proved by a reduction to the theorem about automorphisms of $\vC(\Sigma)$.
\end{meta}

This paper partially resolves the metaconjecture for complexes of regions related to a surface $\Sigma=\Sigma_{g,n}$.  Furthermore, Theorem \ref{AutComm} resolves the metaconjecture where we consider normal subgroups as \emph{objects naturally associated} to $\Sigma$ and the conditions of the theorem provide \emph{sufficiently rich structure}.  This extends work of Brendle-Margalit, who deal with the case where $n=0$ \cite{BM17}.  The case where $g=0$ is the focus of a previous paper by the author \cite{BraidMeta}.  We may assume throughout this paper therefore that $g,n>0$.

\subsection{Main theorem on complexes of regions}\label{introduction_metaconjecture}

For any region $R \subset \Sigma$ an \emph{enveloping region} $\wh R$ of $R$ is a single-boundary region such that $R \subset \wh R$ and $R$ is not a subsurface of any proper single-boundary region contained in $\wh R$.  Let $v$ be a vertex of a complex of regions corresponding to the region $R_v$.  We write $\wh v \subset \Sigma$ for the enveloping region of the region $R_v$, that is, $\wh v := \wh R_v$. 

\subsection*{Minimal vertices}
Let $\vC_A(\Sigma)$ be a complex of regions.  We say that a vertex $v \in \vC_A(\Sigma)$ is \emph{minimal} if for any vertex $u$ such that $\wh u \subset \wh v$, we have that $\wh u$ and $\wh v$ are homeomorphic.  If a vertex $v$ is minimal, then every vertex in the $\Mode(\Sigma)$-orbit of $v$ is also minimal.

The following definition is equivalent to that of elements of small support given in Section \ref{introduction_normal}

\subsection*{Small vertices}
Let $\Sigma = \Sigma_{g,n}$ and let $\vC_A(\Sigma)$ be a complex of regions.  We say that a vertex $v \in \vC_A(\Sigma)$ is \emph{small} if there exist two vertices $u_1, u_2$ such that
\begin{align}
g &\ge g(\wh v) + \max \{ g( \wh u_1) + g( \wh u_2) , 2 \} + 1, \mbox{ and} \\ 
n &\ge n( \wh v) + \max \{ n( \wh u_1) + n( \wh u_2), 1\} + 1.
\end{align}
As before, if $g=0$ or $n=0$ then we ignore $(3)$ and $(4)$ respectively.

\begin{main}\label{BigDaddy}
Let $\vC_A(\Sigma)$ be a complex of regions.  Suppose that every minimal vertex of $\vC_A(\Sigma)$ is small.  Then the natural homomorphism
\[
\eta_A : \Mode(\Sigma) \to \Aut \vC_A(\Sigma)
\]
is an isomorphism if and only if $\vC_A(\Sigma)$ has no holes and no corks.
\end{main}

\subsection*{Outline of the paper}
The majority of the paper is dedicated to proving Theorem \ref{BigDaddy}.  In Section \ref{InjectEx} we first discuss injectivity of the natural homomorphism $\eta_A$ and then exchange automorphisms of complexes of regions.  Very roughly, the proof of Theorem \ref{BigDaddy} proceeds by defining two complexes $\vC_{\vS(A)}(\Sigma)$ and $\vC_{\partial A}(\Sigma)$ which carry similar information to $\vC_A(\Sigma)$.  In each case we prove that the usual natural homomorphism from $\Mode(\Sigma)$ to the group of automorphisms is an isomorphism.
\[
\begin{tikzcd}
& \Mode(\Sigma) \arrow{ddl}[swap]{\eta_A} \arrow{dd}[swap]{\eta_{\partial A}} \arrow{ddr}[swap]{\eta_{\vS(A)}} \arrow{ddrrr}{\cong} & & & \\ \\ 
\Aut \vC_A(\Sigma) \arrow{r} & \Aut \vC_{\partial A}(\Sigma) \arrow{r} & \Aut \vC_{\vS(A)}(\Sigma) \arrow{r} & \cdots \arrow{r} & \Aut \vC(\Sigma)
\end{tikzcd}
\]

In Section \ref{section_subcomplexes} we define a subcomplex $\vC_{\vS(A)}(\Sigma)$ of the curve complex $\vC(\Sigma)$ related to a complex of regions.  We then prove in Theorem \ref{etas} that the homomorphism $\eta_{\vS(A)}$ is an isomorphism.  In Section \ref{section_DS} we define a second complex $\vC_{\partial A}(\Sigma)$.  In this case, the vertices correspond to so-called \emph{dividing sets}, multicurves in $\Sigma$ that separate the surface into precisely two components.  In Theorem \ref{etad} we show that the natural homomorphism $\eta_{\partial A}$ from $\Mode(\Sigma)$ to the automorphism group of this complex is also an isomorphism.  In Section \ref{theproof} we use Theorem \ref{etad} to prove Theorem \ref{BigDaddy}, that is, every homomorphism in the diagram above is an isomorphism.  This outline is analogous to that of Brendle -Margalit \cite[Theorem 1.7]{BM17}.

Finally, in Section \ref{section_normal}, we prove Theorem \ref{AutComm} as an application of Theorem \ref{BigDaddy}.  Similar to Ivanov's application of the curve complex result, the proof relies on constructing a homomorphism
\[
\Comm N \to \Aut \vC_N(\Sigma),
\]
where $\vC_N(\Sigma)$ is a complex of regions associated to a normal subgroup $N$ of $\Mode(\Sigma)$.  This argument uses the mathematical machinery developed by Brendle-Margalit for the closed case.  As such, some details are omitted and appropriate references are given to their paper \cite[Section 6]{BM17}.

\subsection*{Acknowledgments}
The author would like to thank his supervisor, Tara Brendle, for her helpful guidance and support. He is grateful to Dan Margalit for several helpful discussions and suggestions that greatly improved the paper. He would also like to thank Javier Aramayona, Vaibhav Gadre, Tyrone Ghaswala, Chris Leininger, Johanna Mangahas, and Shane Scott for their support and helpful discussions about the paper.

\section{Preliminary results}\label{InjectEx}

In this section we prove that the homomorphism $\eta_A$ from Theorem \ref{BigDaddy} is injective.  This result is in fact more general and will be used many times throughout the paper.  Following the work of McCarthy-Papadopoulos \cite[Section 4]{MCP} and Brendle-Margalit \cite[Section 2]{BM17} we then look at the precise conditions for a complex of regions $\vC_A(\Sigma)$ to admit exchange automorphisms as defined in Section \ref{introduction_complexes}.

\subsection{Injectivity}
We will first prove that $\eta_A$ is an injective group homomorphism.

\begin{lem}\label{Injectivity}
Let $\Sigma = \Sigma_{g,n}$ be a surface such that $g,n>0$.  If $\vC_A(\Sigma)$ is connected then the natural homomorphism
\[
\eta_A : \Mode(\Sigma) \rightarrow \Aut \vC_A(\Sigma)
\]
is injective.
\end{lem}

\begin{proof}
Let $\bc$ be a nonseparating curve in $\Sigma$ and let $R := \Sigma \sm \bc$. The subsurface $R$ is filled by regions that are represented in $A$. Equivalently, there exist regions represented in $A$ whose boundary components are curves that fill $R$.  It follows then that if $f \in \Mode(\Sigma)$ is in the kernel of $\eta_A$ it must also fix $\bc$. Since our choice of $\bc$ was arbitrary we can find a pants decomposition $P$ of $\Sigma$ such that $f$ fixes every curve in $P$. We conclude that $f$ is a product of Dehn twists and is therefore orientation preserving.  In particular, $f \in \PMod(\Sigma)$.  If $T_\bc$ is the Dehn twist defined by the curve $\bc$ then we have that
\[
T_\bc = T_{f(\bc)} = f T_\bc f^{-1}.
\]
Since our choice of $\bc$ was arbitrary and $\PMod(\Sigma)$ is generated by Dehn twists we have that $f$ is in the centre of $\PMod(\Sigma)$.  The centre of $\PMod(\Sigma)$ is trivial and so $\eta_A$ is injective.
\end{proof}

\subsection{Exchange automorphisms}\label{Exchange}
Recall the definitions of exchange automorphisms, holes, and corks from Section \ref{introduction_complexes}
We state two results of Brendle-Margalit relating these notions which together imply the `only if' condition in the statement of Theorem \ref{BigDaddy}.  Note that Brendle-Margalit state Theorems \ref{Ex1} and \ref{Ex2} for closed surfaces only \cite[Theorem 2.1, Theorem 2.2]{BM17}.  The proofs can be adapted for surfaces with punctures using the notion of a small vertex given in Section \ref{introduction_metaconjecture}.

\begin{thm}[Brendle-Margalit]\label{Ex1}
Let $\Sigma$ be a punctured surface or a closed surface of genus $g \ge 3$. Let $\vC_A(\Sigma)$ be a complex of regions with no isolated vertices or edges.  Then $\vC_A(\Sigma)$ admits exchange automorphisms if and only if it has a hole or a cork.  Moreover, two vertices can be exchanged by an exchange automorphism if and only if they are holes with equal fillings or they form a cork pair.
\end{thm}

As an example we consider the cork pair and the holes depicted in Figure \ref{corkhole}.  Recall that the \emph{link} of a vertex $v$, denoted $\Lk(v)$, is the set of all vertices that span an edge with $v$ in the complex. The \emph{star} of a vertex is the union of the vertex and its link.  The vertices corresponding to the regions in Figure \ref{corkhole}(i) have equal stars.  Similarly the vertices described in Figure \ref{corkhole}(ii) have equal links and the vertices do not span an edge with each other.  If two vertices have equal links or equal stars then we can define an automorphism that exchanges the vertices.  The proof of \cite[Theorem 2.1]{BM17} tells us that if two vertices have equal links then they are holes, and if they have equal stars then they are cork pairs.

When the automorphism group of a complex of regions does contain exchange automorphisms, Brendle-Margalit give us an explicit description of the automorphism group of the complex.

\begin{thm}[Brendle-Margalit]\label{Ex2}
Let $\vC_A(\Sigma)$ be a complex of regions that is connected.  If every minimal vertex of $\vC_A(\Sigma)$ is small then
\[
\Aut \vC_A (\Sigma) \cong \Ex \vC_A (\Sigma) \rtimes \Mode(\Sigma).
\]
\end{thm}
\noindent Here, $\Ex  \vC_A (\Sigma)$ is the normal subgroup of $\Aut \vC_A(\Sigma)$ generated by all exchange automorphisms.

\section{Subcomplexes of the separating curve complex}\label{section_subcomplexes}

Given a surface $\Sigma$, let $\vS$ be the set of $\Mode(\Sigma)$-orbits of separating curves in $\Sigma$. We denote by $\vC_\vS(\Sigma)$ the \emph{separating curve complex}, the subcomplex of $\vC(\Sigma)$ spanned by vertices corresponding to separating curves.  In this section we study the automorphisms of particular subcomplexes of the separating curve complex.  

For any separating curve $\bc$ in $\Sigma$ there are two \emph{associated regions} defined by cutting $\Sigma$ along $\bc$.  For any subset $A \subset \vR(\Sigma)$ we say that a $\bc$ separates regions represented in $A$ if both of its associated regions contain regions represented in $A$.  We define $\vC_{\vS(A)}(\Sigma)$ to be the subcomplex of $\vC_\vS(\Sigma)$ spanned by vertices corresponding to curves that separate regions represented in $A$.  The main goal of this section is to prove the following theorem.

\begin{thm}\label{etas}
Let $A \subset \vR(\Sigma)$ and let $\vC_{\vS(A)}(\Sigma)$ be the subcomplex of $\vC_\vS(\Sigma)$ defined above.  If every minimal vertex of $\vC_{\vS(A)}(\Sigma)$ is small then the natural homomorphism
\[
\eta_{\vS(A)} : \Mode(\Sigma) \to \Aut \vC_{\vS(A)}(\Sigma)
\]
is an isomorphism.
\end{thm}
Proving Theorem \ref{etas} is the first step on the proof of Theorem \ref{BigDaddy}.  Note that we may consider $\vC_{\vS(A)}(\Sigma)$ to be a complex of regions (since $\vS(A) \subset \vR(\Sigma)$) and so the definitions of minimal and small vertices of $\vC_{\vS(A)}(\Sigma)$ make sense.  Indeed, Theorem \ref{etas} is just a special case of Theorem \ref{BigDaddy}.  For a surface $\Sigma = \Sigma_{g,n}$, Theorem \ref{etas} has been proven for the cases when $n=0$ \cite[Theorem 1.10]{BM17} and $g=0$ \cite[Theorem 1.5]{BraidMeta}.  This section will deal with the general case when $g,n>0$.

We will prove Theorem \ref{etas} in three steps;

\subsection*{Step 1}
We consider a subset $X \subseteq \vS$ such that $\vS(A) \subseteq X$.  We may then define the subcomplex $\vC_X(\Sigma)$ in the usual way.  If $u$ and $v$ are two vertices of $\vC_X(\Sigma)$ that correspond to curves with homeomorphic associated regions we say that $u$ and $v$ are of the same \emph{vertex type}.  In Proposition \ref{vertex_prop} we prove that for any vertex $v \in \vC_X(\Sigma)$ and any automorphism $\phi \in \Aut \vC_X(\Sigma)$ the vertex $\phi(v)$ is of the same vertex type as $v$.

\subsection*{Step 2}
We consider $\vC_X(\Sigma)$ as above and let $\vC_Y(\Sigma)$ be the subcomplex of $\vC_X(\Sigma)$ obtained by removing all vertices of a particular vertex type.  In Proposition \ref{Step} we prove that if $\eta_X : \Mode(\Sigma) \to \Aut \vC_X(\Sigma)$ is an isomorphism then under certain conditions $\eta_Y : \Mode(\Sigma) \to \Aut \vC_Y(\Sigma)$ is also an isomorphism.

\subsection*{Step 3}
We have from Brendle-Margalit \cite{Sep} and Kida \cite{Kida} that the homomorphism
\[
\eta_\vS : \Mode(\Sigma) \to \Aut \vC_\vS(\Sigma)
\]
is an isomorphism.  We define a sequence of subcomplexes starting with $\vC_\vS(\Sigma)$ and ending with $\vC_{\vS(A)}(\Sigma)$. We are then able to use Step 2 (Proposition \ref{Step}) repeatedly in order to show that the homomorphism $\eta_{\vS(A)}$ from the statement of Theorem \ref{etas} is an isomorphism.

More informally, we begin with the isomorphism $\eta_\vS$ and show that by removing vertex types from the complex, we sustain an isomorphism between the extended mapping class group and the automorphism group of the subcomplex of separating curves.  Ensuring that the conditions of Proposition \ref{Step} are met follows from the assumption that every minimal vertex of $\vC_{\vS(A)}(\Sigma)$ is small.

\subsection{Characteristic vertex types}\label{vertex_types}

In order to tackle Step 1 we will introduce some terminology to help determine different vertex types. We call a separating curve $\bc$ in $\Sigma$ a \emph{$(k,l)$-curve} if it has an associated region $R$ of genus $k$ with $l$ punctures.  Note that a $(k,l)$-curve is also a $(g-k,n-l)$-curve.

If $\bc$ is a $(k,l)$-curve then for any $f \in \Mode(\Sigma)$ we have that $f(\bc)$ is a $(k,l)$-curve.
We call any vertex of $\vC_{\vS}(\Sigma)$ that corresponds to a $(k,l)$-curve a \emph{$(k,l)$-vertex}.  Our goal is therefore to show that the subset of $(k,l)$-vertices is characteristic in certain subcomplexes of $\vC_\vS(\Sigma)$ for any $k$ and $l$.

\subsection*{Sides}
Given a vertex $v$ of a subcomplex of separating curves $\vC_X(\Sigma)$ we say that vertices $u,w$ lie on the \emph{same side} of $v$ if $u,w \in \Lk (v)$ and there exists another vertex in $\Lk (v)$ that does not span an edge with either $u$ or $w$.

The following definitions will be useful when showing that vertex types form characteristic subsets.

\subsection*{Linear simplices}
We define a simplex $\sigma$ of $\vC_X(\Sigma)$ to be \emph{linear} if there is a labeling of its vertices $v_0,\dots, v_m$ such that $v_{i-1}$ and $v_{i+1}$ do not lie on the same side of $v_i$ for all $i = 1, \dots, m-1$. We call the vertices $v_0$ and $v_m$ the  \emph{extreme vertices} of the linear simplex $\sigma$.  We say that a linear simplex $\sigma \subset \vC_X(\Sigma)$ is \emph{maximal} if its vertices do not form a subset of another linear simplex.  Note that a $(1,0)$- or $(0,2)$-vertex $v$ belongs to a linear simplex $\sigma$ only when $v$ is an extreme vertex of $\sigma$.  Indeed, in such cases all vertices of $\Lk(v)$ lie on the same side.

For any two vertices $u, v \in \vC_X(\Sigma)$ we say that $u$ is an \emph{increment} of $v$ in $\vC_X(\Sigma)$ if there exists a maximal linear simplex $\sigma \subset \vC_X(\Sigma)$ in which $u$ and $v$ are sequential with respect to the labeling.  Note that $u$ is an increment of $v$ if and only if $v$ is an increment of $u$.

\subsection*{Linear subcomplexes}
We say that a subcomplex $\vC_X(\Sigma)$ of $\vC_\vS(\Sigma)$ is \emph{linear} if
\begin{center}
$u$ is an increment of $v$ in $\vC_X(\Sigma)$ $\Longrightarrow$ $u$ is an increment of $v$ in $\vC_\vS(\Sigma)$.
\end{center}

\subsection*{Genus increments}
Suppose the vertex $u$ is an increment of the vertex $v$ in some linear subcomplex $\vC_X(\Sigma)$.  Since $\vC_X(\Sigma)$ is linear the vertex $u$ must be an increment of $v$ in $\vC_\vS(\Sigma)$. Suppose $u$ and $v$ correspond to the boundary components of a region $R$ in $\Sigma$.  We observe that $R$ is homeomorphic to either $\Sigma_{1,0}^2$ or $\Sigma_{0,1}^2$.  Indeed, if this is not the case then we can find a curve in $R$ that separates the boundary components of $R$.  The existence of such a curve contradicts the fact that $u$ is an increment of $v$ in $\vC_\vS(\Sigma)$ and therefore $\vC_X(\Sigma)$ cannot be linear.

Given vertices $u, v \in \vC_X(\Sigma)$ and the region $R$ above; we say that $u$ is a \emph{genus increment} of $v$ if $R$ is homeomorphic to $\Sigma_{1,0}^2$, see Figure \ref{GenusIncrement}.  Once again, note that if $u$ is a genus increment of $v$, then $v$ is a genus increment of $u$.

\begin{lem}\label{Incr}
Let $\vC_X(\Sigma)$ be a linear subcomplex of $\vC_\vS(\Sigma)$ and let $u,v$ be vertices of $\vC_X(\Sigma)$. If $u$ is a genus increment of $v$ then $\phi(u)$ is a genus increment of $\phi(v)$ for all $\phi \in \Aut \vC_X(\Sigma)$.
\end{lem}

\begin{proof}
We claim that the vertex $u$ is a genus increment of $v$ if and only if there exist vertices $x$ and $y$, such that the vertex set $\{u,v,x,y\}$ spans a square in $\vC_X(\Sigma)$. The result then follows from the claim. The forward implication of the claim is clear. We take appropriate Dehn twists of representative curves of $u$ and $v$, see Figure \ref{GenusIncrement}.
\begin{figure}[t]
\centering
\labellist \hair 1pt
	\pinlabel {$\bu$} at 400 420
	\pinlabel {$\bv$} at 840 400
    \endlabellist
\includegraphics[scale=0.15]{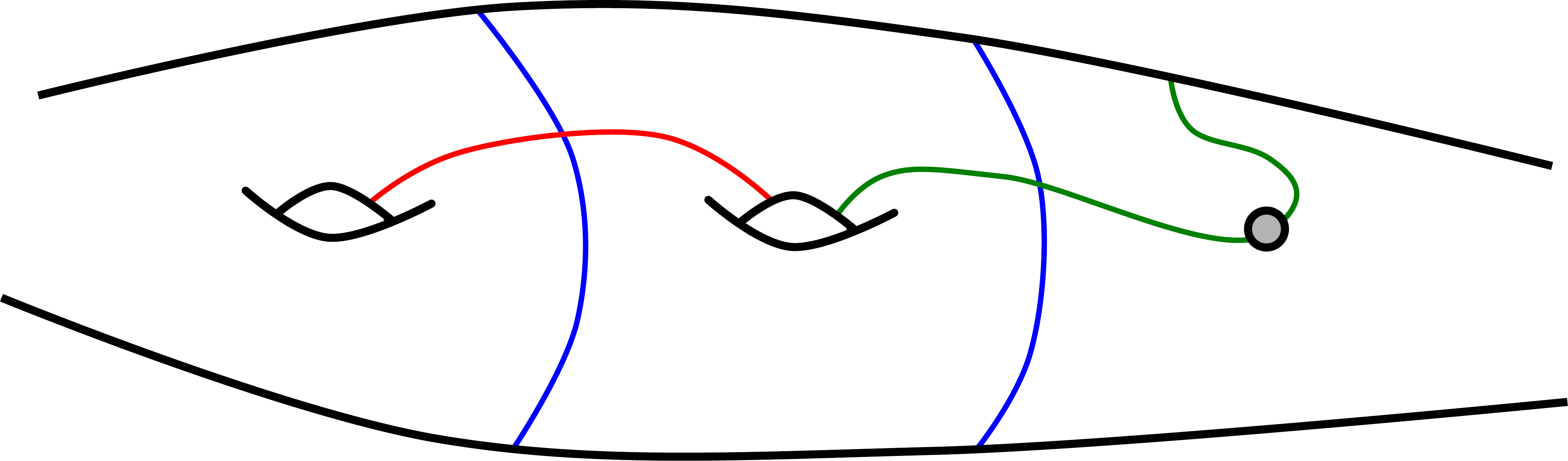}
\caption{The vertex corresponding to the curve $\bu$ is a genus increment of the vertex corresponding to the curve $\bv$.}
\label{GenusIncrement}
\end{figure}

Suppose now that the vertex $u$ is an increment of the vertex $v$ and vertices $x$ and $y$ exist as in the claim. Let $R$ be a region such that $u$ and $v$ correspond to the the boundary components of $R$. Assume $u$ is not a genus increment of $v$, that is, $R$ is homeomorphic to a punctured annulus. Let the vertices $u$ and $x$ correspond to the curves $\bu$ and $\bx$ respectively such that $\bu$ and $\bx$ are in minimal position. Take $B$ to be the regular neighbourhood of $\bu \cup \bx$. One of the components of $\partial B$ is the boundary of a disc $D_1 \subset R$ with a single puncture. Since $y$ does not span an edge with $v$ it corresponds to a curve that intersects either to $\bu$ or the disc $D_1$. This implies that $y$ fails to span an edge with either $u$ or $x$, a contradiction.
\end{proof}

We can now begin to prove that vertex types form characteristic subsets of a linear subcomplex $\vC_X(\Sigma)$.  We do this in two steps, the first of which deals with the minimal vertices of $\vC_X(\Sigma)$.  We will make use of the following result of Andrew Putman \cite{PutmanConnect}.

\begin{lem}[Putman]\label{Put}
Let $G$ be a group acting on a simplicial complex $X$ with $v$ a fixed vertex in $X^0$. Let $S$ be a set of generators of $G$ and assume that;
\begin{enumerate}
\item for all $u \in X^0$, the orbit $G \cdot v$ intersects the connected component of $X$ containing $u$, and
\item for all $s \in S^{\pm 1}$, there is a path $P_s$ in $X$ from $v$ to $s \cdot v$.
\end{enumerate}
Then $X$ is connected.
\end{lem}

Let $v$ be a $(k,l)$-vertex of $\vC_X(\Sigma)$.  Recall from Section \ref{introduction} that $v$ is small if there exists a $(k_1,l_1)$-vertex and a $(k_2,l_2)$-vertex in $\vC_X(\Sigma)$ such that;
\begin{align}
g &\ge k + \max\{ k_1 + k_2 , 2\} + 1, \mbox{ and} \\ 
n &\ge l + \max \{ l_1 + l_2 , 1\} + 1.
\end{align}

\begin{lem}\label{MinCurve}
Let $\vC_X(\Sigma)$ be a linear subcomplex of $\vC_\vS(\Sigma)$ where every minimal vertex is small. Let $v$ be a $(k,l)$-vertex of $\vC_X(\Sigma)$.  If $v$ is a minimal vertex then $\phi(v)$ is a $(k,l)$-vertex for all $\phi \in \Aut \vC_X(\Sigma)$.
\end{lem}

\begin{proof}
We first show that the set of minimal vertices is characteristic.  This is clear, as a vertex $v$ is minimal in $\vC_X(\Sigma)$ if and only if it is an extreme vertex of some maximal linear simplex.

Assume then that the set of minimal vertices contains $(k,l)$-vertices for some values of $k$ and $l$.  We need to show that vertices of this type form a characteristic subset.  For any two minimal vertices $v_1,v_2$ we will write $v_1 \sim v_2$ if;
\begin{enumerate}
\item there exists a vertex $u$ such that $u$ and $v_i$ are extreme vertices of some maximal linear simplex $\sigma_i$,
\item the simplices $\sigma_1$ and $\sigma_2$ have $N$ vertices, and
\item the simplices $\sigma_1$ and $\sigma_2$ have $K \le N$ genus increments.
\end{enumerate}
Let $v_1$ and $v_2$ correspond to the curves $\bc_1$ and $\bc_2$ respectively.  Suppose $u$ corresponds to a curve with associated region $Q$ disjoint from $\bc_1$ and $\bc_2$.  It follows that $\bc_1$ and $\bc_2$ bound regions $R_1$ and $R_2$ such that; $g(R_i) = g(Q) + K$ and that $n(R_i) = n(Q) + N - K$.  We conclude therefore that if $v_1 \sim v_2$ then they are of the same vertex type.

Let $\vM$ be the graph whose vertex set is all the minimal vertices of $\vC_X(\Sigma)$.  Two vertices $v_1,v_2 \in \vM$ share an edge whenever $v_1 \sim v_2$.  Let $v$ be some fixed vertex of $\vM$ corresponding to the curve $\bv$. By the definition of `$\sim$' we have that if two vertices are connected in $\vM$ then the are of the same vertex type.  Let $\vM(k,l)$ be the subgraph of $\vM$ spanned by $(k,l)$-vertices.  The mapping class group $\Mode(\Sigma)$ acts naturally on $\vM(k,l)$ and each vertex $u \in \vM(k,l)$ corresponds to some curve $f(\bv)$.  This implies that the first condition of Lemma \ref{Put} is satisfied with respect to the subgraph $\vM(k,l)$.

There exists a generating set $S$ of $\Mode(\Sigma)$ such that every element of $S$ fixes $\bv$ except a $l$ Dehn twists and a single half twist, see Figure \ref{minimal_curve_graph}.
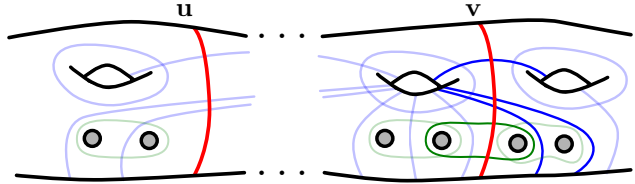
\begin{figure}[t]
\centering
\labellist \hair 1pt
	\pinlabel {$\bu$} at 350 340
	\pinlabel {$\bv$} at 920 340
	\pinlabel {{\huge $\dots$}} at 550 20
	\pinlabel {{\huge $\dots$}} at 550 290
    \endlabellist
\includegraphics[scale=0.2]{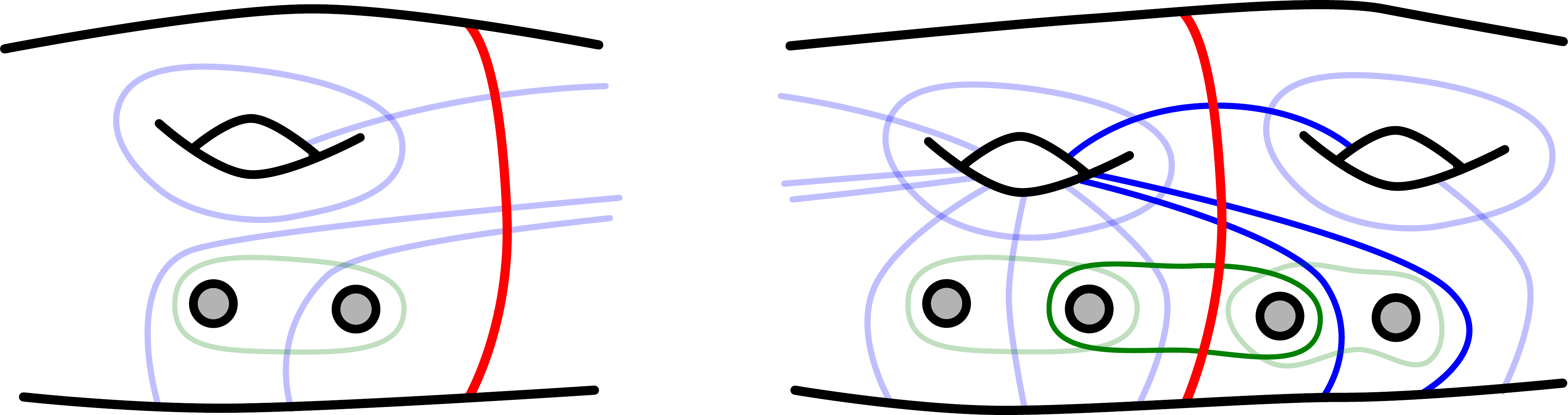}
\caption{Any Dehn twist or half twist about one of the curves shown either fixes the $(k,l)$-curve $\bv$, or else both $\bv$ and its image are contained in a subsurface $R$ homeomorphic to $\Sigma_{k,l+1}^2$.  There exists a curve $\bu$ that does not intersect $R$.}
\label{minimal_curve_graph}
\end{figure}
If $s \in S$ then $\bv, s(\bv)$ are both contained in a subsurface $R \cong \Sigma_{k,l+1}^2$. As $(k,l)$-vertices are small, there exists a minimal vertex $u$ of $\vC_X(\Sigma)$ that spans an edge with both the vertex $v$ and the vertex $s \cdot v$ corresponding to $s(\bv)$.  It follows that $ v \sim s \cdot v$.  This satisfies the second condition of Lemma \ref{Put} and so the subgraph $\vM(k,l)$ is connected.  It follows that $\phi(v)$ is a vertex of $\vM(k,l)$, completing the proof.
\end{proof}

We can now finally prove that each vertex type determines a characteristic subset of vertices in the linear subcomplex $\vC_X(\Sigma)$.

\begin{prop}\label{vertex_prop}
Let $\vC_X(\Sigma)$ be a linear subcomplex of $\vC_\vS(\Sigma)$ where every minimal vertex is small.  Let $v$ be a vertex of $\vC_X(\Sigma)$. If $v$ is a $(k,l)$-vertex then $\phi(v)$ is a $(k,l)$-vertex for all $\phi \in \Aut \vC_X(\Sigma)$.
\end{prop}

\begin{proof}
Let $v$ be a $(k,l)$-vertex of $\vC_X(\Sigma)$ corresponding to the curve $\bv$ and let $\phi$ be an automorphism of $\vC_X(\Sigma)$ as in the statement of the proposition. Suppose the vertex $\phi(v)$ corresponds to the curve $\bc$. We need to show that $\bc$ is a $(k,l)$-curve.

Since $\vC_X(\Sigma)$ is connected, there exists a maximal linear simplex $\sigma$ containing $v$. Suppose one of the extreme vertices of $\sigma$ is a $(\tilde k, \tilde l)$-vertex $u$.  From Proposition \ref{MinCurve} we have that $\phi(u)$ is an extreme vertices of $\phi(\sigma)$ and is also a $(\tilde k, \tilde l)$-vertex.  If there are $N$ vertices between $u$ and $v$ in the labeling of $\sigma$ then there are $N$ vertices between $\phi(u)$ and $\phi(v)$ in the labeling of $\phi(\sigma)$.  Finally, from Lemma \ref{Incr}, if there are $K$ genus increments between $u$ and $v$ in $\sigma$ then there are $K$ genus increments between $\phi(u)$ and $\phi(v)$ in $\phi(\sigma)$.

Without loss of generality we can assume that $k = \tilde k + K$ and $l = \tilde L + N - K$ and so it follows that $\bc$ is a $(k,l)$-curve.
\end{proof}

Note that in order to prove that vertex types determine characteristic subsets for a surface $\Sigma = \Sigma_{g,n}$ where $g=0$ (or $n=0$) we need only define maximal linear simplices. Indeed, all minimal vertices are of the same vertex type and all increments are genus increments (or no increments are genus increments).

\subsection{Sharing pairs}\label{SP}

In Section \ref{vertex_types} we discussed linear subcomplexes of the separating curve complex $\vC_\vS(\Sigma)$.  The purpose of this section is to show that certain intersection data is characteristic to these subcomplexes.  We will generalise the notion of \emph{sharing pairs} defined by Brendle-Margalit \cite[Section 3]{BM17} into two flavours.  In each case, we say a pair of $(k,l)$-curves $\ba,\bb$ \emph{share} a curve $\bc$. If $\bc$ is $(k-1,l)$-curve we call $\ba,\bb$ a \emph{genus sharing pair}. If $\bc$ is $(k,l-1)$-curve we call $\ba,\bb$ a \emph{puncture sharing pair}.  We use these definitions to complete Step 2 of the strategy outlined at the beginning of Section \ref{section_subcomplexes}.  More precisely, we show that an isomorphism $\Mode(\Sigma) \to \Aut \vC_X(\Sigma)$ implies an isomorphism $\Mode(\Sigma) \to \Aut \vC_Y(\Sigma)$, where $\vC_Y(\Sigma)$ is a particular subcomplex of $\vC_X(\Sigma)$ and both are linear subcomplexes of $\vC_\vS(\Sigma)$.

Before we give the definition of \emph{sharing pairs} we introduce \emph{arcs} to facilitate the discussion. Let $R$ be a surface with boundary. In our setting, an \emph{arc} in $R$ is a continuous image of the interval whose endpoints map to the boundary of $R$. Let $\vC_X(\Sigma)$ be a linear subcomplex of $\vC_\vS(\Sigma)$ and let $z$ be a vertex of $\vC_X(\Sigma)$ corresponding to a curve with an associated region $R$.  Let $SA(R)$ be the set whose elements are the, possibly empty, sets of arcs in $R$. We can define a projection map
\[
\pi_z : \vC_{X}(\Sigma) \rightarrow SA(R).
\]
Note that we may also define a map from $\vC_X(\Sigma)$ to $SA(\Sigma \sm R)$.

If $v$ is a vertex of $\vC_X(\Sigma)$ that shares an edge with $z$ then $\pi_z(v) = \emptyset$. If $v$ and $z$ do not share an edge then $v$ corresponds to a curve whose intersection with $R$ is a nonempty collection of disjoint arcs, that is,
\begin{center}
$v$ and $z$ fail to span an edge in $\vC_X(\Sigma)$ $\iff$ $\pi_z(v) \in SA(R) \sm \emptyset$.
\end{center}
For a vertex $v \in \vC_X(\Sigma)$, if the projection $\pi_z(v)$ is a set arcs that belong to the same free isotopy class then it makes sense to think of $\pi_z(v)$ as a single arc.  We call an arc $\alpha$ \emph{non-separating} if $R \setminus \alpha$ is a single connected subsurface, otherwise we call it \emph{separating}. As we can see from Figure \ref{Handle}, it is possible for a vertex $v \in \vC_X(\Sigma)$ to project to a non-separating arc $\pi_z(v) \in SA(R)$.
\begin{figure}[t]
\centering
\includegraphics[scale=0.13]{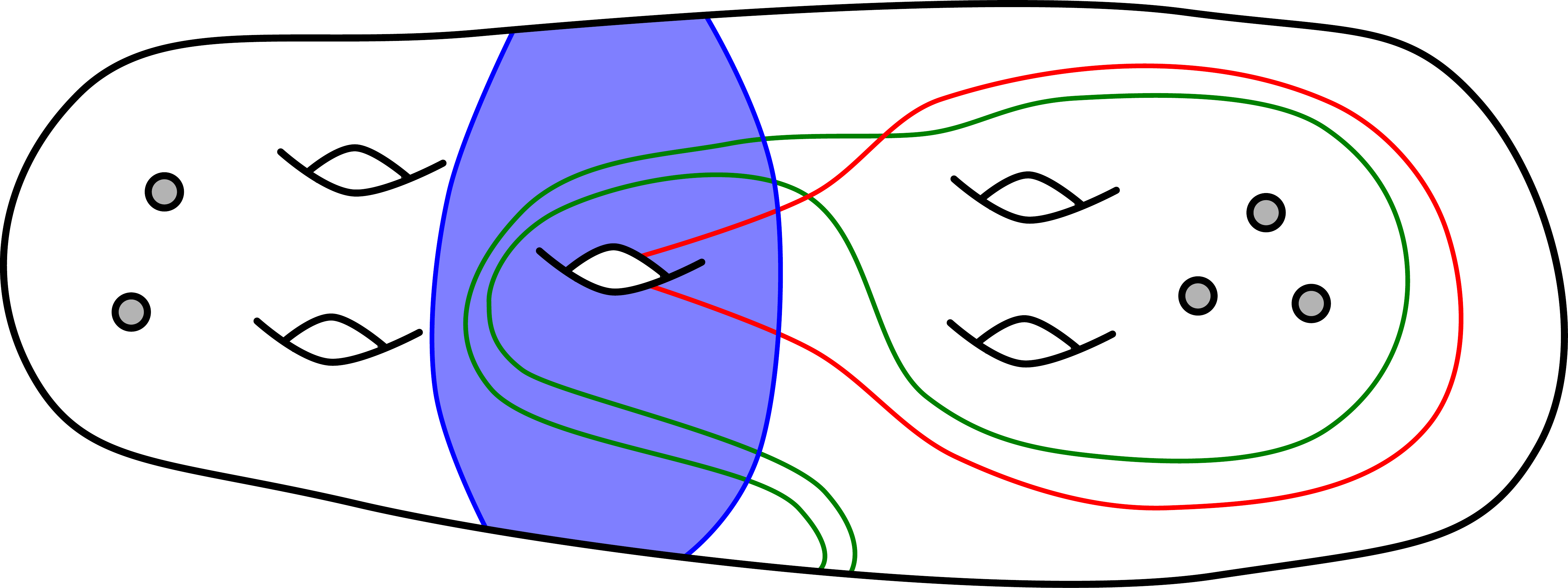}
\caption{Two separating curves such that the union of their projections define the shaded torus with two boundary components, $\Sigma_{1,0}^2$.}
\label{Handle}
\end{figure}

The following definitions, and Lemma \ref{TripleThreat}, are used in the subsequent discussion of \emph{genus sharing pairs}.  We assume that $\vC_X(\Sigma)$ is a linear subcomplex of $\vC_\vS(\Sigma)$.

\subsection*{Unlinked projections and handle pairs}
Let $z \in \vC_X(\Sigma)$ corresponding to $\partial R$ for some region $R$.  Two vertices $u, v$ of $\vC_X(\Sigma)$ are said to have \emph{unlinked projections} if there exists a connected segment of $\partial R$ intersecting a component of $\pi_z(u)$ twice but not intersecting $\pi_z(v)$.  The vertices $u, v$ form a \emph{handle pair} for $R$ if $\pi_z(u)$ and $\pi_z(v)$ are distinct non-separating arcs of $R$ with representatives that lie on some subsurface $Q \subset R$ such that $Q \cong \Sigma_{1,0}^2$, see Figure \ref{Handle}.

Recall that for a vertex $v$ in $\vC_X(\Sigma)$ we say that vertices $u,w$ lie on the same side of $v$ if $u,w \in \Lk (v)$ and there exists another vertex in $\Lk (v)$ which does not span an edge with either $u$ or $w$. If $v$ is a $(k,l)$-vertex then we say that a vertex lies on a \emph{small side} of $v$ if it does not lie on the same side as a $(k+1,l)$-vertex.
The following result is analagous to a result of Brendle-Margalit in the closed case \cite[Lemma 3.2]{BM17}.

\begin{lem}\label{TripleThreat}
Let $\vC_X(\Sigma)$ be a linear subcomplex of $\vC_\vS(\Sigma)$ where every minimal vertex is small.  Let $\phi \in \Aut \vC_X(\Sigma)$. Suppose $\vC_X(\Sigma)$ contains $(k,l)$ and $(k+1,l)$-vertices and let $z$ be a $(k+1,l)$-vertex.  Let $u$ and $v$ be two vertices of $\vC_X(\Sigma)$ such that $\pi_z(u)$ and $\pi_z(v)$ are distinct, non-separating arcs.
\begin{enumerate}
\item The projection $\pi_{\phi(z)}(\phi(u))$ is a non-separating arc;
\item If $\pi_z(u)$ and $\pi_z(v)$ are unlinked non-separating arcs then $\pi_{\phi(z)}(\phi(u))$ and $\pi_{\phi(z)}(\phi(v))$ are unlinked non-separating arcs.
\item If $\pi_z(u)$ and $\pi_z(v)$ are a handle pair then $\pi_{\phi(z)}(\phi(u))$ and $\pi_{\phi(z)}(\phi(v))$ are a handle pair.
\end{enumerate}
\end{lem}

\begin{proof}
Let $R$ be a region of genus $k+1$ with $l$ punctures such that $z$ corresponds to $\partial R$. For the first statement we claim that $\pi_z(u)$ is a non-separating arc if and only if there is more than one $(k,l)$-vertex in $\Lk(u)$ that lies on the small side of $z$.  To prove the forward direction we assume that $\pi_z(u)$ is a non-separating arc.  It follows then that $R \setminus \pi_z(u) \cong \Sigma_{k,l}^2$. As there are infinitely many $(k,l)$-curves in $\Sigma_{k,l}^2$, the implication is clear.

We deal with the other direction of the claim in two cases; either $\pi_z(u)$ contains the homotopy class of a separating arc or it contains more than one homotopy class of non-separating arcs. Suppose we are in the first case. If we cut $R$ by a separating arc it results in two surfaces $R_1$ and $R_2$.  It must be that $R_1 \cong \Sigma_{k_1,l_1}^1$ and $R_2 \cong \Sigma_{k_2,l_2}^1$, with $k_2 \ge k_1$, $k_1 + k_2 = k+1$, and $l_1 + l_2 = l$. If $w$ is a vertex in $\Lk(u)$ that lies on the small side of $z$ then it must correspond to a curve contained in either $R_1$ or $R_2$.  If $w$ is a $(k,l)$-vertex then we have that $k_1 = 1$ and $l_1 =0$.  It follows that $w$ is unique, a contradiction.

In the second case, suppose we cut along two distinct and disjoint non-separating arcs in $R$.  Either we obtain a surface of genus $k$ and $l$ punctures or we obtain one or two surfaces of genus less than $k$.  Therefore, either there exists a single $(k,l)$-vertex adjacent to $u$ on the small side of $z$ or there are none. This completes the proof of the first statement.

To prove the second statement let $u$ and $v$ be adjacent vertices such that $\pi_z(u)$ and $\pi_z(v)$ are unlinked non-separating arcs.  These arcs are distinct if and only if there exists a $(k,l)$-vertex of $\vC_X(\Sigma)$ on the small side of $z$ that is adjacent to $u$ but not $v$. To prove the statement then we claim that the arcs $\pi_z(u)$ and $\pi_z(v)$ are linked if and only if there exists a $(k,l)$-vertex $w$ in $\vC_X(\Sigma)$ that lies on the small side of $z$ and is adjacent to both $u$ and $v$.

If we cut $R$ along disjoint representatives of $\pi_z(u)$ and $\pi(v)$ then we either obtain a surface of genus $k$ and $l$ punctures or we obtain one or two surfaces of genus less than $k$, depending on whether $\pi_z(u)$ and $\pi_z(v)$ are linked or unlinked. The claim follows similarly to the proof of the first statement.

For the final statement we note that two non-separating arcs form a handle pair if and only if they are linked. This completes the proof.
\end{proof}

\subsection*{Genus sharing pairs} We say that two $(k,l)$-vertices form a $(k,l)$-\emph{genus sharing pair} if they correspond to curves with geometric intersection number two and, of the four surfaces obtained by cutting $\Sigma$ along the curves, one is homeomorphic to $\Sigma_{k-1,l}^1$ and two are homeomorphic to $\Sigma_{1,0}^1$.
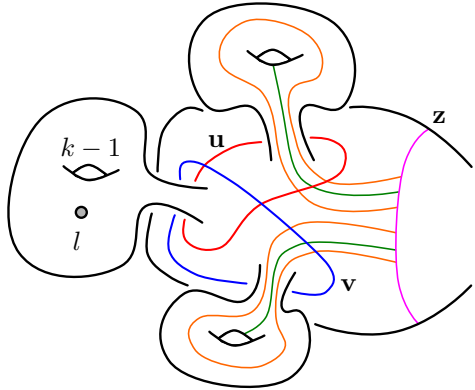
\begin{figure}[t]
\centering
\labellist \hair 1pt
	\pinlabel {$k-1$} at 240 700
	\pinlabel {$l$} at 200 430
	\pinlabel {$\bu$} at 600 720
	\pinlabel {$\bv$} at 980 300
	\pinlabel {$\bz$} at 1240 790
    \endlabellist
\includegraphics[scale=0.13]{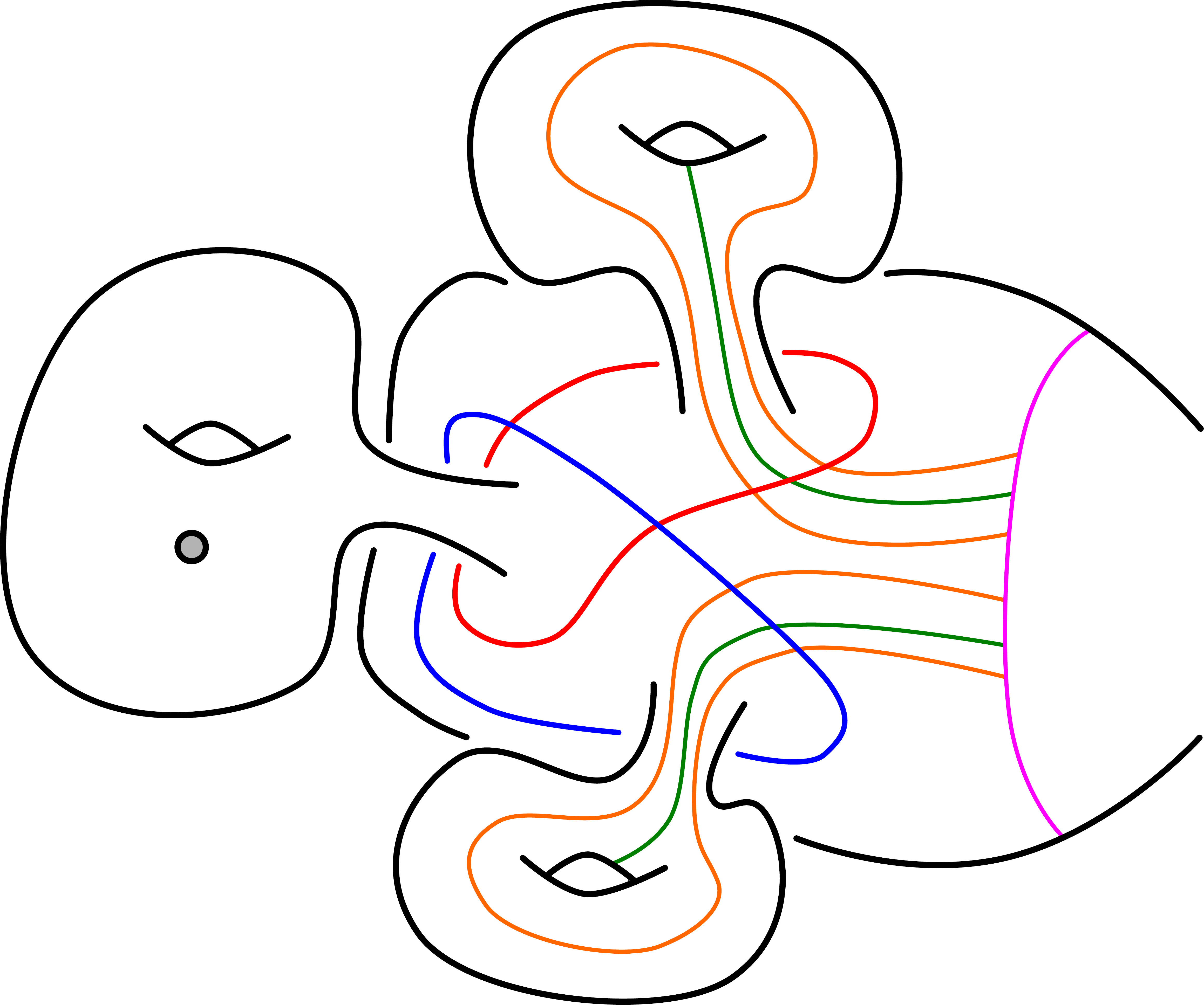}
\caption{A $(k,l)$-genus sharing pair $u,v$ corresponds to the red and blue curves. The vertex $z$ corresponds to light blue curve. The arcs $\pi_z(x_1)$,$\pi_z(y_1)$,$\pi_z(x_2)$ and $\pi_z(y_2)$ are shown in green and orange.}
\label{GSP}
\end{figure}

If two vertices that form a genus sharing pair correspond to the curves $\ba,\bb$ we say that $\ba,\bb$ \emph{share} the $(k-1,l)$-curve $\bc$, where $\bc$ is isotopic to the boundary curve of the region homeomorphic to $\Sigma_{k-1,l}^1$.

\begin{lem}\label{genus_sp}
Let $\vC_X(\Sigma)$ be a linear subcomplex of $\vC_\vS(\Sigma)$ where every minimal vertex is small.  Suppose $\vC_X(\Sigma)$ contains $(k_1,l_1)$- and $(k_2,l_2)$-vertices.  Let vertices $u,v \in \vC_X(\Sigma)$ form a $(k,l)$-genus sharing pair.  If $g \ge k + k_1 + k_2 + 1$ and $n \ge l + l_1 + l_2$ then $\phi(u),\phi(v)$ form a $(k,l)$-genus sharing pair for all $\phi \in \Aut \vC_X(\Sigma)$.
\end{lem}

\begin{proof}
We will show that two vertices $u,v$ form a $(k,l)$-genus sharing pair if and only if there are two $(k_1,l_1)$-vertices $x_1$ and $y_1$, two $(k_2,l_2)$-vertices $x_2$ and $y_2$, and a $(k+1,l)$-vertex $z$ that satisfy the following properties.
\begin{enumerate}
\item Both $u$ and $v$ lie on the small side of $z$;
\item both $x_1$ and $y_1$ are adjacent to $u$, $x_2$ and $y_2$, but not $v$;
\item both $x_2$ and $y_2$ are adjacent to $v$, $x_1$ and $y_1$ but not $u$;
\item both pairs $\pi_z(x_1),\pi_z(y_1)$ and $\pi_z(x_2),\pi_z(y_2)$ are distinct handle pairs; and
\item if $\alpha_1 \in \{ \pi_z(x_1), \pi_z(y_1) \}$ and $\alpha_2 \in \{ \pi_z(x_2), \pi_z(y_2) \}$ then $\alpha_1$ and $\alpha_2$ are unlinked.
\end{enumerate}
The result then follows from Lemmas \ref{vertex_prop} and \ref{TripleThreat}.

Suppose the vertices $u,v$ form a $(k,l)$-genus sharing pair and correspond to the curves $\bu,\bv$. Up to homeomorphism there is a unique confuguration for the curves $\bu,\bv$ shown in Figure \ref{GSP}.  The curves $\bu$ and $\bv$ separate $\Sigma$ into four regions which are homeomorphic to $\Sigma_{k-1,l}^1$, $\Sigma_{1,0}^1$, $\Sigma_{1,0}^1$ and $\Sigma_{g-k-1,n-l}^1$. Take $R$ to be the complement of this final region in $\Sigma$ and let $z$ be the vertex corresponding to $\partial R$. We then define $x_1, y_1, x_2,$ and $y_2$ to be $(k_1, l_1)$- and $(k_2,l_2)$-vertices corresponding to the projected arcs shown in Figure \ref{GSP}. The chosen vertices satisfy the five conditions above.

Now suppose we have vertices $u,v,x_1,y_1,x_2,y_2,$ and $z$ satisfying the above conditions.  Note that conditions (1), (2), and (3) are met whenever $g$ and $n$ are as in the statement of the lemma.  By the fourth condition the arcs $\pi_z(x_1)$ and $\pi_z(y_1)$ are contained in some region $Q_1 \cong \Sigma_{1,0}^2$. Denote the two boundary components of $Q_1$ by $\bz$ and $\bu$. The vertex $z$ must correspond to the curve $\bz$, and the arcs $\pi_z (x_1)$ and $\pi_z (y_1)$ have endpoints on $\bz$. We want to show that the vertex $u$ corresponds to $\bu$.

The surface obtained by cutting along $Q_1$ by $\pi_z(x_1)$ is homeomorphic to a pair of pants $P$. If we then cut $P$ along $\pi_z(y_1)$ the resulting surface is an annulus. It follows that $\pi_z(x_1)$ and $\pi_z(y_1)$ fill $Q_1$. From the second condition we have that $u$ corresponds to a curve that is disjoint from $Q_x$. Since $Q_x$ is of genus one it must be that $u$ corresponds to $\bu$. By symmetry, the vertex $v$ corresponds to $\bv$ the boundary component not isotopic to $\partial R$ of the equivalent region $Q_2$.
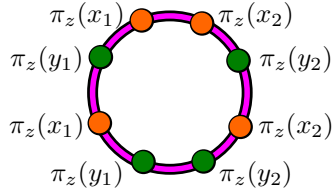
\begin{figure}[t]
\centering
\labellist \hair 1pt
	\pinlabel {$\pi_z(x_1)$} at 0 200
	\pinlabel {$\pi_z(y_1)$} at -50 145
	\pinlabel {$\pi_z(x_1)$} at -50 60
	\pinlabel {$\pi_z(y_1)$} at 0 0
	\pinlabel {$\pi_z(x_2)$} at 210 200
	\pinlabel {$\pi_z(y_2)$} at 260 145
	\pinlabel {$\pi_z(x_2)$} at 260 60
	\pinlabel {$\pi_z(y_2)$} at 210 0
    \endlabellist
\includegraphics[scale=0.3]{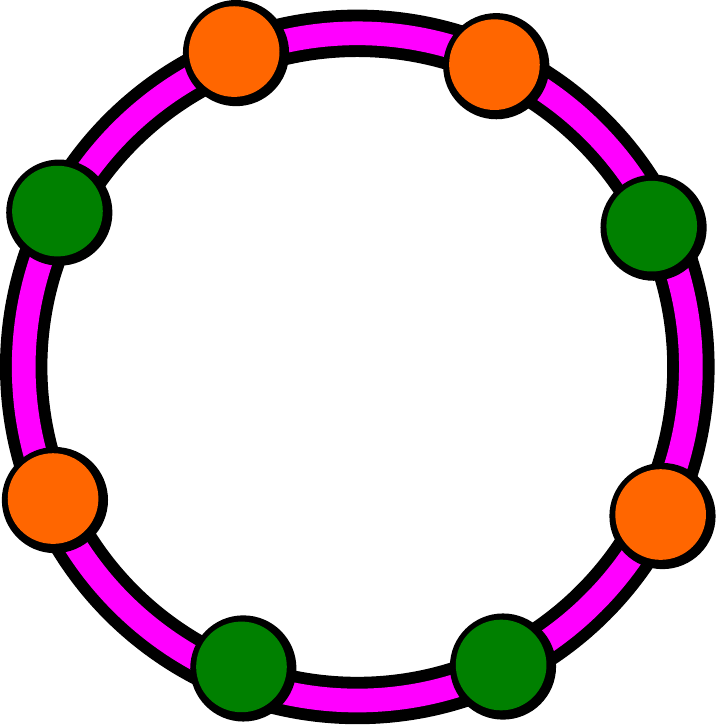}
\caption{The curve $\bz$ intersects $\pi_z(x_1)$,$\pi_z(y_1)$,$\pi_z(x_2)$ and $\pi_z(y_2)$ as shown.}
\label{GR}
\end{figure}

From the fifth condition, we can view $\bz$ as the circle in Figure \ref{GR}. There exist segments $\gamma_1$ and $\gamma_2$ of $\bz$, with $\gamma_x \cup \gamma_y = \bz$, such that the arcs $\pi_z(x_1)$ and $\pi_z(y_1)$ have endpoints in $\gamma_1$ and the arcs $\pi_z(x_2)$ and $\pi_z(y_2)$ have endpoints in $\gamma_2$.

It follows that the intersection of $\pi_z(x_2)$ and $\pi_z(y_2)$ with $Q_1$ is a set of four freely isotopic arcs. Since $Q_2$ is a regular neighbourhood of the arcs $\pi_z(x_2)$ and $\pi_z(y_2)$ we have that the intersection of $Q_1$ and $Q_2$ is an annulus whose boundary components are isotopic to $\bz$.  The curves $\bu$ and $\bv$ must therefore have essential intersection two.

If two separating simple closed curves intersect in two points then they divide $\Sigma$ into four regions, one of which must contain $\bz$. It follows that one of these regions is of genus $k-1$ and has $l$ punctures. Thus, $u,v$ form a genus sharing pair.
\end{proof}

\subsection*{Puncture sharing pairs} We say that two $(k,l)$-vertices form a $(k,l)$-\emph{puncture sharing pair} if they correspond to curves with geometric intersection number two and, of the four surfaces obtained by cutting $\Sigma$ along the curves, one is homeomorphic to $\Sigma_{k,l-1}^1$ and two are homeomorphic to $\Sigma_{0,1}^1$.

If two vertices that form a puncture sharing pair correspond to the curves $\ba,\bb$ we say that $\ba,\bb$ \emph{share} the curve $\bc$, where $\bc$ is isotopic to the boundary curve of the region homeomorphic to $\Sigma_{k,l-1}^1$.

\begin{lem}\label{puncture_sp}
Let $\vC_X(\Sigma)$ be a linear subcomplex of $\vC_\vS(\Sigma)$ where every minimal vertex is small.  Suppose $\vC_X(\Sigma)$ contains $(k_1,l_1)$- and $(k_2,l_2)$-vertices.  Let vertices $u,v \in \vC_X(\Sigma)$ form a $(k,l)$-puncture sharing pair and let $g \ge k + k_1 + k_2$. If $n \ge l + l_1 + l_2 + 1$, or $l_1,l_2 > 0$ and $n \ge l + l_1 + l_2 - 1$ then $\phi(u),\phi(v)$ form a $(k,l)$-puncture sharing pair for all $\phi \in \Aut \vC_X(\Sigma)$.
\end{lem}

\begin{proof}
We will show that two vertices $u,v$ form a $(k,l)$-puncture sharing pair if and only if there is a $(k_1,l_1)$-vertex $x_1$, a $(k_2,l_2)$-vertex $x_2$, and a $(k,l+1)$-vertex $z$ that satisfy the following properties.
\begin{enumerate}
\item Both $u$ and $v$ lie on the small side of $z$;
\item the vertex $x_1$ is adjacent to $u$ and $x_2$ but not $v$; and
\item the vertex $x_2$ is adjacent to $v$ and $x_1$ but not $u$.
\end{enumerate}
The result then follows from Lemma \ref{vertex_prop}.

Suppose the vertices $u,v$ form a $(k,l)$-puncture sharing pair and correspond to the curves $\bu,\bv$. Up to homeomorphism there is a unique configuration for the curves $\bu,\bv$ shown in Figure \ref{PSP}.
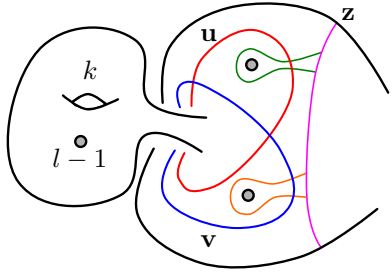
\begin{figure}[t]
\centering
\labellist \hair 1pt
	\pinlabel {$k$} at 240 580
	\pinlabel {$l-1$} at 210 320
	\pinlabel {$\bu$} at 580 680
	\pinlabel {$\bv$} at 580 80
	\pinlabel {$\bz$} at 980 740
    \endlabellist
\includegraphics[scale=0.13]{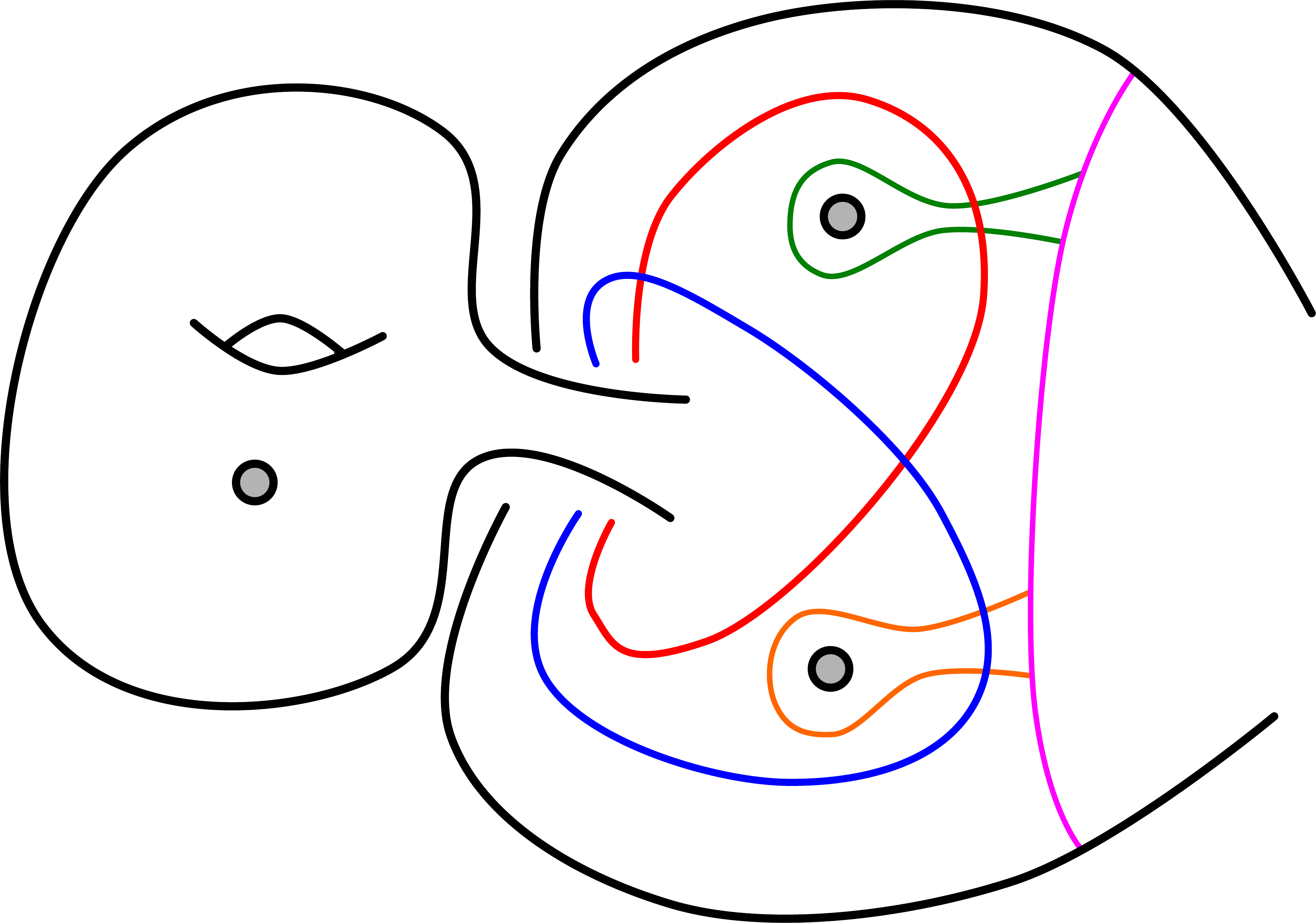}
\caption{A $(k,l)$-puncture sharing pair $u,v$ corresponds to the red and blue curves. The vertex $z$ corresponds to light blue curve. The arcs $\pi_z(x_1)$ and $\pi_z(x_2)$ are shown in green and orange.}
\label{PSP}
\end{figure}
The curves $\bu$ and $\bv$ separate $\Sigma$ into four regions which are homeomorphic to $\Sigma_{k,l-1}^1$, $\Sigma_{0,1}^1$, $\Sigma_{0,1}^1$ and $\Sigma_{g-k,n-l-1}^1$. Take $R$ to be the complement of this final region in $\Sigma$ and let $z$ be the vertex corresponding to $\partial R$. We then define $x_1$ to be the $(k_1,l_1)$-vertex and $x_2$ to be the $(k_2,l_2)$-vertex corresponding to the projected arcs shown in Figure \ref{PSP}. The chosen vertices satisfy the five conditions above.

Now suppose we have vertices $u,v,x,y$ and $z$ satisfying the above conditions.  Note that the three conditions can always be met when $g$ and $n$ satisfy the bounds given in the statement of the lemma.  By the first and second conditions the arc $\pi_z(x_1)$ is contained in some $Q_1$ homeomorphic to an annulus with a single puncture. Denote the two boundary components of $Q_1$ by $\bz$ and $\bu$. The vertex $z$ must correspond to $\bz$ and the arc $\pi_z (x_1)$ has endpoints on $\bz$. We want to show that the vertex $u$ corresponds to $\bu$.

When we cut $Q_1$ along the arc $\pi_z(x_1)$ we get two surfaces; an annulus and a disc with one puncture. The boundary of this annulus is isotopic to $\bu$, a $(k,l)$-curve that is contained in the associated region of $\bz$ with genus $k$.  It follows that $u$ corresponds to $\bu$. By symmetry, the vertex $v$ must correspond to $\bv$, the boundary component not isotopic to $\bz$ of the equivalent region $Q_2$.

From the fourth condition the curve $\bz$ takes the form of the circle in Figure \ref{puncture_bound}(i).  There exist segments $\gamma_1$ and $\gamma_2$ of $\bz$, with $\gamma_1 \cup \gamma_2 = \bz$, such that the arcs $\pi_z (x_1)$ and $\pi_z (x_2)$ have endpoints in $\gamma_1$ and $\gamma_2$ respectively.
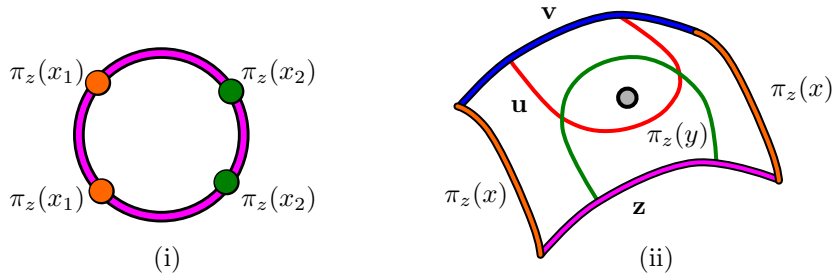
\begin{figure}[h]
\centering
\labellist \hair 1pt
	\pinlabel {(i)} at 110 0
	\pinlabel {(ii)} at 680 0
	\pinlabel {$\pi_z(x_1)$} at -30 220
	\pinlabel {$\pi_z(x_1)$} at -30 70
	\pinlabel {$\pi_z(x_2)$} at 240 220
	\pinlabel {$\pi_z(x_2)$} at 240 70
	\pinlabel {$\bv$} at 555 290
	\pinlabel {$\bu$} at 520 180
	\pinlabel {$\pi_z(y)$} at 705 145
	\pinlabel {$\pi_z(x)$} at 850 200
	\pinlabel {$\pi_z(x)$} at 470 75
	\pinlabel {$\bz$} at 660 60
    \endlabellist
\includegraphics[scale=0.32]{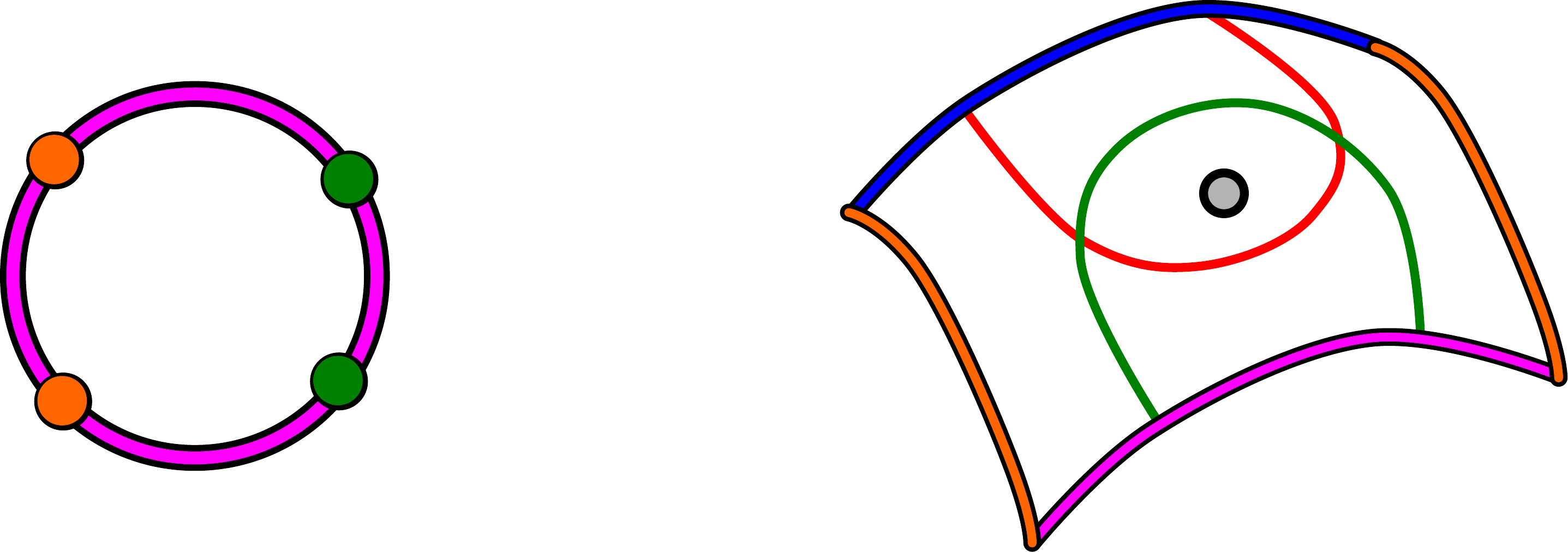}
\caption{The curve $\bz$ intersects $\pi_z(x_1)$ and $\pi_z(x_2)$}
\label{puncture_bound}
\end{figure}

It follows that the intersection of the arc representing $\pi_z (x_2)$ with $Q_1$ is a set of two freely isotopic arcs. If we cut along one of these arcs then, since $\bu$ and $\bv$ must intersect, they take the form shown in Figure \ref{puncture_bound}(ii) where they intersect exactly twice.  If two separating simple closed curves intersect in two points then they divide $\Sigma$ into four regions, one of which must contain $\bz$.  It follows that one of these regions is of genus $k$ and has $l-1$ punctures. Thus, $u,v$ form a genus sharing pair.
\end{proof}

Note that the bounds on $g$ and $n$ found in Lemma \ref{genus_sp} form part of the inequalities $(5)$ and $(6')$ defining small vertices from Section \ref{vertex_types}. Similarly the bounds on $g$ and $n$ from Lemma \ref{puncture_sp} are found in the inequalities $(5')$ and $(6)$. The requirement in Theorem \ref{etas} that all minimal vertices are small therefore allows us to apply Lemmas \ref{genus_sp} and \ref{puncture_sp} to minimal vertices of $\vC_{\vS(A)}(\Sigma)$.

Let $\vC_X(\Sigma)$ and $\vC_Y(\Sigma)$ be linear subcomplexes of $\vC_\vS(\Sigma)$ such that $\vC_Y(\Sigma)$ is a subcomplex of $\vC_X(\Sigma)$.  We will now use the two types of sharing pairs to extend an automorphisms of $\vC_Y(\Sigma)$ to automorphisms of $\vC_X(\Sigma)$.  We do this by introducing \emph{graphs of sharing pairs} and showing that it consists of infinitely many connected components, each corresponding to a unique isotopy class of curves.

If $u_1,v_1$ and $u_2,v_2$ are $(k,l)$-genus sharing pairs that correspond to curves that share the same $(k-1,l)$-curve then we say that $u_1,v_1$ and $u_2,v_2$ are \emph{similar}.  In the same way, we may define \emph{similar} $(k,l)$-puncture sharing pairs to be those that correspond to pairs of curves sharing the same $(k,l-1)$-curve.

\subsection*{Graphs of sharing pairs}
Given a linear subcomplex $\vC_X(\Sigma)$ of $\vC_\vS(\Sigma)$ we construct a graph $\vSP$ with vertices corresponding to all $(k,l)$-sharing pairs of the same type.  Two vertices share an edge in $\vSP$ if they correspond to sharing pairs $u,v$ and $v,w$, such that $u,w$ is also a sharing pair and all three pairs are similar.  Note that this definition holds for both genus sharing pairs and puncture sharing pairs.

From Proposition \ref{vertex_prop}, Lemma \ref{genus_sp}, and Lemma \ref{puncture_sp} one can show that if $u,v$ and $v,w$ share an edge in $\vSP$ then $\phi(u), \phi(v)$ and $\phi(v), \phi(w)$ share an edge, for all $\phi \in \Aut \vC_X(\Sigma)$

It is clear that if two sharing pairs are connected in $\vSP$ then they are similar sharing pairs. This implies that the graph $\vSP$ is made up of various disconnected components.  We will write $\vSP(\bc)$ for the components relating to sharing pairs that correspond to pairs of curves that share the same curve $\bc$.

We now show that $\vSP(\bc)$ is a single connected component of $\vSP$.  We will once again make use of Lemma \ref{Put}.

\begin{lem}\label{ShTr}
Suppose $\mathcal{SP}$ corresponds to $(k,l)$-genus sharing pairs and $g \ge k+3$, or $\mathcal{SP}$ corresponds to $(k,l)$-puncture sharing pairs and $n \ge l+2$.  Then the subgraph $\mathcal{SP}(\bc)$ is a single connected component of $\mathcal{SP}$ for any curve $\bc$.
\end{lem}

\begin{proof}
Let $\ba,\bb$ be $(k,l)$-curves that share the curve $\bc$. Let $R$ be the associated region of $\bc$ that does not contain $\ba$ or $\bb$. Let $\Mod (\Sigma, R)$ be the subgroup of $\Mod (\Sigma)$ that fixes the subsurface $R$.  Every vertex in $\vSP(\bc)$ corresponds to curves $f(\ba), f(\bb)$, for some $f \in \Mod (\Sigma, R)$. This satisfies the first condition in Lemma \ref{Put} with respect to the simplicial complex $\vSP(\bc)$. It remains to show that the second condition is satisfied. This will be done in two cases; the first case deals with $(k,l)$-genus sharing pairs and the second deals with $(k,l)$-puncture sharing pairs.

Suppose the vertices of $\vSP$ correspond to $(k,l)$-genus sharing pairs. The groups $\Mod(\Sigma, R)$ and $\Mod(\Sigma_{g-k+1,n-l}^1)$ are isomorphic. It follows that there exists a finite generating set $S$ for $\Mod(\Sigma, R)$ consisting of Dehn twists about non-separating curves and half twists about $(0,2)$-curves, see Figure \ref{genus_putman}(i).
\begin{figure}[t]
\centering
\labellist \hair 1pt
	\pinlabel {(i)} at 1000 0
	\pinlabel {(ii)} at 2700 0
	\pinlabel {$k-1$} at 240 700
	\pinlabel {$l$} at 200 430
	\pinlabel {$\ba$} at 1000 730
	\pinlabel {$\bb$} at 990 300
	\pinlabel {$k-1$} at 1820 700
	\pinlabel {$l$} at 1780 430
	\pinlabel {$\bd$} at 2910 660
	\pinlabel {$T(\bb)$} at 2710 280
    \endlabellist
\includegraphics[scale=0.11]{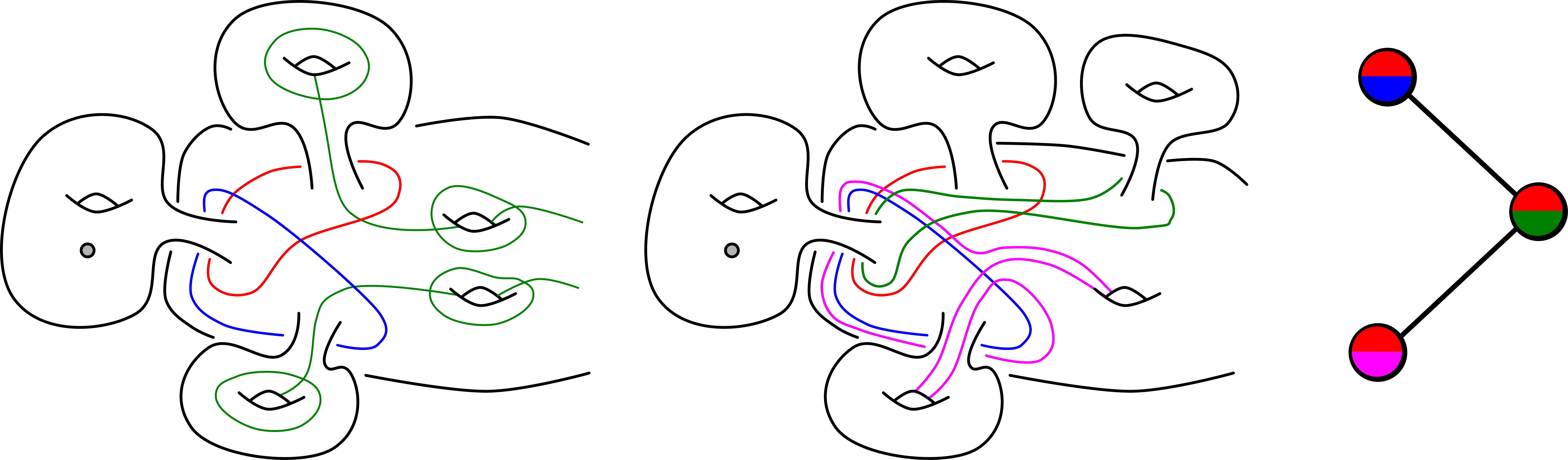}
\caption{(i) Generating twists and half twists with respect to a genus sharing pair corresponding to curves $\ba, \bb$. (ii) The vertex of $\vSP(\bc)$ corresponding to the sharing pair $\ba, \bd$ spans an edge with both $\ba, \bb$ and $\ba, T(\bb)$.}
\label{genus_putman}
\end{figure}

We choose $S$ so that one non-separating curve intersects $\ba$, one non-separating curve intersects $\bb$, and all other curves are disjoint from both $\ba$ and $\bb$.  By symmetry it is enough to consider the single case where $T$ is a Dehn twist about a non-separating curve intersecting $\bb$ and disjoint from $\ba$. We have that $T(\ba) = \ba$ and $\ba,T(\bb)$ share the curve $\bc$.  It remains to show that the vertices corresponding to $\ba,\bb$ and $\ba,T(\bb)$ are connected in $\vSP(\bc)$.

Given $g \ge k +3$ we can find a curve $\bd$ such that that the vertex corresponding to $\ba,\bd$ is adjacent to the vertices corresponding to $\ba,\bb$ and $\ba,T(\bb)$ in $\vSP$, see Figure \ref{genus_putman}(ii).  By Lemma \ref{Put} the result holds $\vSP(\bc)$ when defined with respect to $(k,l)$-genus sharing pairs.

Now suppose the vertices of $\vSP$ correspond to puncture sharing pairs. The groups $\Mod(\Sigma, R)$ and $\Mod(\Sigma_{g-k,n-l+1)}^1)$ are isomorphic.  Once again, we can find a finite generating set $S$ for $\Mod(\Sigma, R)$ consisting of Dehn twists about non-separating curves and half twists about $(0,2)$-curves, see Figure \ref{puncture_putman}(i).
\begin{figure}[t]
\centering
\labellist \hair 1pt
	\pinlabel {(i)} at 870 -40
	\pinlabel {(ii)} at 2550 -40
	\pinlabel {$k$} at 240 550
	\pinlabel {$l-1$} at 200 280
	\pinlabel {$\ba$} at 820 580
	\pinlabel {$\bb$} at 780 80
	\pinlabel {$k$} at 1880 550
	\pinlabel {$l-1$} at 1840 280
	\pinlabel {$H(\bb)$} at 2590 560
	\pinlabel {$\bd$} at 2760 280
    \endlabellist
\includegraphics[scale=0.11]{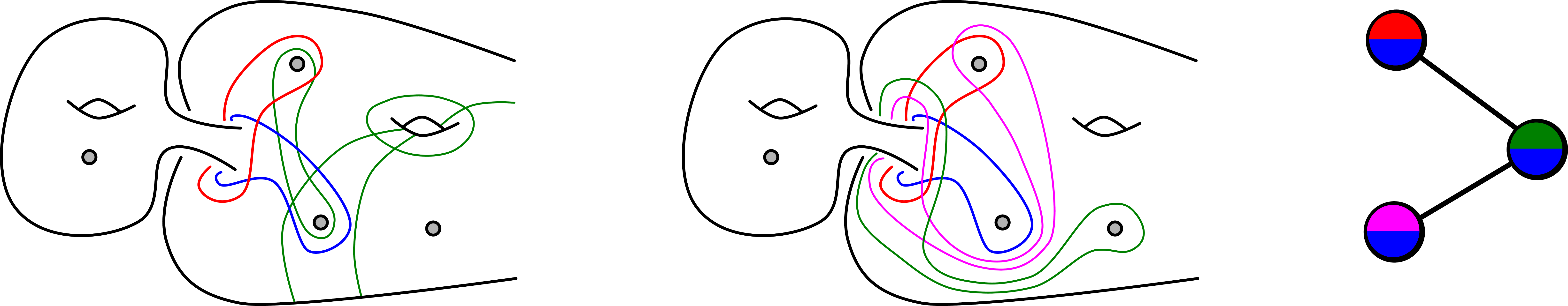}
\caption{(i) Generating twists with respect to a puncture sharing pair corresponding to curves $\ba, \bb$. (ii) The vertex of $\vSP(\bc)$ corresponding to the sharing pair $\bd, \bb$ spans an edge with both $\ba, \bb$ and $H(\bb), \bb$. Note also that $H(\ba) = \bb$.}
\label{puncture_putman}
\end{figure}

Again, we may choose $S$ so that one non-separating curve and one $(0,2)$-curve intersect $\bb$ and one $(0,2)$-curve intersects both $\ba$ and $\bb$, all other curves are disjoint from both $\ba$ and $\bb$.  As before, if $T$ is a Dehn twist about a non-separating curve intersecting $\bb$ and disjoint from $\ba$ then it is clear that $T(\ba),T(\bb)$ share $\bc$. Given $n \ge l +2$ we can find a curve $\bd$ such that the vertex relating to $\ba,\bd$ is adjacent to the vertices relating to $\ba,\bb$ and $T(\ba),T(\bb)$ in $\vSP$. A similar argument follows for the half twist about the $(0,2)$-curve intersecting $\ba$ and not $\bb$.

Finally, for the half twist $H$ about a $(0,2)$-curve intersecting both $\ba$ and $\bb$ it is clear that $H(\ba),H(\bb)$ share $\bc$.  Furthermore, without loss of generality we can assume that $H(\ba) = \bb$. Given $n \ge l+2$ we can find a curve $\bd$ such that the vertex corresponding $\bd,\bb$ is adjacent to both to $\ba,\bb$ and $H(\bb),\bb$, see Figure \ref{puncture_putman}(ii).  By Lemma \ref{Put} the result holds $\vSP(\bc)$ when defined with respect to $(k,l)$-puncture sharing pairs.
\end{proof}

The following proposition will be a key step used repeatedly when proving Theorem \ref{etas}.

\begin{prop}\label{Step}
Let $\vC_X(\Sigma)$ and $\vC_Y(\Sigma)$ be linear subcomplexes of $\vC_\vS(\Sigma)$ such that $\vC_Y(\Sigma)$ is obtained by removing all $(k,l)$-vertices from $\vC_X(\Sigma)$.  Suppose the natural homomorphism
\[
\eta_X : \Mode (\Sigma) \rightarrow \Aut \vC_X(\Sigma)
\]
is an isomorphism.  If automorphisms of $\vC_Y(\Sigma)$ either
\begin{enumerate}
\item preserve $(k+1,l)$-genus sharing pairs and $g \ge (k+1) +3$, or
\item preserve $(k,l+1)$-puncture sharing pairs and $n \ge (l+1) + 2$, 
\end{enumerate}
then the natural homomorphism
\[
\eta_Y : \Mode (\Sigma) \rightarrow \Aut \vC_Y(\Sigma)
\]
is also an isomorphism. 
\end{prop}

\begin{proof}
By Lemma \ref{Injectivity} the map $\eta_Y$ is injective.  It remains to show that it is surjective.  Let $\phi$ be an automorphism of $\Aut \vC_Y(\Sigma)$.  By assumption, either $(k+1,l)$-genus sharing pairs or $(k,l+1)$-puncture sharing pairs are preserved by $\phi$ and by Lemma \ref{ShTr} we have a well defined permutation $\widehat \phi$ of the vertices of $\vC_X(\Sigma)$ such that $\widehat \phi$ restricts to $\phi$ on the vertices of $\vC_Y(\Sigma)$. We will show that $\widehat \phi$ in fact extends to an automorphism of $\vC_X(\Sigma)$.

Suppose vertices $u,v$ of $\vC_X(\Sigma)$ correspond to the curves $\bu$ and $\bv$.  We need to show that the adjacency of $u$ and $v$ in the complex $\vC_X(\Sigma)$ is characteristic in its subcomplex $\vC_Y(\Sigma)$.  If both $u$ and $v$ are vertices of $\vC_Y(\Sigma)$ then this is trivial.  Suppose neither $u$ nor $v$ are vertices of $\vC_Y(\Sigma)$, that is, they are both $(k,l)$-vertices.  Then $u$ and $v$ span an edge in $\vC_X(\Sigma)$ if and only if there are vertices $y_1 \in \vSP(\bu)$ and $y_2 \in \vSP(\bv)$ that span an edge in $\vC_Y(\Sigma)$.

Finally, suppose $v$ is a vertex of $\vC_Y(\Sigma)$ and $u$ is not, that is, $u$ is a $(k,l)$-vertex. The vertices span an edge in $\vC_X(\Sigma)$ if there exists some vertex $y \in \vSP(\bu)$ spanning an edge with, or equal to, $v$ in $\vC_Y(\Sigma)$. Since both $\vC_X(\Sigma)$ and $\vC_Y(\Sigma)$ are connected linear subcomplexes of $\vC_\vS(\Sigma)$ all edges are of this form. We have therefore shown that $\widehat \phi \in \Aut \vC_X(\Sigma)$.

By assumption there exists some $f \in \Mode(\Sigma)$ whose image in $\Aut \vC_X(\Sigma)$ is precisely $\widehat \phi$. Since the restriction of $\widehat \phi$ to the subcomplex $\vC_Y(\Sigma)$ is $\phi$ it follows that the image of $f$ in $\Aut \vC_Y(\Sigma)$ is indeed $\phi$.
\end{proof}

In order to apply Proposition \ref{Step} for $(k,l)$-genus sharing pairs we require that $g \ge k+3$.  Similarly, to apply Proposition \ref{Step} for $(k,l)$-puncture sharing pairs we require $n \ge l+2$.  These conditions are due to Lemma \ref{ShTr}.  Combining these bounds with the bounds from Lemmas \ref{genus_sp} and \ref{puncture_sp} we arrive at the definition of small vertices given in Section \ref{vertex_types}.

\subsection{Navigating between subcomplexes}\label{EndGame}
Recall that $\vC_{\vS(A)}(\Sigma)$ is the subcomplex of $\vC_\vS(\Sigma)$ spanned by vertices that correspond to curves separating regions represented in $A \subset \vR(\Sigma)$.  From this definition we see that $\vC_{\vS(A)}(\Sigma)$ is a linear subcomplex of $\vC_\vS(\Sigma)$.  Indeed, if $\bc_1$ and $\bc_2$ are two curves that separate regions represented in $A$, then every curve separating $\bc_1$ and $\bc_2$ also separate two regions represented in $A$.  As discussed in Section \ref{vertex_types} this implies that $v$ is a minimal vertex of $\vC_{\vS(A)}(\Sigma)$ if and only if it is an extreme vertex of some maximal linear simplex.

Let $v$ be a $(k,l)$-vertex of $\vC_{\vS(A)}(\Sigma)$.  In Section \ref{introduction} we saw that $v$ is small if there exists a $(k_1,l_1)$-vertex and a $(k_2,l_2)$-vertex in $\vC_{\vS(A)}(\Sigma)$ such that;
\begin{align}
\tag{5} g &\ge k + \max\{ k_1 + k_2 , 2\} + 1, \mbox{ and} \\ 
\tag{6} n &\ge l + \max \{ l_1 + l_2 , 1\} + 1.
\end{align}

We wish to prove Theorem \ref{etas} which states that if every minimal vertex of is small then the natural homomorphism
\[
\eta_{\vS(A)} : \Mode(\Sigma) \to \Aut \vC_{\vS(A)}(\Sigma)
\]
is an isomorphism.

\subsection*{Diagrammatically representing orbits}
When $n=0$ the proof of Brendle-Margalit progresses by an inductive argument on $k$, where $(k,0)$-vertices are minimal \cite{BM17}.  Similarly, when $g=0$ the proof of the author uses induction on $l$, where $(0,l)$-vertices are minimal \cite{BraidMeta}.  In effect, these special cases use the fact that each vertex type of $\vC_\vS(\Sigma)$ can be defined by a positive integer.

When $g,n>0$, in general not all minimal curves are of the same vertex type. Furthermore, we require two integers to define each vertex type.  More specifically, every point in $[0,g]\times[0,n]$ with integer coordinates describes a vertex type in $\vC_\vS(\Sigma)$, except for $(0,0)$, $(0,1)$, $(g, n)$, and $(g, n - 1)$.  It will be useful therefore to define
\[
\vV_\vS : = \big ( [0,g]\times[0,n] \big ) \cap \ZZ \times \ZZ \ \sm \ \{(0,0), (0,1), (g, n), (g, n - 1)\}.
\]
Our strategy for proving Theorem \ref{etas} will make use of this notion.  By Proposition \ref{Step} we can remove all $(k,l)$-vertices from $\vC_\vS(\Sigma)$ and, under certain restrictions, the resulting subcomplex $\vC_X(\Sigma)$ will have $\Mode(\Sigma)$ as its group of automorphisms.  Furthermore, the elements of $\vV_\vS \sm \{(k,l)\}$ will correspond to the vertex types of $\vC_X(\Sigma)$.

We continue this process until we reach a subset of $\vV_\vS$ whose elements correspond to the vertex types of $\vC_{\vS(A)}(\Sigma)$.  The proof therefore amounts to verifying that Proposition \ref{Step} can be applied in each instance.  We will see that this is possible due to the assumption that all minimal vertices of $\vC_{\vS(A)}(\Sigma)$ are small.

Note that the correspondence between $\vV_\vS$ and the vertex types of $\vC_\vS(\Sigma)$ is not bijective.  This is because every $(k,l)$-vertex is equal to a $(g-k,n-l)$-vertex.  It follows that the $(k,l)$-vertices correspond to two elements of $\vV_\vS$, unless both $g$ and $n$ are even and $k=g/2$, $l=n/2$.

\subsection*{Diagrammatically representing linear subcomplexes}
As discussed above, we would like to be able to check whether or not we can apply Proposition \ref{Step} to a subcomplex defined by some subset of $\vV_\vS$.  One condition we need to verify is that the subcomplex in question is a linear subcomplex.

To that end, let $(x_1,y_1)$ and $(x_2,y_2)$ be points in $\vV_\vS$.  We write $(x_1,y_1) < (x_2,y_2)$ if
\[
(x_2 - x_1, y_2 - y_1) \in \{ (1,0), (0,1) \}.
\]
Now, for any two points $(x_s,y_s), (x_t, y_t) \in \vV_\vS$ such that $x_s \le x_t$ and $y_s \le y_t$ it is clear that there exists a sequence of points in $\vV_\vS$;
\[
(x_s,y_s) < (x_1,y_1) < (x_t, y_t) < \dots < (x_t, y_t).
\]
Moreover, this sequence forms part of a maximal linear simplex in $\vC_\vS(\Sigma)$ up to the action of $\Mode(\Sigma)$.  It follows that in order to verify that a subset $\vV_X$ of $\vV_\vS$ describes a linear subcomplex $\vC_X(\Sigma)$ of $\vC_\vS(\Sigma)$ we need to show that for any two points $(x_s,y_s), (x_t, y_t) \in \vV_X$ such that $x_s \le x_t$ and $y_s \le y_t$ there exists a sequence as above in $\vV_\vS$.

\begin{proof}[Proof of Theorem \ref{etas}]
Let $(k_1,l_1)$-, $\dots$, $(k_m, l_m)$-vertices be minimal such that $k_1 < k_i$ and $l_m < l_i$ for all $i$.  Our first goal is to apply Proposition \ref{Step} until we arrive at the subcomplex obtained by removing all $(k,l-1)$-vertices, for $k_1 \le k \le g$ and $0 \le l \le l_m$.  Since all minimal vertices are small we have that either $l_m=0$ or $n \ge 3 l_m +1$.  If $l_m \ge 2$ then for a connected linear subcomplex $\vC_X(\Sigma)$ containing $(0,2)$-vertices we apply Lemma \ref{puncture_sp} to see that $(g,l_m)$-puncture sharing pairs are preserved by automorphisms of $\vC_X(\Sigma)$.  In fact, we have that all $(k,l)$-puncture sharing pairs are preserved by automorphisms for $k$ and $l$ as above.  We may apply Proposition \ref{Step} until we arrive at the desired subcomplex.  We use the discussion preceding this proof to verify that all subcomplexes we pass through are connected linear subcomplexes, see Figure \ref{diagram1}.
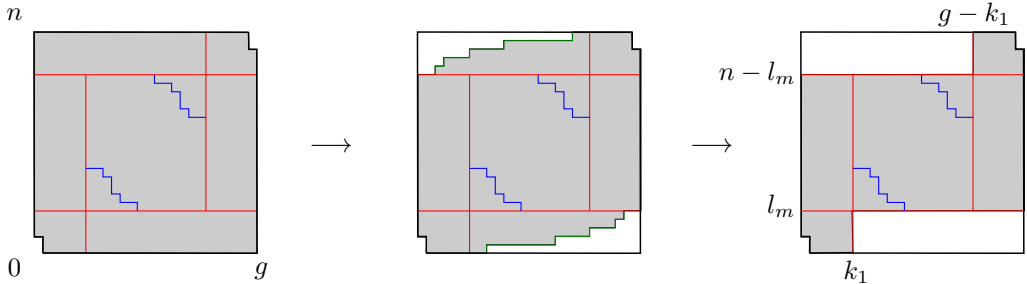
\begin{figure}[t]
\centering
\labellist \hair 1pt
	\pinlabel {$n$} at -70 900
	\pinlabel {$0$} at -70 -50
	\pinlabel {$g$} at 850 -65
	\pinlabel {$\longrightarrow$} at 1110 400
	%
	%
	\pinlabel {$\longrightarrow$} at 2530 400
	%
	\pinlabel {$n - l_m$} at 2700 680
	\pinlabel {$l_m$} at 2790 180
	\pinlabel {$k_1$} at 3070 -65
	\pinlabel {$g - k_1$} at 3510 900
    \endlabellist
\includegraphics[scale=0.1]{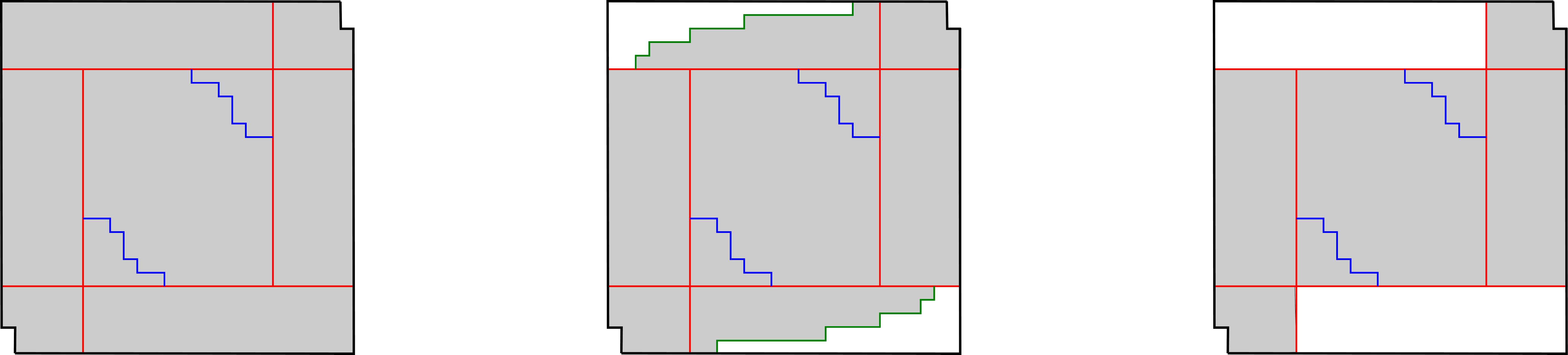}
\caption{In the first diagram the shaded area contains every point in $\vV_\vS$.  The horizontal red lines are $l_m$ and $n-l_m$, where $l_m$ is the lowest value of $l$ such that $\vC_{\vS(A)}(\Sigma)$ contains $(k,l)$-vertices.  We are able to find a sequence of linear subcomplexes between $\vC_\vS(\Sigma)$ and the subcomplex obtained by removing all $(k,l-1)$-vertices, where $k_1 \ge k \ge g$ and $0 \le l \le l_m$.  Here, $k_1$ is the lowest value of $k$ such that $(k,l)$-vertices belong to $\vC_{\vS(A)}(\Sigma)$.}
\label{diagram1}
\end{figure}

If $l_m = 1$ then $n \ge 4$ then by Lemma \ref{puncture_sp} we can show that $(k,1)$-puncture sharing pairs are preserved by automorphisms for $k_1 \le k \le g$.  This requires the fact that each linear subcomplex contains $(0,2)$-vertices.  We then proceed as above.  If $l_m=0$ then we need not remove any vertices to arrive at the desired subcomplex.

The next step is to once again apply Proposition \ref{Step} multiple times in order to obtain the automorphism group of the subcomplex obtained by further removing all $(k-1,l)$-vertices, for $0 \le k \le k_1$ and $l_m \le l \le n$.  Since all minimal vertices are small, we have that either $k_1=0$ or $g \ge 3 k_1 + 1$.  Suppose $k_1 \ge 1$ and let $\vC_X(\Sigma)$ be any maximal linear subcomplex of $\vC_\vS(\Sigma)$ containing $(1,0)$-vertices.  Since $g \ge k_1 + 3$, by Lemma \ref{genus_sp} we have that $(k,l)$-genus sharing pairs are preserved by automorphisms of $\vC_X(\Sigma)$ for values of $k$ and $l$ as above. Similar to the previous step, we may remove all such vertices and sustain an isomorphism between $\Mode(\Sigma)$ and the automorphism group of the verious subcomplexes complex, see Figure \ref{diagram2}. 
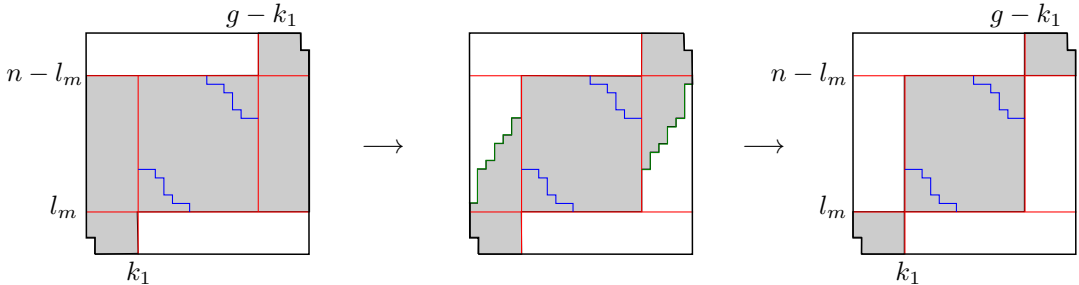
\begin{figure}[h]
\centering
\labellist \hair 1pt
	\pinlabel {$n - l_m$} at -150 680
	\pinlabel {$l_m$} at -80 180
	\pinlabel {$k_1$} at 200 -65
	\pinlabel {$g - k_1$} at 660 900
	%
	\pinlabel {$\longrightarrow$} at 1110 400
	%
	%
	\pinlabel {$\longrightarrow$} at 2530 400
	%
	\pinlabel {$n - l_m$} at 2700 680
	\pinlabel {$l_m$} at 2790 180
	\pinlabel {$k_1$} at 3070 -65
	\pinlabel {$g - k_1$} at 3510 900
    \endlabellist
\includegraphics[scale=0.1]{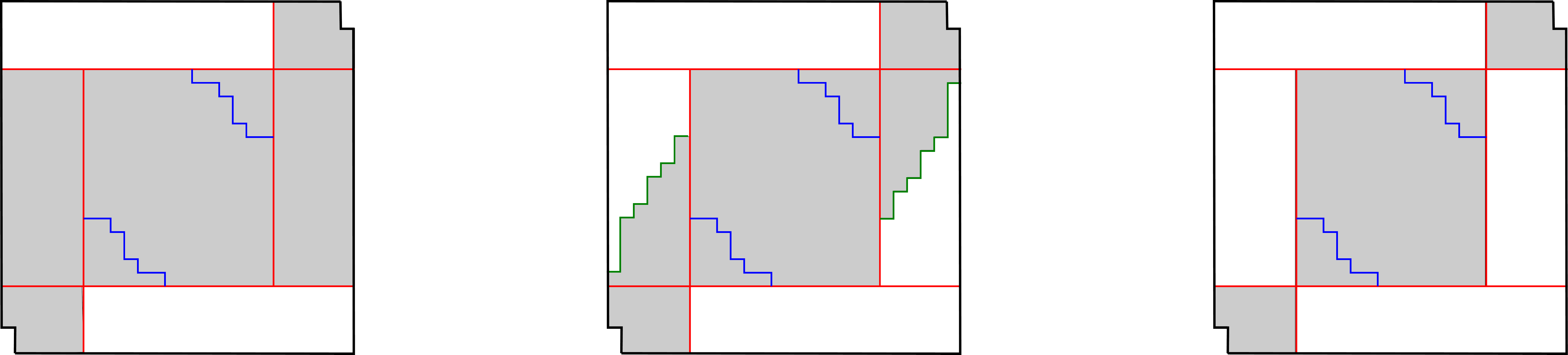}
\caption{The subcomplex of $\vC_\vS(\Sigma)$ obtained by removing all $(k,l)$-vertices, where $k_1 \ge k \ge g$ and $0 \le l \le l_m$, or $0 \le k \le k_1$ and $l_m \le l \le n$ has the extended mapping class group as its group of automorphisms.}
\label{diagram2}
\end{figure}
If $k_1=0$ we need not remove any vertices to arrive at the desired suubcomplex.  As such, we need not use Lemma \ref{genus_sp}.

The final step is to remove all $(k,l)$-vertices, for $k < k_i$ and $l < l_i$ for any $i \in \{1, \dots, m\}$.  Since every minimal vertex is small we have that the inequalities in both Lemmas \ref{genus_sp} and \ref{puncture_sp} are satisfied for either $(k+1,l)$-genus sharing pairs or $(k,l+1)$-puncture sharing pairs.  We may apply Proposition \ref{Step} again (see Figure \ref{diagram3}), and we conclude that $\eta_{\vS(A)} : \Mode(\Sigma) \to \vC_{\vS(A)}(\Sigma)$ is an isomorphism.
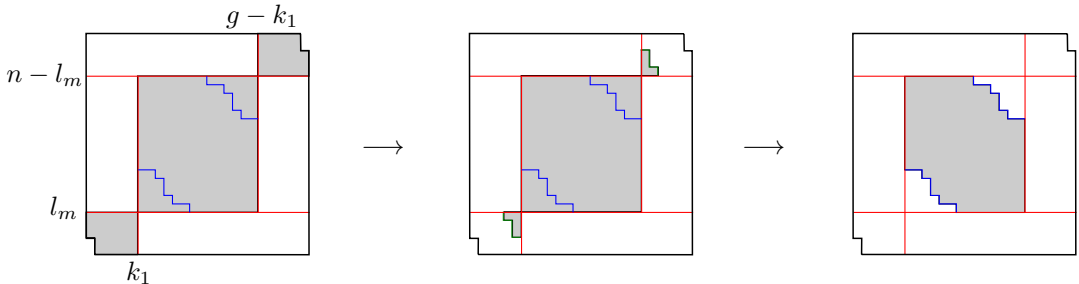
\begin{figure}[h]
\centering
\labellist \hair 1pt
	\pinlabel {$n - l_m$} at -150 680
	\pinlabel {$l_m$} at -80 180
	\pinlabel {$k_1$} at 200 -65
	\pinlabel {$g - k_1$} at 660 900
	%
	\pinlabel {$\longrightarrow$} at 1110 400
	%
	%
	\pinlabel {$\longrightarrow$} at 2530 400
	%
    \endlabellist
\includegraphics[scale=0.1]{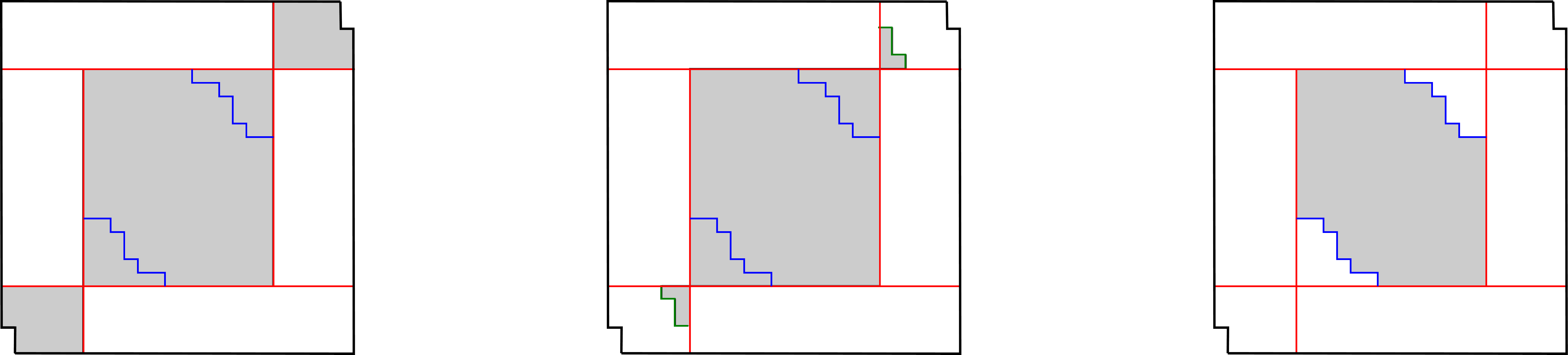}
\caption{The blue lines has endpoints at $(k_1,l_1)$ and $(k_m,l_m)$, and corners at $(k_i,l_i)$ where $(k_i,l_i)$-vertices are minimal for all values of $i \in \{1,2,\dots,m\}$.  This figure, together with Figures \ref{diagram1} and \ref{diagram2}, show that there exists a sequence of linear subcomplexes between $\vC_\vS(\Sigma)$ and $\vC_{\vS(A)}(\Sigma)$.}
\label{diagram3}
\end{figure}
\end{proof}

\section{Complexes of dividing sets}\label{section_DS}
The purpose of this section is to connect Theorem \ref{etas} with complexes of regions. We do this by using a generalisation of separating curves for a surface $\Sigma$ of strictly positive genus introduced by Brendle-Margalit \cite[Section 4]{BM17}.  As in the previous section we shall therefore assume that $\Sigma = \Sigma_{g,n}$, and that $g,n > 0$.

\subsection*{Dividing sets}
A \emph{dividing set} in $\Sigma$ is a multicurve that divides the surface into exactly two regions. We allow for one of the regions to be an annulus, that is, the multicurve may consist of two isotopic non-separating curves. As with separating curves, we call the two regions obtained by cutting along a dividing set $\bd$ the \emph{associated regions} of $\bd$. We say that two dividing sets are \emph{nested} if one is contained entirely in one of the associated regions of the other, otherwise we say that they \emph{intersect}. If two dividing sets intersect then their respective multicurves may intersect or they may not.

Let $\vDS$ denote the set of all $\Mode(\Sigma)$-orbits of dividing sets in $\Sigma$. For a subset $D \subseteq \vDS$ we define the simplicial flag complex $\vC_D(\Sigma)$ analogously to complexes of regions. The vertices of $\vC_D(\Sigma)$ correspond to all homotopy classes of dividing sets that represent elements of $D$.  We say that a vertex corresponds to a dividing set if it corresponds to the equivalence class of that dividing set. Two vertices span an edge in $\vC_D(\Sigma)$ if they correspond to nested dividing sets.  As with complexes of regions there is a natural homomorphism
\[
\eta_D : \Mode (\Sigma) \rightarrow \Aut \vC_D (\Sigma)
\]
for every subset $D \subseteq \vDS$.

For any dividing set $\bd$ an \emph{enveloping region} $\wh R_\bd$ of $\bd$ is a single-boundary region such that $\bd \subset \wh R_\bd$ and $\bd$ is not contained in any proper single-boundary subsurface of $\wh R_\bd$.  If the vertex $v \in \vC_D(\Sigma)$ corresponds to the dividing set $\bv$, we write $\wh v$ for the enveloping region of $\bv$, that is, $\wh v := \wh R_\bv$.  The following definitions are equivalent to those made in Section \ref{introduction} in the context of complexes of dividing sets.

\subsection*{Minimal vertices}
Let $\vC_D(\Sigma)$ be a complex of dividing sets.  We say that a vertex $v \in \vC_D(\Sigma)$ is \emph{minimal} if for any vertex $u$ such that $\wh u \subset \wh v$, we have that $\wh u$ and $\wh v$ are homotopic.

The following definition of small vertices is inherited from the subcomplexes of separating curves we visited in the previous section.

\subsection*{Small vertices}
Let $\Sigma = \Sigma_{g,n}$ and let $\vC_D(\Sigma)$ be a complex of dividing sets.  We say that a vertex $v \in \vC_D(\Sigma)$ is \emph{small} if there exist two vertices $u_1, u_2$ such that
\begin{align}
g &\ge g(\wh v) + \max \{ g( \wh u_1) + g( \wh u_2) , 2 \} + 1, \mbox{ and} \\ 
n &\ge n( \wh v) + \max \{ n( \wh u_1) + n( \wh u_2), 1\} + 1.
\end{align}

For $A \subset \vR (\Sigma)$ we define $\partial A \subseteq \vDS$ to be the subset consisting of dividing sets where each of the associated regions contain a region represented in $A$.
In this section we will use Theorem \ref{etas} to prove the following result.

\begin{thm}\label{etad}
Let $\vC_{\partial A}(\Sigma)$ be a complex of dividing sets for some $A \subset \vR(\Sigma)$.  If every minimal vertex of $\vC_{\partial A}(\Sigma)$ is small then the natural homomorphism
\[
\eta_{\partial A} : \Mode(\Sigma) \to \Aut \vC_{\partial A}(\Sigma)
\]
is an isomorphism.
\end{thm}
Notice that in the special case $\Sigma_{0,n}$ we have $\vS = \vDS$ and $\vS(A) = \partial A$ for $A \subset \vR (\Sigma)$. In general $\vS(A) = \partial A \cap \vS$.  Suppose $u_1, u_2 \in \vC_{\partial A}(\Sigma)$ are two vertices that span an edge.  If $u_1, u_2$ correspond to dividing sets $\bu_1, \bu_2$ separated by the dividing set $\bv$ then $\bv$ must also separate two regions represented in $A \subset \vR(\Sigma)$.  If follows that there exists a vertex $v \in \vC_{\partial A}(\Sigma)$ corresponding to $\bv$.  We will use this fact throughout this section and the proof of Theorem \ref{etad} without mention.

\subsection*{The case with annular dividing sets}

We call a dividing set \emph{annular} when it has an annular associated region. Clearly, there is a bijection between the isotopy classes of annular dividing sets and isotopy classes of non-separating curves. Suppose then that annular dividing sets are represented in $\partial A \subseteq \vDS$, that is, annuli are represented in $A \subset \vR(\Sigma)$. It follows from \cite[Lemma 4.1]{BM17} that the vertices of $\vC_{\partial A}(\Sigma)$ that correspond to annular dividing sets form a characteristic subset. We thus obtain an injective homomorphism
\[
\Aut \vC_{\partial A}(\Sigma) \hookrightarrow \Aut \vN (\Sigma),
\]
where $\vN(\Sigma)$ is the complex of non-separating curves. From Lemma \ref{Injectivity} and \cite[Theorem 1.4]{Nonsep} we have that the composition
\[
\Mode(\Sigma) \xhookrightarrow{\eta_{\partial A}} \Aut \vC_{\partial A}(\Sigma) \hookrightarrow \Aut \vN (\Sigma) \xrightarrow{\cong} \Mode(\Sigma)
\]
is injective and equal to the identity map, therefore $\eta_{\partial A} : \Mode(\Sigma) \rightarrow \Aut \vC_{\partial A}(\Sigma)$ is an isomorphism. In the remainder of this section we will assume that annular dividing sets are not represented in $\partial A \subseteq \vDS$ and prove that the homomorphism is an isomorphism in this case as well.

\subsection{Characteristic vertex types}\label{vertex_types_DS}
Assume throughout this section that no annular dividing sets are represented in $\partial A$.  Let $\sigma$ be any simplex in the complex $\vC_{\partial A}(\Sigma)$ consisting of vertices $v_1,\dots,v_m$. We call a collection of pairwise nested multicurves $\bv_1,\dots,\bv_m$ a \emph{normal form} representative for $\sigma$ if each $v_i$ corresponds to $\bv_i$. We state the following result of Brendle-Margalit \cite[Lemma 4.3]{BM17}.

\begin{lem}[Brendle-Margalit]\label{Normal}
Let $\sigma$ be a simplex of $\vC_{\partial A}(\Sigma)$. There exists a normal form representative of $\sigma$, unique up to isotopy.
\end{lem}

As dividing sets are a generalisation of separating curves, we may employ similar techniques when studying complexes of dividing sets. In particular, we can define \emph{sides} of vertices corresponding to dividing sets by analysing their links as in Section \ref{vertex_types}.  Recall, two vertices $u,w \in \Lk(v)$ lie on the same side of the vertex $v$ if there exists another vertex in $\Lk(v)$ that does not span an edge with either $u$ or $w$.

We say that a vertex $v$ of $\vC_{\partial A}(\Sigma)$ is \emph{$1$-sided} if every vertex of $\Lk(v)$ lies on the same side of $v$.  We say that $v$ is $2$-sided if there are vertices of $\Lk(v)$ that lie on different sides of $v$.  If $v$ is an isolated vertex we call it $0$-sided.  Notice that every $1$-sided vertex is minimal.  There may, however, be minimal vertices corresponding to multicurves that are not $1$-sided.

\subsection*{Vertex types}
For all $v \in \vC_{\partial A}(\Sigma)$ corresponding to a dividing set $\bv$, we define $| v |$ to be the number of components of $\bv$.
\begin{enumerate}
\item We say that a $1$-sided vertex $v$ is
\begin{center}
type $S_1$ if $|v|=1$, \eand type $M_1$ if $|v|\ge2$.
\end{center}
\item If $v$ is $2$-sided, and every vertex on one side of $v$ is type $S_1$ then we say $v$ is
\begin{center}
type $S_X$ if $|v| = 1 $, \eand type $M_X$ if $|v|\ge2$.
\end{center}
\item Finally, if $v$ is any other $2$-sided vertex we say $v$ is
\begin{center}
type $S_2$ if $|v| = 1 $, \eand type $M_2$ if $|v|\ge2$.
\end{center}
\end{enumerate}
Here, the letters `$S$' and `$M$' indicate that the vertex corresponds to a separating curve or multicurve respectively.

Our goal now is to show that vertices of type $S_1,S_2$ and $S_2$ form characteristic subsets of $\vC_{\partial A}(\Sigma)$, that is, separating curves determine a characteristic subset of vertices in $\vC_{\partial A}(\Sigma)$.  Recall from Section \ref{vertex_types} that a linear simplex is one with an ordering of the vertices determined by the sides of the corresponding curves. We use the same terminology in the case of dividing sets.

\subsection*{Linear simplices} A simplex $\sigma$ of $\vC_{\partial A}(\Sigma)$ is \emph{linear} if there is a labeling of its vertices $v_1,\dots,v_m$ such that $v_{i-1}$ and $v_{i+1}$ do not lie on the same side of $v_i$ for all $i = 1, \dots, m-1$. We call the vertices $v_0$ and $v_m$ the extreme vertices of the linear simplex $\sigma$.  As discussed in Section \ref{vertex_types} we have the following result.

\begin{lem}\label{LinSim2}
Let $\vC_{\partial A}(\Sigma)$ be a complex of dividing sets and let $\phi$ be an automorphism. If $\sigma = \{v_1,\dots,v_m\}$ is a maximal linear simplex then $\phi(\sigma) = \{\phi(v_1),\dots,\phi(v_m)\}$ is a maximal linear simplex.
\end{lem}
We now move on to showing that the various vertex types form characteristic subsets, beginning with vertices of type $S_1$.

\begin{lem}\label{char1}
Let $\vC_{\partial A}(\Sigma)$ be a complex of dividing sets.  If every minimal vertex of $\vC_{\partial A}(\Sigma)$ is small then the type $S_1$ vertices form a characteristic subset.
\end{lem}

\begin{proof}
It follows from the definition of a maximal linear simplex that a vertex $v$ is $1$-sided if and only if it is an extreme vertex of some maximal linear simplex. We will show then that a vertex is type $M_1$ if and only if it is $1$-sided and there exist vertices $u,w$ such that;
\begin{enumerate}
\item $u$ and $w$ span a triangle with $v$, and
\item any other $1$-sided vertex spanning a triangle with $u$ and $w$ spans an edge with $v$.
\end{enumerate}
To prove one direction suppose $v$ is type $M_1$ and corresponds to the multicurve $\bv$. Let $u$ be a vertex in the $\Mode (\Sigma)$-orbit of $v$ corresponding to the multicurve $\bu$ disjoint from $\bv$ such that exactly one of the curves in $\bv$ is isotopic to a curve in $\bu$. Let $R$ be the unique region defined by cutting along $\bu$ and $\bv$ that contains more than one dividing set represented in $\partial A$. We now define $w$ to be the vertex of $\vC_{\partial A}(\Sigma)$ corresponding to $\partial R$. Clearly the vertices $u,v$ and $w$ span a triangle. Now, any choice of $1$-sided vertex, other than $v$, that spans a triangle with $u$ and $w$ must correspond to a dividing set contained in $R$. It follows then that any such vertex spans an edge with $v$.

Now assume that $v$ is a vertex type $S_1$ corresponding to $\bv$ and let $u$ and $w$ be vertices corresponding to dividing sets $\bu,\bw$ satisfying the conditions above. Let $R$ be the region of $\Sigma$ with boundary defined by $\bv$ and containing $\bu,\bw$. Since $v$ does not correspond to $\bu$ or $\bw$ there exists an element in the $\Mode (\Sigma)$-orbit of $\bv$ that is disjoint from $\bu$ and $\bw$ and intersects $\bv$. This completes the proof.
\end{proof}

We treat the remaining cases seemingly out of order by first showing that type $S_2$ vertices form a characteristic subsets before dealing with type $S_X$ vertices.

\begin{lem}\label{char2}
Let $\vC_{\partial A}(\Sigma)$ be a complex of dividing sets.  If every minimal vertex of $\vC_{\partial A}(\Sigma)$ is small then the type $S_2$ vertices form a characteristic subset.
\end{lem}

\begin{proof}
It follows from Lemmas \ref{LinSim2} and \ref{char1} that the sets $S_1$, $M_1$ and $S_2 \cup M_2$ form characteristic subsets of $\vC_{\partial A}(\Sigma)$.  It remains only to show that we can distinguish between type $M_2$ vertices and type $S_2$ vertices. We claim that a vertex $v$ is type $M_2$ if and only if;
\begin{enumerate}
\item there exist two vertices $u$ and $w$ that span a triangle with $v$, and 
\item there exists exactly one vertex that spans an triangle with $u$ and $w$ and that fails to span an edge with $v$.
\end{enumerate}

To prove the forward direction of the claim we assume $v$ is type $M_2$ and consider three cases separately; $|v| \ge 4$, $|v| = 3$, and $|v| = 2$.
In each case we will define vertices $u$ and $w$ that are on different sides of $v$.  Suppose $u,v,w$ correspond to the multicurves $\bu,\bv,\bw$ respectively.  In order to define the unique dividing set implicit in the claim we require that the (possibly connected) subsurface bounded by $\bu$ and $\bv$ is a (possibly empty) collection of annuli and a single pair of pants $P_u$.  We define a pair of pants $P_w$ related to the dividing set $\bw$ in the same way. Here, we go against convention slightly by defining a pair pants to be homeomorphic to either $\Sigma^3_{0,0}$ or $\Sigma^2_{0,1}$. Furthermore, we require that if a component curve of $\bv$ bounds $P_u$ (or $P_w$) it must bound an annulus with $\bw$ (or $\bu$).  The unique vertex spanning edges with $u$ and $w$ but not $v$ corresponds to the dividing set 
\[
\{\bu \cap \bw\} \cup \{\partial P_u \setminus \bv\} \cup \{\partial P_w \setminus \bv\}.
\]
An example is shown in Figure \ref{N=4}.
\begin{figure}[t]
\centering
\labellist \hair 1pt
	\pinlabel {$\bv$} at -30 325
	\pinlabel {$P_w$} at 380 380
	\pinlabel {$P_u$} at 900 190
    \endlabellist
\includegraphics[scale=0.2]{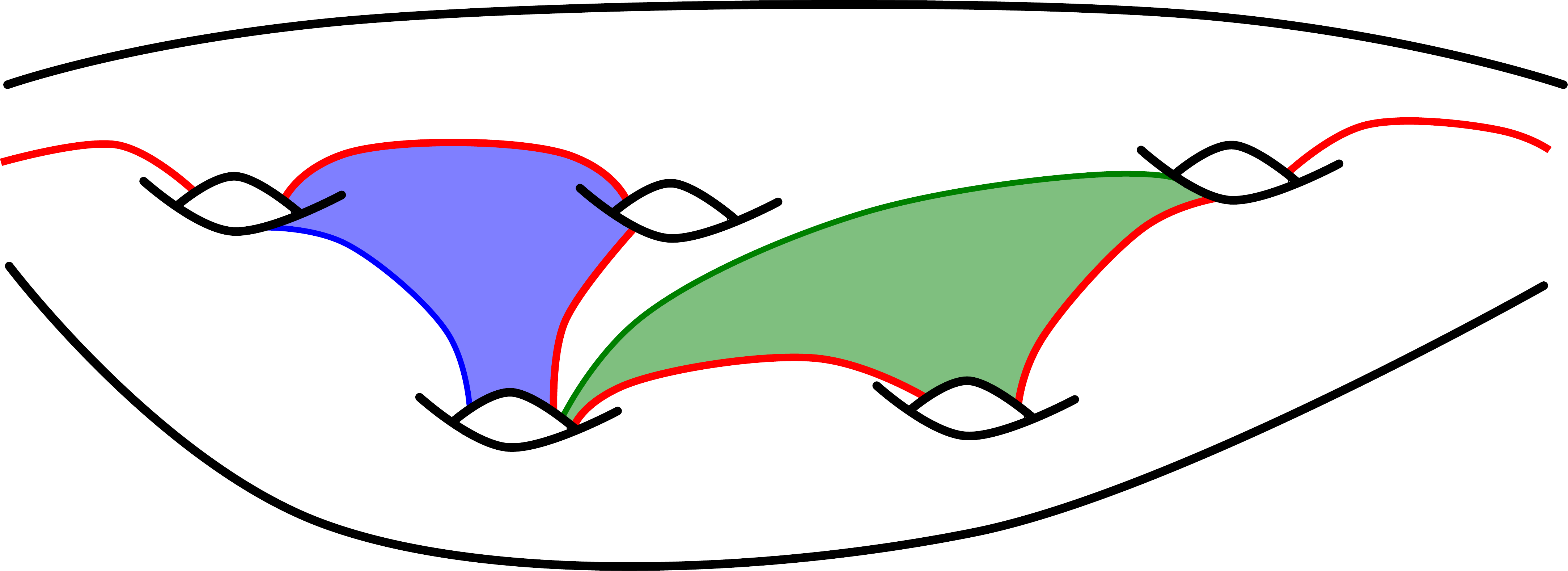}
\caption{The red multicurve is $\bv$. We see that we can construct the desired pairs of pants $P_u$ and $P_w$.}
\label{N=4}
\end{figure}
It follows that in order to prove the claim, hence the lemma, it requires to find the required pairs of pants $P_u, P_w$ with respect to a multicurve $\bv$.

First we consider the case where $|v| \ge 4$. The pair of pants $P_u$ will consist either of three boundary components, or two boundary components and a single puncture. Suppose that such a $P_u$ does not exist, then every dividing set $\bd$ nested with $\bv$ will be isotopic to $\bv$. This contradicts our assumption that $v$ is $2$-sided. Similarly, we can find a pair of pants $P_w$ satsifying the conditions above, see Figure \ref{N=4}.

Now let $| v | = 3$ and let $R_u$ and $R_w$ be the two associated regions of $\bv$ such that $g(R_w) \ge g(R_u)$. Suppose we can choose a dividing set $\bu$ in $R_u$ with four components, two of which are isotopic to distinct components of $\bv$. Since $v$ is $2$-sided we can find an appropriate choice of $\bw$ contained in $R_w$ where either $\bw$ has two components and $P_w$ is homeomorphic to $\Sigma_{0,0}^3$ or $\bw$ has three components and $P_w$ is homeomorphic to $\Sigma_{0,1}^2$. This is shown in Figure \ref{N=3} (i) and (ii), where $\bu$ is the dividing set on the right and $\bw$ is on the left.
\begin{figure}[h]
\centering
\labellist \hair 1pt
	\pinlabel {(i)} at 250 -50
	\pinlabel {(ii)} at 910 -50
	\pinlabel {(iii)} at 1580 -50
	\pinlabel {(iv)} at 2260 -50
	\pinlabel {$\bv$} at 320 360
	\pinlabel {$\bv$} at 980 360
	\pinlabel {$\bv$} at 1645 360
	\pinlabel {$\bv$} at 2310 360
    \endlabellist
\includegraphics[scale=0.15]{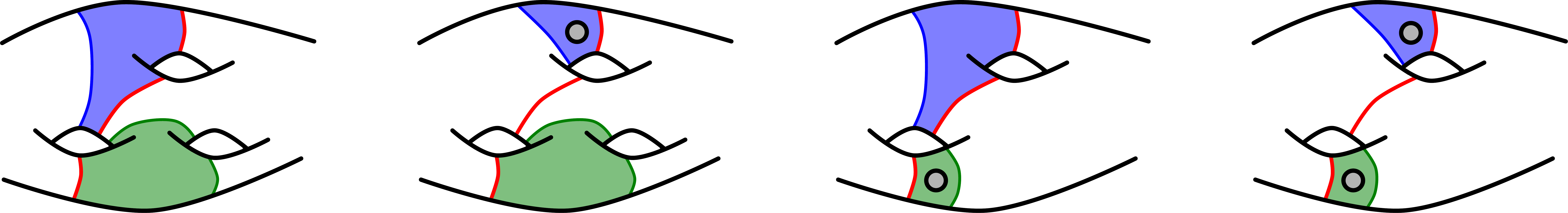}
\caption{The red multicurve with three components is $\bv$. We see that we can always construct desired pairs of pants $P_u$ and $P_w$.}
\label{N=3}
\end{figure}
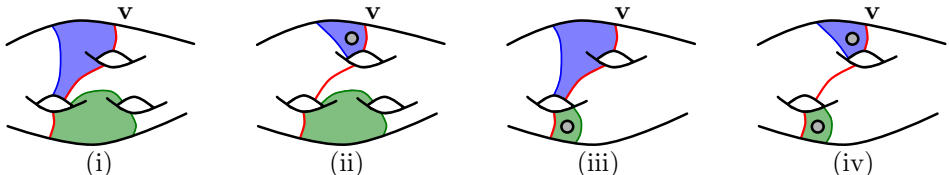
Similarly, suppose we can choose $\bu$ with three boundary components, two of which belong to $\bv$ and where the region $P_u$ is homeomorphic to $\Sigma^2_{0,1}$. Once again, as $v$ is $2$-sided, there is an appropriate choice of $\bw$ in $R_w$. A picture can be seen in Figure \ref{N=3} (iii) and (iv), again $\bu$ is on the right and $\bw$ is the dividing set on the on the left.

If neither choice of $\bu$ exists it follows that there are no $1$-sided vertices of $\vC_{\partial A}(\Sigma)$ corresponding to dividing sets in $R_u$ with associated region of lower genus or fewer punctures than $R_u$.  We deduce that there is a $1$-sided vertex of $\vC_{\partial A}(\Sigma)$ corresponding to a dividing set $\bd$ such that
\[
g(R_u) \le g(\wh R_\bd) \le g(R_u) + 1 \eand n(\wh R_\bd) = n(R_u),
\]
where $\wh R_\bd$ is the enveloping region of the dividing set $\bd$.  Note that $\bd$ will have one or two components.  All $1$-sided vertices are minimal and so by assumption all $1$-sided vertices are small. From the definition of small, we have that there exist two vertices corresponding to dividing sets $\bd_1, \bd_2$ such that;
\begin{eqnarray*}
g(R_w) = g - g(R_u) - 2 &\ge& g(\wh R_\bd) + g(\wh R_{\bd_1}) + g(\wh R_{\bd_1}) + 1 - g(\wh R_\bd) - 2,
\\ &\ge& g(\wh R_{\bd_1}) + g(\wh R_{\bd_2}) - 1, \eand
\\ n(R_w) = n - n(R_u) &\ge& n(\wh R_\bd) + n(\wh R_{\bd_1}) + n(\wh R_{\bd_1}) + 1 - n(\wh R_\bd),
\\ &\ge& n(\wh R_{\bd_1}) + n(\wh R_{\bd_2}) + 1.
\end{eqnarray*}
Without loss of generality we can assume that $g(\wh R_{\bd_2}) \ge g(\wh R_{\bd_1})$.  From the first inequality we have that $g(R_w) \ge g(\wh R_{\bd_1})$ if $g(\wh R_{\bd_2}) \ge 1$.  However, if $g(\wh R_{\bd_2}) < 1$ then $g(\wh R_{\bd_1}) = 0 \le g(R_w)$.  From the second inequality it is clear that $n(R_w) \ge n(\wh R_{\bd_1}) + 1$.  We conclude that there exists an element $f \in \Mode(\Sigma)$ such that $f(\bd_1)$ is contained in $R_w$.  Moreover, there exists a dividing set $\bw$ in $R_w$ separating $f(\bd_1)$ and $\bv$ such that $\bw$ has an associated region $Q_w \subset R_w$ with three boundary components, where $g(Q_w)=g(R_w)$, and $n(Q_w) = n(R_w) - 1$, see Figure \ref{N=3}(iii).  We may now choose $\bu \subset R_u$ to have two components, as depicted in the left hand side dividing set of see Figure \ref{N=3}(iii).  This completes the proof in the case where $|v|=3$.

Now we deal with the case where $| v |=2$. If both associated regions of $\bv$ contain dividing sets $\bu$ and $\bw$ such that $P_u$ and $P_w$ are homeomorphic to $\Sigma_{0,1}^2$ then we are done. If this is not the case then since $v$ is of type $M_2$ there exists a vertex in $\vC_{\partial A}(\Sigma)$ spanning an edge with $v$ that is either of type $S_X$ or $S_2$. Any such vertex is not $1$-sided and so we can find a dividing set $\bu$ with three components, two of which are shared by $\bv$. As before we can therefore find the desired pairs of pants $P_u$ and $P_w$.

We now assume that $v$ is a vertex of type $S_2$. If $u$ and $w$ lie on the same side of $v$ then up to relabeling there are infinitely many vertices in the $\Mode (\Sigma)$-orbit of $v$ spanning edges with $u$ and $w$ but not with $v$. Suppose then that $u$ and $w$ lie on different sides of $v$. If the vertices $u,v$ correspond to the dividing sets $\bu,\bv$ then the subsurface bounded by these curves cannot be an annulus, as $\bv$ is a separating curve. It follows that there are infinitely many vertices in the $\Mode (\Sigma)$-orbit of $v$ spanning edges with $u$ and $w$ but not with $v$. This completes the proof.
\end{proof}

Finally we complete the proof that the vertices of $\vC_{\partial A}(\Sigma)$ corresponding to separating curves form characteristic subsets by distinguishing type $S_X$ vertices and type $M_X$ vertices.

\begin{lem}\label{char3}
Let $\vC_{\partial A}(\Sigma)$ be a complex of dividing sets.  If every minimal vertex of $\vC_{\partial A}(\Sigma)$ is small then the type $S_X$ vertices form a characteristic subset.
\end{lem}

\begin{proof}
From Lemmas \ref{LinSim2} and \ref{char1} we see that the subset of type $S_X$ and $M_X$ vertices forms a characteristic subset.
Let $u$ be either a type $S_X$ vertex or a type $M_X$ vertex.  Suppose $u$ corresponds to a dividing set with associated regions $R$ and $Q$ such that only type $S_1$ vertices correspond to dividing sets in $R$. We will call $u$ \emph{genus separating} if $g(Q) \ge 1$, see Figure \ref{M_2}(i).
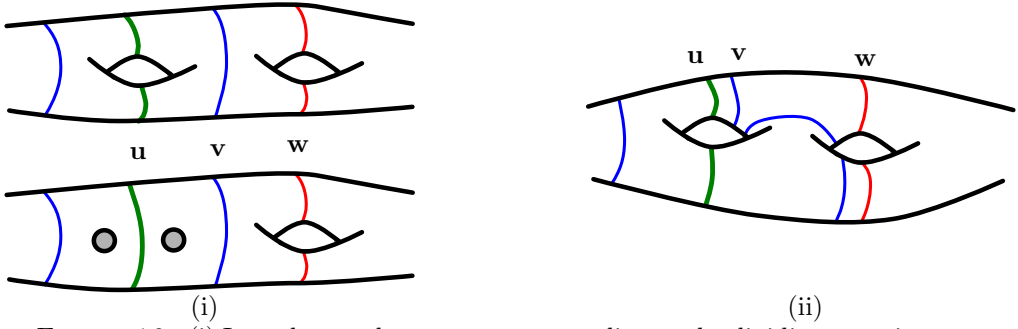
\begin{figure}[t]
\centering
\labellist \hair 1pt
	\pinlabel {(i)} at 300 -25
	\pinlabel {(ii)} at 1200 -25
	\pinlabel {$\bu$} at 200 210
	\pinlabel {$\bv$} at 320 215
	\pinlabel {$\bw$} at 440 220
	\pinlabel {$\bu$} at 1035 355
	\pinlabel {$\bv$} at 1100 360
	\pinlabel {$\bw$} at 1290 350
    \endlabellist
\includegraphics[scale=0.25]{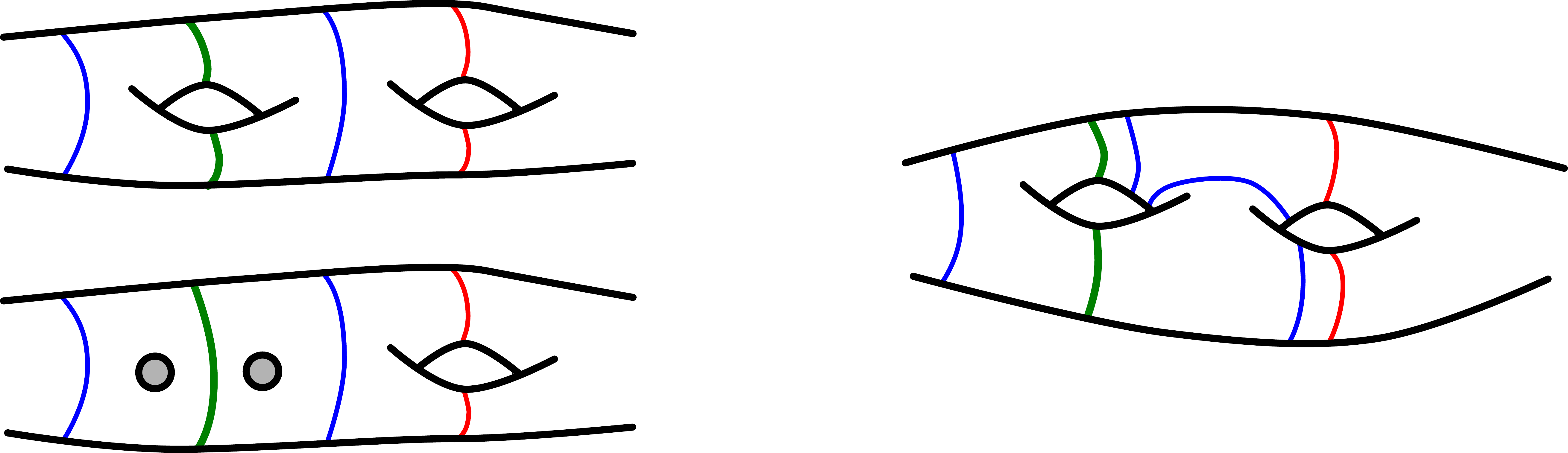}
\caption{(i) In each case the vertex corresponding to the dividing set $\bu$ is genus separating.  As such, we can find a curve $\bv$ with an associated region of positive genus that does not contain $\bu$.  (ii) The vertex corresponding to $\bu$ is of type $M_X$.  As it is genus separating we can find a dividing set $\bv$ with three components.}
\label{M_2}
\end{figure}
Note that if all minimal vertices are small, then all type $S_X$ vertices are genus separating.  We begin my showing that the subset of genus separating vertices forms a characteristic subset.  We claim that $u$ is genus separating if and only if there exists a type $S_2$ vertex $v$ and a type $M_2$ vertex $w$ such that
$u$ and $w$ lie on different sides of $v$.

To prove the claim, first assume that $u$ is genus separating and corresponds to the dividing set $\bu$.  We can define a curve $\bv$ such that $\bu$ and $\bv$ bound a region $P$ homeomorphic to $\Sigma_{0,0}^3$ or $\Sigma_{0,1}^2$ depending on whether $u$ is type $M_X$ or $S_X$ respectively.  Let $v$ correspond to $\bv$.  Now, since all $1$-sided vertices are minimal, and all minimal vertices are small, we have that there exist vertices of $\vC_{\partial A}(\Sigma)$ corresponding to dividing sets $\bd_1$ and $\bd_2$ contained in $Q$ such that
\[
g(Q) \ge g - g(R) - 1 \ge g(\wh R_{\bd_1}) + g(\wh R_{\bd_2}) \eand n(Q) \ge n - n(R) \ge n(\wh R_{\bd_1}) + n(\wh R_{\bd_2}) + 1
\]
if $u$ is type $M_X$. If $u$ is a type $S_X$ vertex then for $\bd_1$ and $\bd_2$ as above we have
\[
g(Q) \ge g - g(R) \ge g(\wh R_{\bd_1}) + g(\wh R_{\bd_2}) + 1 \eand n(Q) \ge n - n(R) - 1 \ge n(\wh R_{\bd_1}) + n(\wh R_{\bd_2}).
\]
In either case we conclude that there exists a type $M_2$ vertex $w$ corresponding to a dividing set that separates $\bd_i$ from $\bv$, for some $i \in \{1,2\}$.  Thus, we have that $u$ and $w$ lie on different sides of $v$.

Now assume that $u$ is not genus separating and let $v$ be a type $S_2$ vertex that spans an edge with $u$.  Suppose $v$ corresponds to a separating curve with associated region $Q_v \subset Q$.  If there exists a vertex $w$ as above then it must correspond to a dividing set contained in $Q_v$.  This implies that $g(Q) \ge g(Q_v) \ge 1$, a contradiction.

In order to prove the lemma we will show that type $M_X$ genus separating vertices form a characteristic subset. We claim that if $u$ is genus separating then $u$ is type $M_X$ if and only if there exists a type $M_2$ vertex $v$ such that;
\begin{enumerate}
\item the vertices $u,v$ are sequential in a maximal linear simplex, and
\item there is no type $S_2$ vertex $w$ such that $u,v,w$ are sequential in a maximal linear simplex.
\end{enumerate}
First we let $u$ be a vertex of type $M_X$.  Since $u$ is genus separating we can find a vertex that corresponds to a dividing set with three components, as shown in Figure \ref{M_2}(ii).  It is clear that there is no type $S_1$, $S_X$, or $S_2$ vertex $w$ such that $u,v,w$ are sequential in a maximal linear simplex.

If $u$ is type $S_X$ then since it is genus separating we can find a vertex $v$ that corresponds to a dividing set with two components.  We can then find a vertex $w$ corresponding to a separating curve such that $u,v,w$ are sequential in a maximal linear simplex.
\end{proof}

Combining Lemmas \ref{char1}, \ref{char2}, and \ref{char3} we have that the set veritces of $\vC_{\partial A}(\Sigma)$ 
corresponding to separating curves forms a characteristic subset.

\subsection{The case without annular dividing sets}
We can now prove Theorem \ref{etad} which states that if every minimal vertex of $\vC_{\partial A}(\Sigma)$ is small then the natural homomorphism
\[
\eta_{\partial A} : \Mode (\Sigma) \rightarrow \Aut \vC_{\partial A}(\Sigma)
\]
is an isomorphism.  We will make use of Theorem \ref{etas} from Section \ref{section_subcomplexes} and the results of Section \ref{vertex_types_DS}.

\begin{proof}[Proof of Theorem \ref{etad}]
By Lemma \ref{Injectivity} we have that $\eta_{\partial A}$ is injective. We want then to show that $\eta_{\partial A}$ is surjective. Let $\phi \in \Aut \vC_{\partial A}(\Sigma)$.  It follows from Lemmas \ref{char1}, \ref{char2}, and \ref{char3} that $\phi$ restricts to an automorphism $\widehat \phi$ of $\vC_{\vS(A)}(\Sigma)$. Here we think of $\vC_{\vS(A)}(\Sigma)$ as a full subcomplex of $\vC_{\partial A}(\Sigma)$.  All minimal vertices of $\vC_{\vS(A)}(\Sigma)$ are also minimal vertices of $\vC_{\partial A}(\Sigma)$ and so, by assumption, they are small.  By Theorem \ref{etas} there exists a mapping class $f \in \Mode (\Sigma)$ such that $\eta_{\vS(A)}(f)= \widehat \phi$. We need to show that $\eta_{\partial A}(f) = \phi$.

It suffices to show that an automorphism of $\vC_{\partial A}(\Sigma)$ restricting to the identity on $\vC_{\vS(A)}(\Sigma)$ must be the identity. To do this, we show by induction on the distance from a vertex to the subcomplex $\vC_{\vS(A)}(\Sigma)$.  Since $\vC_{\partial A}(\Sigma)$ is connected the result follows.

By assumption, the automorphism $\phi$ restricts to the identity for all vertices distance zero from $\vC_{\vS(A)}(\Sigma)$. Assume then that $\phi$ restricts to the identity for all vertices of $\vC_{\partial A}(\Sigma)$ distance $k$ from $\vC_{\vS(A)}(\Sigma)$. We deal with the inductive step separately for $1$-sided vertices and $2$-sided vertices.

Let $v$ be a $1$-sided vertex of $\vC_{\partial A}(\Sigma)$ that is distance $k+1$ from a vertex of $\vC_{\vS(A)}(\Sigma)$. Let $u$ be a vertex of $\vC_{\partial A}(\Sigma)$ spanning an edge with $v$ and distance $k$ from a vertex of $\vC_{\vS(A)}(\Sigma)$. Let $u,v$ correspond to $\bu,\bv$. There exist elements of the $\Mode(\Sigma)$-orbit of $\bu$ that fill the associated region of $\bv$ containing $\bu$. The vertex $v$ is $1$-sided, hence is the unique vertex whose link contains vertices corresponding to such dividing sets. It follows that $\phi$ must also fix $v$.

Assume now that $v$ is a $2$-sided vertex that is distance $k+1$ from $\vC_{\vS(A)}(\Sigma)$.  Let $u$ be a vertex of $\vC_{\partial A}(\Sigma)$ adjacent to $v$ that is distance $k$ from $\vC_{\vS(A)}(\Sigma)$.  Let $w$ be a $1$-sided vertex of $\vC_{\partial A}(\Sigma)$ that is not on the same side of $v$ as $u$.  It follows that $w$ is at most distance $k+1$ from $\vC_{\vS(A)}(\Sigma)$.  If $u,v,w$ correspond to $\bu,\bv,\bw$ then using similar methods to the previous step we can show that the orbits of $\bu$ and $\bw$ fill the associated regions of $\bv$.  As all distance $k$ vertices, and all $1$-sided, distance $k+1$ vertices are are fixed by $\phi$ we conclude that $v$ is also fixed by $\phi$, completing the proof.
\end{proof}

\section{Complexes of regions}\label{theproof}
In this section we will complete the resolution of the metaconjecture in the case of surfaces with punctures, that is we prove Theorem \ref{BigDaddy}.  A key step to this result is invoking Theorem \ref{etad} which we proved in the previous section.  We relate the complex of dividing sets $\vC_{\partial A}(\Sigma)$ and the complex of regions $\vC_A(\Sigma)$.  This is achieved by observing a bijection between the vertices of the complex $\vC_{\partial A}(\Sigma)$ and particular \emph{joins} in the complex $\vC_A(\Sigma)$.  This allows us to construct an injective homomorphism
\[
\partial : \Aut \vC_A(\Sigma) \rightarrow \Aut \vC_{\partial A}(\Sigma).
\]
We then consider the injective homomorphism $\eta_{\partial A}^{-1} \circ \partial \circ \eta_A$, where $\eta_{\partial A}$ is the isomorphism from Theorem \ref{etad}, and show that it is the identity of $\Mode(\Sigma)$.

\subsection{Types of join}
First we define a map
\[
\Phi : \Big \{\text{vertices of }\vC_{\partial A}(\Sigma) \Big \} \rightarrow \Big \{\text{subcomplexes of } \vC_A(\Sigma) \Big \}.
\]
Given a vertex $v$ of $\vC_{\partial A}(\Sigma)$ corresponding to a dividing set $\bv$, define $\Phi(v)$ to be the full subcomplex of $\vC_A(\Sigma)$ spanned by the vertices that correspond to regions contained in the associated regions of the dividing set $\bv$.

Recall that a subcomplex $X \subset \vC_A(\Sigma)$ is a \emph{join} if $X$ is spanned by disjoint subsets of vertices $V_1, \dots, V_m$, such that every vertex in $V_i$ spans an edge with every vertex in $V_j$ for all $i \ne j$. Assuming the subcomplex spanned by the vertices in $V_i$ is not itself a join for each $i$, we say that $X$ has $m$ \emph{join components}. If a join component consists of a solitary vertex we call it a \emph{singular component}, and a \emph{non-singular component} otherwise. We say that a join $X$ is \emph{$k$-sided} if $X$ has exactly $k$ non-singular components.

In the following three lemmas we show that the image $\Phi$ is characteristic in $\vC_A(\Sigma)$.  Roughly, as the naming convention suggests, we show that $k$-sided vertices of $\vC_{\partial A}(\Sigma)$ map to $k$-sided joins in $\vC_A(\Sigma)$.  While not strictly true, it is only in a distinct minority of cases where this intuition fails.  We split the vertices of $\vC_{\partial A}(\Sigma)$ into three types; \emph{strong} $2$-sided vertices, \emph{weak} $2$-sided vertices, and finally all $1$-sided vertices.

\subsection*{Strong $2$-sided vertices}

Recall from Section \ref{vertex_types_DS} that a vertex $v$ of $\vC_{\partial A}(\Sigma)$ is $2$-sided if there are vertices of $\Lk(v)$ that lie on different sides.  We call a $2$-sided vertex of $\vC_{\partial A}(\Sigma)$ \emph{strong} if there are infinitely many vertices on each of its sides, otherwise we call it \emph{weak}.

Note that all $2$-sided vertices are strong, unless one of the associated regions is homeomorphic to either $\Sigma_0^3$ or $\Sigma_{0,1}^2$, \emph{and} annular dividing sets are represented in $\partial A$. Furthermore, annular dividing sets are represented in $\partial A$ if and only if non-separating annuli are represented in $A$.

We begin by characterising the image of all strong $2$-sided vertices of $\vC_{\partial A}(\Sigma)$ under the map $\Phi$. To that end, we say that a $2$-sided join $X$ in $\vC_A(\Sigma)$ is \emph{maximal} if there is no vertex $z$ in $\vC_A(\Sigma) \sm X$ such that $X \cup \{ z \}$ spans a $2$-sided join.

\begin{lem}\label{strong}
The restriction of the map
\[
\Phi : \Big \{ \text{strong $2$-sided vertices of }\vC_{\partial A}(\Sigma) \Big \} \to \Big \{ \text{maximal $2$-sided joins of } \vC_A(\Sigma) \Big \}
\]
is a bijection.
\end{lem}

\begin{proof}
We must first show that this map makes sense, that is, for any strong $2$-sided vertex $v \in \vC_{\partial A}(\Sigma)$ the subcomplex $\Phi(v)$ is a maximal $2$-sided join in $\vC_A(\Sigma)$. Let $v$ correspond to the dividing set $\bv$ and suppose $L$ and $R$ are the two associated regions of $\bv$. We write $V_L$ for the subcomplex spanned by vertices corresponding to non-peripheral regions of $L$. We define $V_M$ to be the subcomplex spanned by peripheral regions of $L$ (and $R$) and define $V_R$ analogously to $V_L$. Now, every vertex in $\Phi(v)$ is contained in either $V_L$, $V_M$, or $V_R$. Furthermore every vertex of $V_L$ spans an edge with every vertex of $V_M$ and $V_R$. The same is true for $V_M$ and $V_R$ and so $\Phi(v) = V_L * V_M * V_R$, a join. By definition of a strong $2$-sided vertex, $L$ and $R$ are filled by regions represented in $A$. It follows that $V_L$ and $V_R$ are non-singular join components of $\Phi(v)$. Furthermore, it is clear that each vertex of $V_M$ spans an edge with every other vertex of $V_M$, and so $\Phi(v)$ is a $2$-sided join with $|V_M|$ singular components.

Suppose $\Phi(v)$ is not a maximal $2$-sided join. Then there exists a vertex $z$ not in $\Phi(v)$ such that $\Phi(v) \cup \{ z \}$ spans a $2$-sided join. Every vertex that is not in $\Phi(v)$ corresponds to a region that intersects both $L$ and $R$. It follows that the subcomplex spanned by $V_L$, $V_R$, and the vertex $z$ is not a join.  This implies that the subcomplex spanned by $\Phi(v) \cup \{z\}$ is not $2$-sided, which is a contradiction.  It follows that $\Phi(v)$ is indeed maximal.

It remains to show that all maximal $2$-sided joins of $\vC_A(\Sigma)$ are of this form. Let $X = V_1 * V_2 * \dots * V_m$ be a such a join, where $V_1$ and $V_2$ are the two non-singular components. Each $V_i$ corresponds to a subsurface $R_i$ of $\Sigma$, that is, the vertices in $V_i$ correspond to regions that fill $R_i$. Now, both $R_1$ and $R_2$ are non-separating and the complement of $\{ R_i \}_{i=1}^m$ must be a collection of annuli, as otherwise $X$ is not be maximal. Now, for $i >2$ each component $V_i$ is a single vertex. If this vertex does not correspond to an annulus then we can find a region represented in $A$ that intersects $R_i$ and either $R_1$ or $R_2$. The subcomplex spanned by $X$ and a vertex corresponding to this region is $2$-sided join and so $X$ is not maximal, a contradiction. Similarly, it must be that each annulus $R_i$, for $i >2$, has boundary components that are isotopic to boundary components of $R_1$ and $R_2$. It follows then $R_1$ has boundary components that are isotopic to a $2$-sided dividing set in $\partial A$.
\end{proof}

\subsection*{Weak $2$-sided vertices}

We now move on to the weak $2$-sided vertices of $\vC_{\partial A}(\Sigma)$. Recall that these only occur when one of the associated regions is a pair of pants or a punctured annulus, and non-separating annuli are represented in $A$.

Suppose $X$ is a $1$-sided join with more than two join components, that is, one non-singular component and at least two singular components. Let $u,w \in X$ be two such singular components. We say that $X$ is a \emph{filling} join if there are no vertices $x,y \in \vC_A(\Sigma)$ such that $\{ u, w, x, y \}$ spans a square. We call a filling join $X$ \emph{maximal} if there exist no vertex $z$ in $\vC_A(\Sigma) \sm X$ such that $X \cup \{ z \}$ spans a filling join or a $2$-sided join.

\begin{lem}\label{weak}
If the complex $\vC_A(\Sigma)$ has no corks then the restriction of the map
\[
\Phi : \Big \{ \text{weak $2$-sided vertices of }\vC_{\partial A}(\Sigma) \Big \} \to \Big \{ \text{maximal filling joins of } \vC_A(\Sigma) \Big \}
\]
is a bijection.
\end{lem}

\begin{proof}
Let $v$ be a weak $2$-sided vertex of $\vC_{\partial A}(\Sigma)$ corresponding to the dividing set $\bv$. Let $L$ be the associated region of $\bv$ region homeomorphic to either $\Sigma_0^3$ or $\Sigma_{0,1}^2$ and let $R$ be the other associated region of $\bv$. Let $V_R$ be the subcomplex spanned by vertices corresponding to regions contained in $R$. As in the proof of Lemma \ref{strong}, let $V_M$ be the subcomplex spanned by the vertices corresponding to annuli with boundary components isotopic to boundary components of $R$. Finally, define $V_L$ to be the possibly empty subcomplex consisting of the single vertex corresponding to $L$. Note that $L$ contains no other regions represented in $A$. It follows then that $\Phi(v)$ is equal to the join $V_L * V_M * V_R$ and that each vertex in $V_M$ spans an edge with all other vertices in $\Phi(v)$. Furthermore, $V_R$ contains infinitely many vertices and is not a join, so $\Phi(v)$ is a $1$-sided join in $\vC_A(\Sigma)$. Now, let $u$ and $w$ be any two distinct vertices of $V_L \cup V_M$ corresponding to regions $Q_u , Q_w \subseteq L$. If $Q_u$ is homotopic to $L$ then there is no vertex $z$ that spans an edge with both $u$ and $w$.  If $Q_u$ and $Q_w$ are annuli and a region $Q_x$ intersects $Q_u$ and not $Q_w$, then every region $Q_y$ that interscts $Q_w$ must also intersect either $Q_u$ or $Q_x$, see Figure \ref{filling_join}(i).
\begin{figure}[t]
\centering
\labellist \hair 1pt
	\pinlabel {(i)} at 300 -30
	\pinlabel {(ii)} at 1340 -40
	\pinlabel {$Q_u$} at 270 530
	\pinlabel {$Q_x$} at 535 420
	\pinlabel {$Q_w$} at 370 130
    \endlabellist
\includegraphics[scale=0.2]{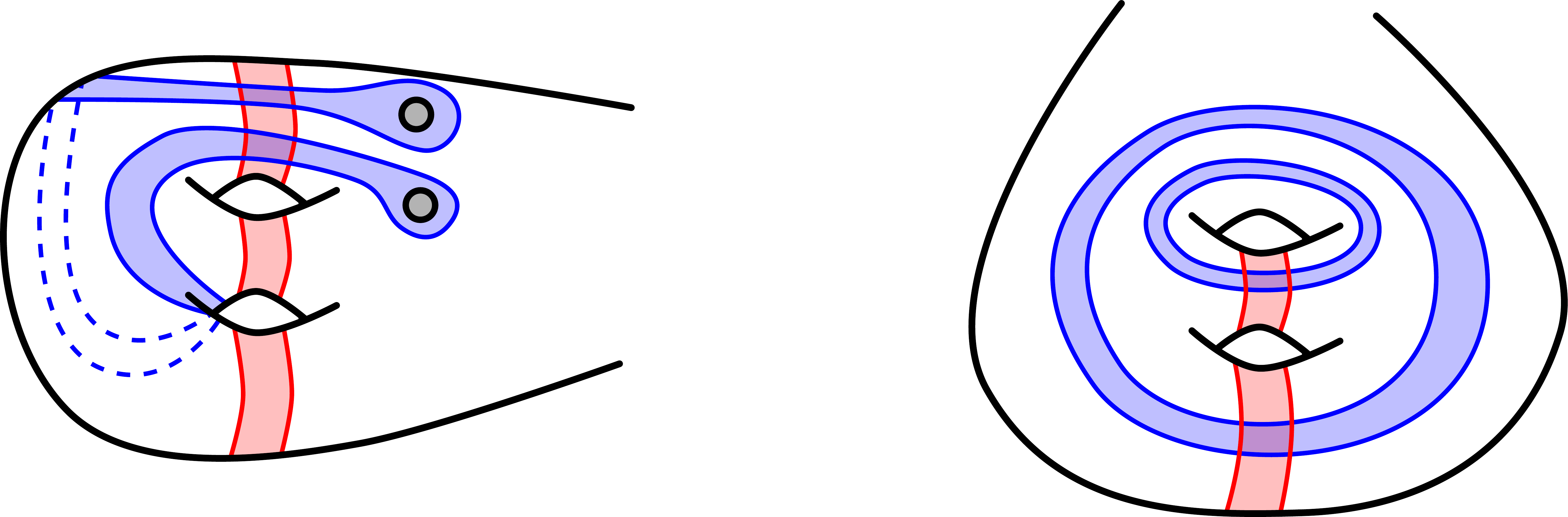}
\caption{(i) Any region that essentially intersects $Q_w$ must intersect either $Q_u$ or $Q_x$.  (ii) The vertices corresponding to the regions shown span a square in a complex of regions.}
\label{filling_join}
\end{figure}
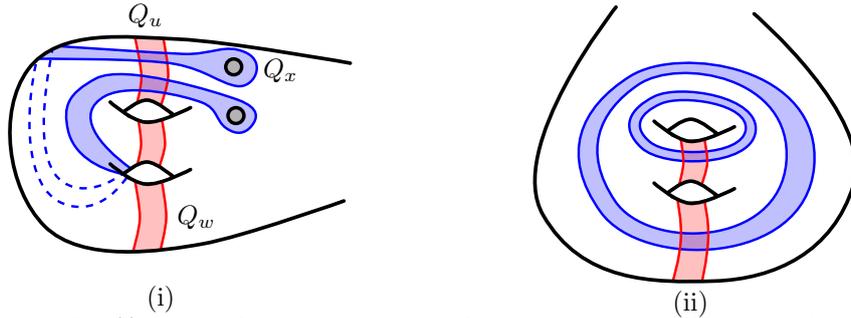
It follows that there are no vertices $x,y$ that span a square with $u,w$, hence $\Phi(v)$ is a filling join of $\vC_A(\Sigma)$.

Suppose now that $\Phi(v)$ is not maximal. Then there exists a vertex $z$ not in $\Phi(v)$ that spans a filling join or a $2$-sided join with $\Phi(v)$. However, every vertex that is not in $\Phi(v)$ also fails to span an edge with one of the vertices in $V_M$. If $L \cong \Sigma_{0,1}^2$ then it follows that if $ X := \Phi(v) \cup \{ z \}$ spans a join, it must span a $1$-sided join with a sole singular join component. In particular $X$ is neither a filling join nor a $2$-sided join, hence $\Phi(v)$ is maximal. Suppose then that $L \cong \Sigma_0^3$. As above, we see that $X$ cannot be $2$-sided and so we assume that $X$ is a filling join. By definition, $X$ must have two singular join components $u,w \in \vC_A(\Sigma)$ corresponding to non-separating annuli.  Furthermore, the complement of two such annuli in $\Sigma$ is a single connected region. We can therefore find vertices $x,y \in \vC_A (\Sigma)$ corresponding to annuli that span a square with $u$ and $w$, see Figure \ref{filling_join}(ii).  It follows that $X$ is not a filling join, hence $\Phi(v)$ is maximal.

It remains to show that all maximal filling joins in $\vC_A(\Sigma)$ are induced by weak $2$-sided vertices of $\vC_{\partial A}(\Sigma)$. Let $X = V * v_1 * \dots * v_m$ be such a join where $V$ is the sole non-singular join component, and each $v_i$ is a vertex in $\vC_A(\Sigma)$. Suppose the subcomplex $V$ is spanned by regions in the subsurface $R$. Since $X$ is $1$-sided, $R$ is connected. Furthermore, since $X$ is maximal, $R$ is non-separating. Suppose both $v_1$ and $v_2$ do not correspond to annuli. There exists a region represented in $A$ that is disjoint from $R$ and intersects both of these non-annular regions. It follows that there exists a vetex $z \in \vC_A(\Sigma) \sm X$ that spans a $2$-sided join with $X$, which is a contradiction. Thus, we have shown that at most one singular component of $X$ corresponds to a region other than an annulus. If $v_1$ corresponds to annulus that is not peripheral in $R$ then the $\Mode(\Sigma)$-orbit of this annulus fills the complementary region of $R$. As above, this contradicts the maximality of $X$.

Suppose now that there are at least four singular components of $X$ that correspond to annuli. Any region with four boundary components contains a non-peripheral annulus. Since $R$ is non-separating, it has a unique complementary region $L$ with at least four boundary components. Any such region is filled by non-separating annuli hence $X$ is not maximal. We have therefore proven that there are at most three singular vertices of $X$ that correspond to annuli.  Similar to the above argument, if $L$ has three boundary components and contains a puncture then it is filled by non-separating annuli, hence $X$ is not maximal. If $Q$ has two boundary components it must contain a puncture, otherwise $X$ would only have one singular join component (corresponding to the annulus).  If $L$ has one boundary component then $\vC_A(\Sigma)$ contains corks.  It follows that $Q$ is either a pair of pants or a punctured annulus, completing the proof.
\end{proof}

\subsection*{All $1$-sided vertices}

Finally, we deal with the $1$-sided vertices of $\vC_{\partial A}(\Sigma)$. If a $1$-sided join $X$ has two join components, that is, one singular component and one non-singular component, we call it \emph{perfect}. A perfect join $X$ is \emph{maximal} if there exist no vertices $z$ in $\vC_A(\Sigma) \sm X$ such that $X \cup \{ z \}$ spans a join.

\begin{lem}\label{allone}
If the complex $\vC_A(\Sigma)$ has no holes and no corks then the restriction of the map
\[
\Phi : \Big \{ \text{$1$-sided vertices of }\vC_{\partial A}(\Sigma) \Big \} \to \Big \{ \text{maximal perfect joins of } \vC_A(\Sigma) \Big \}
\]
is a bijection.
\end{lem}

\begin{proof}
We begin by showing that $\Phi(v)$ is a maximal perfect join if $v$ is a $1$-sided vertex corresponding to the dividing set $\bv$. Let $L$ and $R$ be the two associated regions of $\bv$ such that $L$ does not contain any non-homotopic dividing sets. By the definition of $\partial A$, this implies that if $Q$ is a region in $L$ that is represented in $A$ then $\bv \subseteq \partial Q$. Since the complex $\vC_A(\Sigma)$ does not contain any holes, it must be that $\bv = \partial Q$.  This implies that either $Q$ and $L$ are homotopic, or $Q$ is peripheral in $L$.  If $Q$ is peripheral then it may also be homotopic to $L$, if it is a non-separating annulus.  Otherwise, $L$ is not represented in $A$, as $\vC_A(\Sigma)$ has no corks.  In any of these cases we have that $Q$ is the sole region contained in $L$ that is represented in $A$ and so $\Phi(v)$ is a perfect join. To see that it is maximal we note that any vertex $z$ not in $\Phi(v)$ cannot span an edge with the singular component of $\Phi(v)$, hence $\Phi(v)$ and $z$ do not span a join.

It remains to show that every maximal perfect join $X = V * u$ is of this form. Let $V$ correspond to the region $R$. Since $X$ is a maximal perfect join, $R$ must be a connected non-separating subsurface. Let $L$ be the complementary region of $R$. Suppose $\vC_{\partial A}(\Sigma)$ contains a vertex that correspond to a dividing that is not homotopic the boundary of $L$.  In this case $L$ must contain more than one region represented in $A$. It follows that we can find a vertex of $\vC_A(\Sigma)$ that is not in $X$ yet spans a join with $X$. This contradicts the maximality of $X$.  It follows that the only vertices of $\vC_{\partial A}(\Sigma)$ that corresponds to dividing sets in $L$ are those which are homotopic to its boundary, that is, $X$ is the image of a $1$-sided dividing set.
\end{proof}

\subsection{Completing the proof}

As a consequence of Lemmas \ref{strong}, \ref{weak}, and \ref{allone} we have that an automorphism of $\vC_A(\Sigma)$ induces an automorphism on the vertices of $\vC_{\partial A}(\Sigma)$. We will now show that this automorphism extends to an automorphism of the entire complex. To that end, we say that a subcomplex $V$ of $\vC_A(\Sigma)$ is \emph{compatible} with a subcomplex $W$ if $V=V_1 * V_2$ where $V_1$ is not empty and $V_1 \subseteq W$ \cite[Section 5]{BM17}. We can now state the following results of Brendle-Margalit. These facts are vital in proving Theorem \ref{BigDaddy}.

\begin{lem}[Brendle-Margalit]\label{partial}
Let $u$ and $v$ be vertices of the connected complex of dividing sets $\vC_{\partial A}(\Sigma)$. Then $u$ and $v$ span an edge if and only if $\Phi(u)$ is compatible with $\Phi(v)$.
\end{lem}
In other words, vertices $u$ and $v$ correspond to nested dividing sets if and only if their images in $\Phi$ are compatible.

\begin{lem}[Brendle-Margalit]\label{BigDaddyHelp}
Let $\vC_A(\Sigma)$ be a connected complex of regions with no holes and no corks.
\begin{enumerate}
\item Let $R$ be represented in $A$; then there is a simplex in $\vC_{\partial A}(\Sigma)$ that corresponds to $\partial R$.
\item The complex $\vC_{\partial A}(\Sigma)$ is connected.
\end{enumerate}
\end{lem}

Before completing the proof of the main theorem of this chapter we note that there is a partial order on vertices of $\vC_A(\Sigma)$. We say that $u \preceq v$ if the link of $v$ is contained in the link of $u$. A vertex is \emph{link-minimal} when it is minimal with respect to this ordering.  If a vertex $v$ is a singular join component of a perfect join in $\vC_A(\Sigma)$ then we say that it is a \emph{$1$-sided vertex} of the complex.
Finally, if a vertex $v$ corresponds to an annulus then we call $v$ an \emph{annular vertex}.

\begin{proof}[Proof of Theorem \ref{BigDaddy}]
Let $\vC_A(\Sigma)$ be a connected complex of regions such that all minimal vertices are small.  Suppose that $\vC_A(\Sigma)$ contains no holes and no corks. We would like to show that
\[
\eta_A : \Mode(\Sigma) \rightarrow \Aut \vC_A(\Sigma)
\]
 is an isomorphism. It follows from Lemma \ref{Injectivity} that $\eta_A$ is injective. It remains to show that it is surjective. Combining all the results of Section \ref{theproof} we have that there is a well-defined homomorphism
\[
\partial :\Aut \vC_A(\Sigma) \rightarrow \Aut \vC_{\partial A}(\Sigma)
\]
where for any $v \in \vC_{\partial A}(\Sigma)$ we define $\partial (\phi) (v)$ by $\Phi^{-1} \circ \phi \circ \Phi (v)$. We will show that the homomorphism $\partial$ is injective.

Suppose $\partial(\phi)$ is the identity for some $\phi \in \Aut \vC_A(\Sigma)$. Let $v$ be a $1$-sided, annular vertex of $\vC_A(\Sigma)$ corresponding to an annulus with boundary components isotopic to the curve $\bv$.  We would like to show that $\phi(v)=v$. Let the regions $R$ and $Q$ be the associated regions of $\bv$. We want to find a vertex of $\vC_{\partial A}(\Sigma)$ corresponding to a curve which is not isotopic to the curve $\bv$. Since $\vC_A(\Sigma)$ is connected then, up to renaming regions, it contains a vertex $w$ corresponding to a subsurface of $Q$ that is not homotopic to $R$. If $w$ corresponds to an annulus then the desired vertex of $\vC_{\partial A}(\Sigma)$ corresponds to the isotopy class of the boundary components of the annulus. If $w$ does not correspond to an annulus then from Lemma \ref{BigDaddyHelp} we can find the desired vertex. We do not consider the case where $w$ corresponds to the region $Q$ itself as $\vC_A(\Sigma)$ does not contain corks.

Having found vertex in $\vC_{\partial A}(\Sigma)$ that corresponds to a non-peripheral curve in $Q$ we deduce that there exist vertices of $\vC_{\partial A}(\Sigma)$ corresponding to curves that fill $Q$. Each of these vertices is fixed by $\partial (\phi)$ by assumption and so it follows that $\phi(v)$ corresponds to a region disjoint from $Q$. Since $v$ is a $1$-sided vertex and $\vC_A(\Sigma)$ has no holes we have that $\phi(v)=v$. It can be shown using a similar argument that if $v$ is a $1$-sided, non-annular, link-minimal vertex of $\vC_A(\Sigma)$ then we can deduce that $\phi(v)=v$.

Now assume that $v$ is any other vertex of $\vC_A(\Sigma)$. Let $Q$ be a complementary region of a region $R_v$, such that $v$ corresponds to $R_v$. Since $v$ is not a $1$-sided, annular vertex and $\vC_A(\Sigma)$ does not contain holes, we have that there exist vertices that span edges with $v$ and that correspond to regions contained in $Q$. We will label the set of all such vertices $\vQ$.

Suppose vertices $u,w \in \vQ$ correspond to regions $R_u , R_w \subset Q$. Define $\wt R_u$ to be a non-separating subsurface of $Q$ containing $R_u$ such that for any other subsurface $R$ with $R_u \subset R \subset \wt R_u$ we have that $R$ and $\wt R_u$ are homotopic. Define $\wt R_w$ analogously.  Writing $u \le w$ if $\wt R_u \subseteq \wt R_w$ up to homotopy, we see $\le$ is a partial order on the vertices of $\vQ$. We claim that a $\le$-minimal vertex of $\vQ$ is either a $1$-sided, annular vertex or it is a $1$-sided, non-annular, link-minimal vertex.

We prove the claim in three steps. First we assume that $u \in \vQ$ is a $2$-sided, non-annular vertex. Let $P$ be a complementary region of $R_u$ that does not contain the boundary of $Q$. There are no holes in $\vC_A(\Sigma)$, so there must exist a vertex $w$ of $\vC_A(\Sigma)$ corresponding to a subsurface of $P$. This implies that $w < u$.

In the second case we assume that $u$ is a $1$-sided, non-annular vertex that is not link-minimal. Since $u$ is non-annular, and there are no holes in $\vC_A(\Sigma)$, the region $R_u$ must be non-separating. However $u$ is not link-minimal, so there must be a vertex $w$ such that $R_w \subset R_u$, hence $w < u$.

Finally we suppose that $u$ is a $2$-sided, annular vertex. Denote by $P$ the complementary region of $R_u$ that does not contain the boundary of $Q$. Since $\vC_A(\Sigma)$ has no corks, there must be a vertex $w$ of $\vC_A(\Sigma)$ represented by a proper subsurface of $P$. Once again, since $\vC_A(\Sigma)$ has neither holes nor corks, it must be that there exists a vertex $w < u$.

We have therefore characterised all $\le$-minimal vertices. We now have that one of the following conditions hold;
\begin{enumerate}
\item there exist $1$-sided, annular vertices and $1$-sided, non-annular, link-minimal vertices in $\vQ$ corresponding to regions that fill the region $Q$,
\item there exists a $1$-sided annular vertex of $\vC_A(\Sigma)$ corresponding to the boundary of $Q$ and no vertices of $\vC_A(\Sigma)$ correspond to non-peripheral subsurfaces of $Q$, or
\item the region $Q$ is homeomorphic to $\Sigma_0^2$, $\Sigma_{0,1}^2$, or $\Sigma_0^3$ and $\vQ$ contains $1$-sided, annular vertices corresponding to non-separating annuli.
\end{enumerate}

Indeed, if there exists a $\le$-minimal vertex of $\vQ$ corresponding to a non-peripheral region in $Q$ then by the above claim we must be in the first case. If however, all $\le$-minimal vertices of $\vQ$ are peripheral then the boundary of $Q$ may be connected or it may not. If it is connected then since there are no corks, the region $Q$ is not represented in $A$ and we are in the second case. If the boundary of $Q$ is not connected then each of its boundary components must be nonseparating curves and hence nonseparating annuli are represented in $A$. If $Q$ is homeomorphic to anything other than an annulus, a punctured annulus, or a pair of paints then we can find nonseparating annuli in that fill $Q$. This contradicts our assumption that all $\le$-minimal vertices are peripheral, so we must be in the third case.

Given a vertex $v$, let $\vV$ be the set of all the $1$-sided, annular vertices and $1$-sided, non-annular, link-minimal vertices in the link of $v$. We may now conclude that $v$ is the unique vertex of $\vC_A(\Sigma)$ such that $\Lk(v)$ contains $\vV$. Since we have shown that such vertices are fixed by $\phi$ it follows that $\phi(v)=v$.  We have succeeded in proving that $\partial$ is injective.

If every minimal vertex of $\vC_A(\Sigma)$ is small then it follows that every minimal vertex of $\vC_{\partial A}(\Sigma)$ is small.  From Lemma \ref{BigDaddyHelp} the complex $\vC_{\partial A}(\Sigma)$ is connected and by Theorem \ref{etad} the natural homomorphism $\eta_{\partial A}$ is an isomorphism. The diagram
\[
\begin{tikzcd}
\Mode(\Sigma) \arrow[hookrightarrow]{rr}{\eta_A} \arrow{dr}{\cong}[swap]{\eta_{\partial A}} & & \Aut \vC_A(\Sigma) \arrow[hookrightarrow]{dl}{\partial} \\ 
 & \Aut \vC_{\partial A}(\Sigma) & 
\end{tikzcd}
\]
is commutative from the definition of $\partial$. Finally we may conclude that both $\eta_A$ and $\partial$ are isomorphisms, completing the proof.
\end{proof}

\section{Geometric normal subgroups}\label{section_normal}
In this section we will prove Theorem \ref{AutComm} which states that many normal subgroups of $\Mode(\Sigma)$ are geometric.  We do this by defining a complex of regions related to a normal subgroup $N$ and an application of Theorem \ref{BigDaddy}.  We begin with some definitions regarding commutativity of certain mapping classes.  Let $R$ be region of $\Sigma$. Recall that a \emph{partial pseudo-Anosov} element of $\Mode(\Sigma)$ is the image of a pseudo-Anosov element of $\Mode(R)$ under the map $\Mode(R) \rightarrow \Mode(\Sigma)$ induced by the inclusion of $R$ in the surface $\Sigma$.

\subsection*{Pure mapping classes}
Using the terminology of Ivanov \cite{TeichModGroups} we call a mapping class $f$ \emph{pure} if it can be written as a product $f_1 \dots f_k$ where;
\begin{enumerate}
\item each $f_i$ is a partial pseudo-Anosov element or a power of a Dehn twist; and
\item if $i \ne j$ then the supports of $f_i$ and $f_j$ have disjoint non-homotopic representatives.
\end{enumerate}
The $f_i$ are called the \emph{components} of $f$. We note that the pure elements of $\Mode(\Sigma)$ are also elements of $\Mod(\Sigma)$.

We call a subgroup of $\Mode(\Sigma)$ \emph{pure} if each of its elements is pure. The support of a pure subgroup is well-defined and is invariant under passing to finite index subgroups.

\subsection*{Basic subgroups}
Let $N$ be a pure normal subgroup of $\Mode(\Sigma)$ and $G$ a finite index subgroup of $N$.  There exists a strict partial order on subgroups of $G$ as follows:
\[
H \prec H' \quad \mbox{if} \quad C_G(H') \subsetneq C_G(H).
\]
This means that `$\prec$' is a transitive binary relation, but no subgroup is related to itself. We say a subgroup of $G$ is a \emph{basic subgroup} if among all non-abelian subgroups of $G$ it is minimal with respect to the strict partial order.
Crucially, we will make use of the fact that if $B$ is a basic subgroup with support $R$, then the centraliser of $B$ is supported in the complement of $R$ \cite[Section 6.2]{BM17}.

Recall from Section \ref{introduction_normal} that for any $f \in \Mode(\Sigma)$ we write $R_f$ for a single-boundary subsurface such that $f$ is supported in $R_f$ and $f$ is not supported in any single-boundary proper subsurface of $R_f$.  Furthermore, $f \in N$ is of minimal support if for all elements $h \in N$ such that $R_h \subset R_f$ we have that $R_h$ and $R_f$ are homeomorphic.  We state the following result of Brendle-Margalit \cite[Lemma 6.4]{BM17}.  While stated in their paper only for closed surfaces, the proof translates to punctured surfaces.

\begin{lem}[Brendle-Margalit]\label{Action}
Let $N$ be a pure normal subgroup of $\Mode(\Sigma)$ that contains an element $f$ of minimal support and let $G$ be a finite index subgroup of $N$.
\begin{enumerate}
\item The support of a basic subgroup of $G$ is a non-annular region of $\Sigma$.
\item If $B$ is a basic subgroup of $N$ then $B \cap G$ is a basic subgroup of $G$; similarly, any basic subgroup of $G$ is a basic subgroup of $N$.
\item $N$ contains a basic subgroup whose support is a subsurface of $R_f$.
\item $\Mode(\Sigma)$ acts on the set of supports of basic subgroups of $G$.
\end{enumerate}
\end{lem}

\subsection{A complex of regions}
It follows from Lemma \ref{Action} that we may define a complex of regions $\vC^\sharp_N(\Sigma)$ whose vertices correspond to the supports of basic subgroups of $N$.  In particular, from Lemma \ref{Action}(1) we have that $\vC^\sharp_N(\Sigma)$ has no corks.  If we assume that every element of $N$ with minimal support is also of small support, then by Lemma \ref{Action}(3) the complex $\vC^\sharp_N(\Sigma)$ contains minimal vertices that are small.  Indeed, every minimal vertex of $\vC_N(\Sigma)$ must be small, as otherwise there exists an element of $N$ with minimal support that is not of small support.  This complex may be disconnected and it may contain holes.

Suppose $v \in \vC^\sharp_N(\Sigma)$ is a hole.  If $v$ corresponds to $R$ then there are components $Q_1, \dots, Q_k$ of $\Sigma \sm R$ such that no vertices of $\vC_N^\sharp(\Sigma)$ correspond to regions in any $Q_i$.  Define the \emph{filling} of $v$ to be the region $R \cup Q_1 \cup \dots \cup Q_k$. We now define $\vC_N^\flat(\Sigma)$ to be the complex of regions obtained by replacing the holes in $\vC_N^\sharp(\Sigma)$ with vertices corresponding to their fillings.

Brendle-Margalit show that the complex $\vC^\flat_N(\Sigma)$ has no holes, no corks, and that every minimal vertex is small \cite[Lemma 2.4]{BM17}. Furthermore, they show that the small vertices of $\vC^\flat_N(\Sigma)$ lie in the same connected component of the complex \cite[Lemma 6.5]{BM17}. We define this connected component to be the complex of regions $\vC_N(\Sigma)$.  It is easy to check that since $\vC^\flat_N(\Sigma)$ has no holes and no corks the complex $\vC_N(\Sigma)$ has no holes and no corks. We have therefore proven the following result.

\begin{prop}\label{PRP1}
Let $N$ be a pure normal subgroup of $\Mode(\Sigma)$ such that every element of minimal support is also of small support.  Then the natural homomorphism
\[
\Mode(\Sigma) \rightarrow \Aut \vC_N(\Sigma)
\]
is an isomorphism.
\end{prop}

\subsection{Completing the proof}
We can now move on to proving Theorem \ref{AutComm}.  The next proposition is the key step and mirrors the argument of Ivanov discussed Section \ref{introduction_complexes}

\begin{prop}[Brendle-Margalit]\label{phi_injective}
Let $N$ be a pure normal subgroup of $\Mode(\Sigma)$ such that every element of minimal support is also of small support.  There exists a natural injective homomorphism
\[ 
\Phi : \Comm N \to \Aut \vC_N(\Sigma).
\]
\end{prop}

\begin{proof}
We first define the homomorphism.  For any basic subgroup $B$ of $N$ we write $v_B$ for the vertex of $\vC_N(\Sigma)$ corresponding to the support of $B$.  Suppose $\alpha : G_1 \to G_2$ is an isomorphism between finite index subgroups of $N$.  We define $\alpha_\star := \Phi([\alpha])$ by
\[
\alpha_\star (v_B) = v_{\alpha(B \cap G_1)}.
\]
We must first show that there is a vertex of $\vC_N(\Sigma)$ that can be expressed as $v_{\alpha(B \cap G_1)}$, that is, we must show that $\alpha(B \cap G_1)$ is a basic subgroup of $N$.  This is indeed the case, as by Lemma \ref{Action}(2) the group $B \cap G_1$ is a basic subgroup of $G_1$ and since $\alpha$ is an isomorphism we have that $\alpha(B \cap G_1)$ is a basic subgroup of $G_2$.  Once again, Lemma \ref{Action}(2) tells us that $\alpha(B \cap G_1)$ is a basic subgroup of $N$, so $v_{\alpha(B \cap G_1)}$ is a vertex of $\vC_N(\Sigma)$.

We will now show that $\alpha_\star$ is a well defined automorphism of the vertices of $\vC_N(\Sigma)$.  In particular we must show that $\alpha_\star = \beta_\star$ whenever $[\alpha]=[\beta]$.  Suppose then that $H_1$ and $H_2$ are finite index subgroups of $N$ and $\beta : H_1 \to H_2$ is an isomorphism such that $[\alpha]=[\beta]$ as above.  We would like to show that
\[
v_{\alpha(B \cap G_1)} = v_{\beta(B \cap H_1)}.
\]
The isomorphisms $\alpha$ and $\beta$ agree on some finite index subgroup of $N$.  We may therefore assume that $H_1$ is a finite index subgroup of $G_1$ such that $\alpha(B \cap H_1) = \beta(B \cap H_1)$.  Since $B \cap H_1$ has finite index in $B \cap G_1$ it must be that $\beta(B \cap H_1) = \alpha(B \cap H_1)$ has finite index in $\alpha(B \cap G_1)$.  It follows that the supports of $\alpha(B \cap G_1)$ and $\beta(B \cap H_1)$ are equal, and so the vertices $v_{\alpha(B \cap G_1)}$ and $v_{\beta(B \cap H_1)}$ are equal.

It is possible for one vertex of $\vC_N(\Sigma)$ to correspond to the support of two basic subgroups.  We must also deal with this issue.  To that end, assume that $B'$ is a basic subgroup of $N$ such that $v_B = v_{B'}$.  We would now like to show that
\[
v_{\alpha(B \cap G_1)} = v_{\alpha(B' \cap G_1)}.
\]
As mentioned above, the centraliser of a basic subgroup with support $R$ is supported in the complement of $R$.  As the support of a basic subgroup is invariant under passing to finite index subgroups, we have that the centralisers of $B \cap G_1$ and $B' \cap G_1$ are equal.  It follows that the centralisers of $\alpha( B \cap G_1 )$ and $\alpha ( B' \cap G_1 )$ are also equal.  This is enough to show that $v_{\alpha(B \cap G_1)} = v_{\alpha(B' \cap G_1)}$, see \cite[Lemma 6.2]{BM17}.

In order to prove that $\Phi$ is a homomorphism it remains only to show that $\alpha_\star$ preserves the edges of $\vC_N(\Sigma)$.  We claim that two vertices $v_B$ and $v_B'$ span an edge if and only if $B$ and $B'$ commute.  Now, the subgroups $B$ and $B'$ commute if and only if $B \cap G_1$ and $B' \cap G_1$ commute.  This occurs precisely when the subgroups $\alpha ( B \cap G_1 )$ and $\alpha ( B' \cap G_1 )$ commute.  This implies that $\alpha ( B' \cap G_1)$ is a subset of the centraliser of $\alpha ( B \cap G_1)$.  This is enough to show that the vertices $v_{\alpha(B \cap G_1)}$ and $v_{\alpha(B' \cap G_1)}$ spans an edge.  This completes the proof that $\alpha_\star$ is a well defined automorphism of $\vC_N(\Sigma)$, hence $\Phi$ is a well defined homomorphism.

Finally we wish to show that $\Phi$ is injective.  Assume then that $\alpha_\star$ is the identity.  We will show that $[\alpha]$ is the identity and in particular that $\alpha$ is the identity.  For any $f \in G_1$ we would like to show that $\alpha(f) = f$.  We will write $f_\star := \eta_A(f)$ and $\alpha(f)_\star := \eta_A(\alpha(f))$, where $\eta_A$ is the isomorphism from Theorem \ref{BigDaddy}.  It is enough therefore to show that $f_\star = \alpha(f)_\star$.  As above we may assume that $B$ is contained in $G_1$.  Furthermore, it is true that $f_\star(v_B) = v_{fBf^{-1}}$, see \cite[Lemma 6.7(4)]{BM17}. We have
\[
f_\star(v_B) = v_{fBf^{-1}} = \alpha_\star (v_{f B f^{-1}}) = v_{\alpha(f B f^{-1})} = v_{\alpha(f) \alpha(B) \alpha(f)^{-1}}
\]
\[
= \alpha(f)_\star(v_{\alpha(B)}) = \alpha(f)_\star \alpha_\star (v_B) = \alpha(f)_\star (v_B).
\]
We have therefore shown that $\alpha(f)_\star = f_\star$ and so $\alpha(f) = f$.  It follows that the homomorphsim $\Phi$ is injective.
\end{proof}

We can now complete the proof of the main theorem of the paper.  Before giving the proof we note that if $G$ is a group and $H$ is a finite index subgroup then $\Comm G \cong \Comm G \cap H$.  Indeed, any isomorphism of finite index subgroups of $G$ restricts to an isomorphism of finite index subgroups of $G \cap H$ and any finite index subgroup of $G \cap H$ is also finite index in $G$.

\begin{proof}[Proof of Theorem \ref{AutComm}]
Let $N$ be a normal subgroup of $\Mode(\Sigma)$.  It can be shown that the natural homomorphisms
\[
\Mode(\Sigma) \to \Aut N \to \Comm N
\]
are injective, for example see Brendle-Margalit \cite[Proof of Theorem 1.1]{BM17}.  Let $P$ be a pure normal subgroup of finite index in $\Mode(\Sigma)$.  The fact that such a subgroup exists is shown by Ivanov \cite{TeichModGroups}.  Suppose every element of minimal support in $N$ is also of small support, then by Propositions \ref{PRP1} and \ref{phi_injective} we have the following commutative diagram:
\[
\begin{tikzcd}
\Mode(\Sigma) \arrow[hookrightarrow]{dd} \arrow{rr}{\cong} &  & \Aut \vC_{N \cap P}(\Sigma) \\ \\ 
\Aut N \arrow[hookrightarrow]{r} & \Comm N \arrow{r}{\cong} & \Comm N \cap P \arrow[hookrightarrow]{uu}{\Phi}
\end{tikzcd}
\]
This is enough to show that $\Phi$ is an isomorphism.  It follows that every homomorphism in the diagram is an isomorphism,  completing the proof.
\end{proof}

\subsection{Application to the Johnson filtration}\label{section_JF}

Finally we will apply Theorem \ref{AutComm} to the Johnson filtration which was introduced in Section \ref{introduction_normal}.  Given two curves $\ba, \bb$ with algebraic intersection two, there exists a product of Dehn twists $T_\ba$, $T_\bb$ that belongs to $\vJ_k(\Sigma)$ for any $k$, as argued by Farb \cite[Theorem 5.10]{FarbProblems}.  It follows that for any $k$, if $f 
\in \vJ_k(\Sigma)$ is an element of minimal support then $R_f$ is homeomorphic to $\Sigma_{2,0}^1$, $\Sigma_{1,1}^1$, or $\Sigma_{0,3}^1$.

Suppose that $\Sigma = \Sigma_{g,n}$ where $g \ge 5$ and $n > 0$ and let $f$ be an element of minimal support in $\vJ_k(\Sigma)$.  We first consider the case where $R_f \cong \Sigma_{2,0}^1$.  Now, $f$ is of small support if there exist elements $h_1,h_2 \in \vJ_k(\Sigma)$ satisfying the inequalities $(1)$ and $(2)$.  If $n \ge 3$ we may choose $h_1,h_2$ such that $R_{h_1} \cong R_{h_2} \cong \Sigma_{1,1}^1$.  Hence,
\begin{align*}
g = 5 = 2 + 1 + 1 + 1 &= g(R_f) + g(R_{h_1}) + g(R_{h_1}) + 1, \mbox{and} \\ 
n \ge 3 = 0 + 1 + 1 + 1 &= n(R_f) + n(R_{h_1}) + n(R_{h_1}) + 1.
\end{align*}
Next, we consider the case where $R_f \cong \Sigma_{1,1}^1$.  Here, if $n \ge 3$ we may choose $h_1, h_2 \in \vJ_k(\Sigma)$ such that $R_{h_1} \cong \Sigma_{2,0}^1$ and $R_{h_2} \cong \Sigma_{1,1}^1$.  It follows that
\begin{align*}
g = 5 = 1 + 2 + 1 + 1 &= g(R_f) + g(R_{h_1}) + g(R_{h_2}) + 1, \mbox{and} \\ 
n \ge 3 = 1 + 0 + 1 + 1 &= n(R_f) + n(R_{h_1}) + n(R_{h_2}) + 1.
\end{align*}
Finally, if $f$ is such that $R_f \cong \Sigma_{0,3}^1$ and $n \ge 5$ we choose $h_1,h_2$ such that $R_{h_1} \cong \Sigma_{2,0}^1$ and $R_{h_2} \cong \Sigma_{1,1}^1$ and so
\begin{align*}
g = 5 \ge 0 + 2 + 1 + 1 &= g(R_f) + g(R_{h_1}) + g(R_{h_2}) + 1, \mbox{and} \\ 
n \ge 5 = 3 + 0 + 1 + 1 &= n(R_f) + n(R_{h_1}) + n(R_{h_2}) + 1.
\end{align*}
Applying Theorem \ref{AutComm} we have that $\vJ_k(\Sigma)$ is geometric, that is, we arrive at Corollary \ref{JF}.


\bibliography{thesisbib}{}

\begin{thebibliography}{10}

\bibitem{An}
B.~H. {An}.
\newblock {Automorphisms of braid groups on orientable surfaces}.
\newblock {\em ArXiv e-prints}, July 2015.

\bibitem{BL94}
Hyman Bass and Alexander Lubotzky.
\newblock Linear-central filtrations on groups.
\newblock {\em Contemp. Math.}, 169, 1994.

\bibitem{BDR}
Juliette Bavard, Spencer Dowdall, and Kasra Rafi.
\newblock Isomorphisms between big mapping class groups.
\newblock {\em pre-print}, 2017.

\bibitem{Strongly}
Brian Bowditch.
\newblock Rigidity of the strongly separating curve graph.
\newblock {\em Michigan Mathematical Journal}, (65):813-832, 2016.

\bibitem{Sep}
Tara~E. Brendle and Dan Margalit.
\newblock Commensurations of the {Johnson} kernel.
\newblock {\em Geometry and Topology}, (8):1361-1384, 2004.

\bibitem{BM17}
Tara~E. Brendle and Dan Margalit.
\newblock Normal subgroups of mapping class groups and the metaconjecture of
  {Ivanov}.
\newblock {\em \href{https://arxiv.org/abs/1710.08929}{arxiv:1710.08929}},
  2017.

\bibitem{BMP}
Tara~E. Brendle, Dan Margalit, and Andrew Putman.
\newblock Generators for the hyperelliptic {Torelli} group and the kernel of
  the {Burau} representation at t = -1.
\newblock {\em Inventiones Mathematicae}, 200(1):263-310, 2015.

\bibitem{BPS}
Martin Bridson, Alexandra Pettet, and Juan Souto.
\newblock The abstract commensurator of the {Johnson} kernels.
\newblock {\em in preperation}.

\bibitem{CMM}
Matt Clay, Johanna Mangahas, and Dan Margalit.
\newblock Normal right-angled {Artin} subgroups of mapping class groups.
\newblock {\em In preperation}.

\bibitem{FarbProblems}
Benson Farb.
\newblock Fifteen problems about the mapping class groups.
\newblock {\em Proceedings of Symposia in Pure Mathematics}, 74 - Problems on
  Mapping Class Groups and Related Topics, 2006.

\bibitem{FI05}
Benson Farb and Nikolai~V. Ivanov.
\newblock The {Torelli} geometry and its applications. research announcement.
\newblock {\em Mathematical Research Letters}, 293-301, 2005.

\bibitem{Ham14}
Ursula Hamenst{\"a}dt.
\newblock Distance in the curve graph.
\newblock {\em Geometriae Dedicata}, 168(1):101--112, 2014.

\bibitem{IIM}
E.~{Irmak}, N.~V. {Ivanov}, and J.~D. {McCarthy}.
\newblock {Automorphisms of surface braid groups}.
\newblock {\em ArXiv Mathematics e-prints}, June 2003.

\bibitem{Nonsep}
Elmas Irmak.
\newblock Complexes of nonseparating curves and mapping class groups.
\newblock {\em Michigan Mathematical Journal}, (54):81-110, 2006.

\bibitem{Arc}
Elmas Irmak and John~D. McCarthy.
\newblock Injective simplicial maps of the arc complex.
\newblock {\em Turkish Journal of Mathematics}, (34):2055-2077, 2010.

\bibitem{IV}
Nikolai~V. Ivanov.
\newblock Automorphisms of complexes of curves and of {Teichm\"uller} spaces.
\newblock {\em International Mathematics Research Notices}, (14):651-666, 1997.

\bibitem{TeichModGroups}
Nocolai~V. Ivanov.
\newblock {\em Subgroups of {Teichm\"{u}ller} Modular Groups}.
\newblock American Mathematical Society, 1992.

\bibitem{DJ1}
Denis Johnson.
\newblock The structure of the {Torelli} group. {I}. a finite set of generators
  for $\mathcal{I}$.
\newblock {\em Annals of Mathematical Studies}, 118(3):423-442, 1983.

\bibitem{DJ2}
Denis Johnson.
\newblock The structure of the {Torelli} group. {II}. a characterization of the
  group generated by twists on bounding curves.
\newblock {\em Topology}, 24(2):113-126, 1985.

\bibitem{YK01}
Yasushi Kasahara.
\newblock An expansion of the {Jones} representation of genus 2 and the
  {Torelli} group.
\newblock {\em Algebr. Geom. Topol.}, 1:39-55, 2001.

\bibitem{Kida}
Yoshikata Kida.
\newblock Automorphisms of the {Torelli} complex and the complex of separating
  curves.
\newblock {\em Journal of the Mathematical Society of Japan}, 63, 2011.

\bibitem{KOR}
Mustafa Korkmaz.
\newblock Automorphisms of complexes of curves on punctured spheres and on
  punctured tori.
\newblock {\em Topology and its applications}, 95, 1999.

\bibitem{Arccurve}
Mustafa Korkmaz and Athanase Papadopoulos.
\newblock On the arc and curve complex.
\newblock {\em Mathematical Proceedings of the Cambridge Philosophical
  Society}, (148):473-483, 2010.

\bibitem{Triangulation}
Mustafa Korkmaz and Athanase Papadopoulos.
\newblock On the ideal triangulation graph of a punctured surface.
\newblock {\em Annalesde l'Institut Fourier}, (62):1367-1384, 2012.

\bibitem{Pants}
Dan Margalit.
\newblock Automorphisms of the pants complex.
\newblock {\em Geometry and Topology}, (8):1361-1384, 2004.

\bibitem{MMHyper}
H.~A. {Masur} and Y.~N. {Minsky}.
\newblock {Geometry of the complex of curves I: Hyperbolicity}.
\newblock {\em Inventiones Mathematicae}, 138:103--149, October 1999.

\bibitem{MCP}
John~D. McCarthy and Athanase Papadopolous.
\newblock Simplicial actions of mapping class groups.
\newblock {\em Handbook of Teichm\"uller theory}, 3, 2013.

\bibitem{BraidMeta}
Alan McLeay.
\newblock Normal subgroups of the braid group and the metaconjecture of
  {Ivanov}.
\newblock {\em \href{https://arxiv.org/abs/1801.05209}{arxiv:1801.05209}},
  2017.

\bibitem{GM92}
Geoffrey Mess.
\newblock The {Torelli} groups for genus 2 and 3 surfaces.
\newblock {\em Topology}, 31(4):775-790, 1992.

\bibitem{AP07}
Andrew Putman.
\newblock Cutting and pasting the {Torelli} group.
\newblock {\em Geom. Topol.}, 11:829-865, 2007.

\bibitem{PutmanConnect}
Andrew Putman.
\newblock A note on the connectivity of certain complexes associated to
  surfaces.
\newblock {\em L{'}Enseignement Math\'{e}matique}, (2) 54: 287-301, 2008.

\bibitem{AP12}
Andrew Putman.
\newblock Small generating sets for the {Torelli} group.
\newblock {\em Geom. Topol.}, 16(1):111-125, 2012.

\bibitem{RS09}
Kasra Rafi and Saul Schleimer.
\newblock Covers and the curve complex.
\newblock {\em Geom. Topol.}, 13(4):2141--2162, 2009.

\end{thebibliography}
\bibliographystyle{plain}

\end{document}